%% file: GRDE_clean.tex
\documentclass[11pt]{amsart}%
\usepackage[pdftex]{graphicx,color}
\usepackage{comment}
\usepackage{amsmath,amsthm, amssymb}
\usepackage{amsfonts}
\usepackage{hyperref}
\usepackage{cite}
\usepackage[toc,page]{appendix}
\usepackage{amsmath}
\usepackage{amssymb}
\usepackage{graphicx}%
\providecommand{\U}[1]{\protect\rule{.1in}{.1in}}
\newtheorem{theorem}{Theorem}[section]

\newtheorem{corollary}[theorem]{Corollary}

\newtheorem{definition}[theorem]{Definition}
\newtheorem{example}[theorem]{Example}

\newtheorem{lemma}[theorem]{Lemma}
\newtheorem{notation}[theorem]{Notation}

\newtheorem{proposition}[theorem]{Proposition}
\newtheorem{remark}[theorem]{Remark}

\newtheorem{ass}{Assumption}

\input{defs}
\numberwithin{equation}{section}

\newcommand{\GraphicsDirectory}{./}
\input{svg-graphics}
\excludecomment{Aside}
\begin{document}
\title{Global Existence of Geometric Rough Flows}
\date{\today}
\author{Bruce K. Driver}

\begin{abstract}
In this paper we consider rough differential equations on a smooth manifold
$\left(  M\right)  .$ The main result of this paper gives sufficient
conditions on the driving vector-fields so that the rough ODE's have global
(in time) solutions. The sufficient conditions involve the existence of a
complete Riemannian metric $\left(  g\right)  $ on $M$ such that the covariant
derivatives of the driving fields and their commutators to a certain order
(depending on the roughness of the driving path) are bounded. Many of the
results of this paper are generalizations to manifolds of the fundamental
results in \cite{Bailleul2015a}.

\end{abstract}
\subjclass[2010]{34F05 (primary), 60H10, 34C40 }
\keywords{Rough paths, Rough flows, Riemannian manifolds}
\maketitle
\tableofcontents

\section{Introduction\label{sec.10}}

\subsection{Overview\label{sec.10.1}}

If $\left[  0,T\right]  \times\mathbb{R}^{N}\ni\left(  t,y\right)  \rightarrow
Y_{t}\left(  y\right)  \in\mathbb{R}^{N}$ is a smooth time dependent vector
field on $\mathbb{R}^{N},$ it is natural to consider solving for $y:\left[
0,T\right]  \rightarrow\mathbb{R}^{N}$ and $\left(  t_{0},y_{0}\right)
\in\left[  0,T\right]  \times\mathbb{R}^{N}$ the ordinary differential
equation,%
\begin{equation}
\dot{y}\left(  t\right)  =\dot{Y}_{t}\left(  y\left(  t\right)  \right)
\text{ with }y\left(  t_{0}\right)  =y_{0}, \label{e.10.1}%
\end{equation}
where the dot indicates derivatives in $t.$ The following theorem is then well
known and easy to prove.

\begin{theorem}
\label{thm.10.1}If there exists $a,b\in\left(  0,\infty\right)  $ such that
\[
\left\vert \dot{Y}_{t}\left(  y\right)  \right\vert \leq a+c\left\vert
y\right\vert \text{ }\forall~\left(  t,y\right)  \in\left[  0,T\right]
\times\mathbb{R}^{N},
\]
then Eq. (\ref{e.10.1}) has a unique solution.
\end{theorem}

In the setting of stochastic differential equations one wishes to solve
(\ref{e.10.1}) in the case where $Y_{t}$ is a random vector field which
typically is rough in $t$ and in particular no longer differentiable in $t.$
Such equations have been under active investigation ever since It\^{o}'s
pioneering work \cite{Ito1944,Ito1946} and use the random structure of $Y$ in
order to give meaning to Eq. (\ref{e.10.1}). More recently in his pioneering
work on \textbf{rough paths}, Terry Lyons
\cite{Lyons1994,Lyons1995,Lyons1998a} was able to show that one may use the
statistical properties of $Y$ in order to \textquotedblleft
enhance\textquotedblright\ $Y$ with added information (see Example
\ref{ex.10.22} below) where the degree of enhancement needed increases with
the lack of differentiability of $t\rightarrow Y_{t}.$ Once the enhanced $Y$
is found, Lyons was able to dispense with the randomness and again make
deterministic sense of Eq. (\ref{e.10.1}). In particular he was able to prove
the analogue of Theorem \ref{thm.10.1} when $c=0.$ The main goal of this paper
is to give a version of Lyons' rough differential equation (RDE for short)
existence theorem when $\mathbb{R}^{N}$ is replaced by a smooth manifold $M$
and the size of $\dot{Y}_{t}$ is now measured by a Riemannian metric $\left(
g\right)  $ on $M.$ Along the way we will also see that boundedness condition
on $Y$ (i.e. $c=0)$ in Lyons existence theorem may be considerably weakened.
For some history and other results on relaxing the $c=0$ condition, see
\cite{Lejay2012}.

This paper was inspired and highly influenced by Bailleul's paper
\cite{Bailleul2015a}. In fact, this paper grew out of an unsuccessful attempt
to understand the unbounded vector-field existence result stated in
\cite[Theorem 4.1.]{Bailleul2015a}. After finishing the first draft of this
paper the author discovered \cite{Bailleul2018} which contains similar results
to those in this paper when $M$ is a Euclidean or more generally a Banach
space. More recently I have been alerted to the work of Brault and Lejay
\cite{brault:hal-01716945,brault:hal-01839202} and Martin Weidner
\cite{Weidner2018}. In principle Brault and Lejay develop a more general
version of the results in Section \ref{sec.11} of this paper. However, there
seems to be a mistaken assertion in Eq. (21) on page 11 of
\cite{brault:hal-01716945} which will likely require a rewriting of their
results. [I do suspect most of the results in \cite{brault:hal-01716945} are
morally true.] The existence results in \cite{Weidner2018} are quite analogous
to those given in this paper including the idea of using a conformal change of
the metric to extend the theory away from bounded vector fields, compare
\cite[Proposition 5.2]{Weidner2018} with Proposition \ref{pro.15.10}. However,
some of the details of the proofs are different. In particular, Weidner makes
use of a number of localization arguments in order exploit known RDE results
on $\mathbb{R}^{N}$ while in this paper all of the proofs are intrinsic.
Moreover, the proofs given here follow Bailleul's method in
\cite{Bailleul2015a} which does not require knowing previous RDE results on
$\mathbb{R}^{N}.$ Other than using some ordinary differential equation
estimates on manifolds developed in \cite{Driver2018} (some of these same
estimates are also in \cite{Weidner2018}), this paper is essentially self-contained.

Before describing the results in this paper precisely, let us first rephrase
(see Theorem \ref{thm.10.6}) what it means to solve Eq. (\ref{e.10.1}) in such
a way that easily generalizes to the rough path setting. We begin by fixing
some notation.

\begin{notation}
\label{not.10.2}Throughout this paper, we let $M$ be a smooth finite
dimensional manifold and $\Gamma\left(  TM\right)  $ be the linear space of
smooth vector fields on $M.$ Typically $g$ will denote a Riemannian metric on
$M,$ $\nabla=\nabla^{g}$ be the associated Levi-Civita covariant derivative,
$R=R^{g}$ be the curvature tensor of $\nabla,$ and $d\left(  \cdot
,\cdot\right)  =d_{g}\left(  \cdot,\cdot\right)  $ be the length metric
associated to $g.$
\end{notation}

We will also abuse notation and use $d$ to denote a fixed integer in
$\mathbb{N}.$ This should not cause confusion as distance function will always
come with arguments.

\begin{definition}
\label{def.10.3}A $d$\textbf{-dimensional dynamical system} on $M$ is a linear
map, $\mathbb{R}^{d}\ni w\rightarrow V_{w}\in\Gamma\left(  TM\right)  .$
\end{definition}

A $d$-dimensional dynamical system on $M$ is completely determined by knowing
$\left\{  V_{e_{j}}\right\}  _{j=1}^{d}\subset\Gamma\left(  TM\right)  $ where
$\left\{  e_{j}\right\}  _{j=1}^{d}$ is the standard basis for $\mathbb{R}%
^{d}.$

The object of this paper is to describe necessary conditions for a
\textquotedblleft rough ordinary differential equations\textquotedblright%
\ associated to $V$ (RODE for short) to have global (in time) solutions. These
conditions will be in the form of the existence of a Riemannian metric, $g,$
on $M$ which is complete and \textquotedblleft controls\textquotedblright\ the
size of $V$ \textquotedblleft appropriately.\textquotedblright\ The following
two theorems serves as warm-up to the general results to be stated in
subsection \ref{sec.10.3} below.

\begin{notation}
\label{not.10.4}Given a dynamical system, $V,$ and Riemannian metric, $g,$ on
$M,$ let
\begin{equation}
\left\vert V\right\vert _{M}:=\sup_{w\in\mathbb{R}^{d}:\left\vert w\right\vert
=1}\sup_{m\in M}\left\vert V_{w}\left(  m\right)  \right\vert _{g}
\label{e.10.2}%
\end{equation}
and
\begin{equation}
\left\vert \nabla V\right\vert _{M}:=\sup_{w\in\mathbb{R}^{d}:\left\vert
w\right\vert =1}\sup_{v\in TM:\left\vert v\right\vert =1}\left\vert \nabla
_{v}V_{w}\right\vert _{g}. \label{e.10.3}%
\end{equation}

\end{notation}

\begin{theorem}
\label{thm.10.5}Suppose that $V$ is a dynamical system and $g$ is a complete
Riemannian metric on $M$ such that $\left\vert \nabla V\right\vert _{M}%
<\infty,$ then for every $x\in C^{1}\left(  \left[  0,T\right]  ,\mathbb{R}%
^{d}\right)  ,$ $s\in\left[  0,T\right]  ,$ and $m\in M,$ there exists
$\sigma\in C^{1}\left(  \left[  0,T\right]  ,M\right)  $ such that
\[
\dot{\sigma}\left(  t\right)  =V_{\dot{x}\left(  t\right)  }\left(
\sigma\left(  t\right)  \right)  \text{ with }\sigma\left(  s\right)  =m.
\]

\end{theorem}

\begin{proof}
For a proof of this classical theorem the reader may refer to \cite[Corollary
2.12]{Driver2018} with $Y_{t}:=V_{\dot{x}\left(  t\right)  }.$ A more general
form of this theorem may also be found in Lemma \ref{lem.15.3} of the Appendix
\ref{sec.15}.
\end{proof}

If $V$ is a dynamical system satisfying the conclusions of Theorem
\ref{thm.10.5}, let $\left[  0,T\right]  ^{2}\times M\ni\left(  t,s,m\right)
\rightarrow\varphi_{t,s}^{x}\left(  m\right)  \in M$ be the flow map defined
by requiring,%
\begin{equation}
\frac{d}{dt}\varphi_{t,s}^{x}\left(  m\right)  =V_{\dot{x}\left(  t\right)
}\left(  \varphi_{t,s}^{x}\left(  m\right)  \right)  \text{ with }%
\varphi_{s,s}\left(  m\right)  =m. \label{e.10.4}%
\end{equation}
We will usually abbreviate the previous equation by writing%
\begin{equation}
\frac{d}{dt}\varphi_{t,s}^{x}=V_{\dot{x}\left(  t\right)  }\circ\varphi
_{t,s}^{x}\text{ with }\varphi_{s,s}^{x}=Id_{M}. \label{e.10.5}%
\end{equation}

In the next theorem, we will give an alternate characterization of the flow
$\varphi_{t,s}^{x}$ which is suitable for defining $\varphi^{x}$ where $x\in
C^{1}\left(  \left[  0,T\right]  ,\mathbb{R}^{d}\right)  $ is replaced by much
rougher paths. In the hypothesis of Theorem \ref{thm.10.6} we will not only
require that Eq. (\ref{e.10.3}) holds but that also if $\left\vert
V\right\vert _{M}<\infty.$Later in the introduction we will (based on the
results of Appendix \ref{sec.15}) see that by replacing $g$ by an appropriate
conformally equivalent metric we may remove this added restriction on $V.$

\begin{theorem}
\label{thm.10.6}Suppose that $V$ is a dynamical system and $g$ is a complete
Riemannian metric on $M$ such that both Eqs. (\ref{e.10.2}) and (\ref{e.10.3})
are satisfied. Then for every $x\in C^{1}\left(  \left[  0,T\right]
,\mathbb{R}^{d}\right)  ,$ $\varphi_{t,s}^{x}\in\mathrm{Diff}\left(  M\right)
,$ $\varphi^{x}$ is \textbf{multiplicative }(i.e. $\varphi_{t,s}^{x}%
\circ\varphi_{s,u}^{x}=\varphi_{t,u}^{x}$ for all $s,t,u\in\left[  0,T\right]
),$ and there exists $\mathcal{K<}\infty$ such that
\begin{equation}
d\left(  \varphi_{t,s}^{x}\left(  m\right)  ,e^{V_{x_{s,t}}}\left(  m\right)
\right)  \leq\mathcal{K}\left\vert t-s\right\vert ^{2}\text{ }\forall
~s,t\in\left[  0,T\right]  \text{ and }m\in M, \label{e.10.6}%
\end{equation}
where
\[
x_{s,t}:=x\left(  t\right)  -x\left(  s\right)  \text{ }\forall~s,t\in\left[
0,T\right]  .
\]
Conversely, if $\left\{  \varphi_{t,s}\right\}  _{t,s\in\left[  0,T\right]
}\subset\mathrm{Diff}\left(  M\right)  $ is a multiplicative and there exists
$\mathcal{K<}\infty$ such that Eq. (\ref{e.10.6}) holds, then $\varphi
_{t,s}=\varphi_{t,s}^{x}.$
\end{theorem}

\begin{proof}
Suppose that $\varphi_{t,s}^{x}$ is defined as in Eq. (\ref{e.10.5}). The
estimate in Eq. (\ref{e.10.6}) is now a direct consequence of \cite[Theorem
4.11]{Driver2018} applied with $\kappa=1.$The fact that $\varphi_{t,s}^{x}\in\mathrm{Diff}\left(  M\right)  $ and
$\varphi^{x}$ is multiplicative is a standard property of smooth flows, see
for example, \cite[Theorem 2.14]{Driver2018} for a more detailed summary of
such flows.

Conversely, if $\varphi_{t,s}\in\mathrm{Diff}\left(  M\right)  $ is
multiplicative and satisfies the estimate in Eq. (\ref{e.10.6}) and $f\in
C^{\infty}\left(  M\right)  ,$ then%
\[
\left\vert f\left(  \varphi_{t,s}\left(  m\right)  \right)  -\left[  f\left(
m\right)  +V_{x_{s,t}}f\left(  m\right)  \right]  \right\vert \leq C\left\vert
t-s\right\vert ^{2}.
\]
Dividing this estimate by $\left\vert t-s\right\vert $ implies,%
\[
\left\vert \frac{f\left(  \varphi_{t,s}\left(  m\right)  \right)  -f\left(
m\right)  }{\left(  t-s\right)  }-V_{\frac{x_{s,t}}{t-s}}f\left(  m\right)
\right\vert \leq C\left\vert t-s\right\vert
\]
and upon letting $t\rightarrow s,$ gives%
\[
\frac{d}{dt}\left\vert _{t=s}\right\vert f\left(  \varphi_{t,s}\left(
m\right)  \right)  =\left(  V_{\dot{x}\left(  s\right)  }f\right)  \left(
m\right)  .
\]
Since the previous equality is valid for all $f\in C^{\infty}\left(  M\right)
,$ we conclude that $\frac{d}{dt}_{t=s}\varphi_{t,s}\left(  m\right)  $ exists
and
\[
\frac{d}{dt}_{t=s}\varphi_{t,s}\left(  m\right)  =V_{\dot{x}\left(  s\right)
}\left(  m\right)  .
\]
Then using the multiplicative property of $\varphi_{t,s}$ it follows that%
\[
\frac{d}{dt}\varphi_{t,s}\left(  m\right)  =\frac{d}{d\varepsilon}|_{0}%
\varphi_{t+\varepsilon,t}\circ\varphi_{t,s}\left(  m\right)  =V_{\dot
{x}\left(  t\right)  }\left(  \varphi_{t,s}\left(  m\right)  \right)  ,
\]
i.e. $\varphi_{t,s}$ satisfies flow ODE in Eq. (\ref{e.10.4}).
\end{proof}

A basic idea of rough paths and Ballieul \cite{Bailleul2015a} is to generalize
Eq. (\ref{e.10.6}) so as to allow for much rougher paths, $x,$ see Theorem
\ref{thm.10.27}. The rest of this introduction is devoted to stating the main
results of this paper. As usual in rough path theory, the rougher $x$ becomes
the more extra information one must enhance $x$ with in order to give meaning
to Eq. (\ref{e.10.4}). The next subsection introduces the basic rough path
language we will need in this paper. The summary of the main theorems of the
paper will follow in subsection \ref{sec.10.3}.

\subsection{Basic Notations\label{sec.10.2}}

Although our presentation here will be self-contained, the reader's wishing
for more background on rough paths may consult the monographs
\cite{Lyons2002,LCT2007,Friz2010,Friz2014} and the survey article of
\cite{Lejay2003}. The reader may also wish to consult
\cite{Cass2016,Cass2015,Cass2012a,Driver2017b} for the beginnings of general
theory of rough paths on manifolds. For a more detailed description of the
algebra presented here, see \cite[Subsection 1.2]{Driver2018}.

\begin{definition}
[Tensor Algebras]\label{def.10.7}Let $T\left(  \mathbb{R}^{d}\right)
:=\oplus_{k=0}^{\infty}\left[  \mathbb{R}^{d}\right]  ^{\otimes k}$ be the
tensor algebra over $\mathbb{R}^{d}$ so the general element of $\omega\in
T\left(  \mathbb{R}^{d}\right)  $ is of the form
\[
\omega=\sum_{k=0}^{\infty}\omega_{k}\text{ with }\omega_{k}\in\left(
\mathbb{R}^{d}\right)  ^{\otimes k}\text{ for }k\in\mathbb{N}_{0}%
\]
where we assume $\omega_{k}=0$ for all but finitely many $k.$ Multiplication
is the tensor product and associated to this multiplication is the Lie
bracket,%
\begin{equation}
\left[  A,B\right]  _{\otimes}:=A\otimes B-B\otimes A\text{ for all }A,B\in
T\left(  \mathbb{R}^{d}\right)  . \label{e.10.7}%
\end{equation}

\end{definition}

A good reference for nilpotent Lie algebras and related material is
\cite{ReutenauerBook} although everything we need will be described here and
explained in more detail when needed in Section \ref{sec.12} below.

\begin{definition}
[Free Lie Algebra]\label{not.10.8}The \textbf{free Lie algebra over
}$\mathbb{R}^{d}$ will be taken to be the Lie-subalgebra, $F\left(
\mathbb{R}^{d}\right)  ,$ of $\left(  T\left(  \mathbb{R}^{d}\right)  ,\left[
\cdot,\cdot\right]  _{\otimes}\right)  $ generated by $\mathbb{R}^{d}.$
\end{definition}

\begin{remark}
\label{rem.10.9}If $\left(  \mathfrak{g},\left[  \cdot,\cdot\right]  \right)
$ is a Lie algebra and $V\subset\mathfrak{g}$ is a subspace, then using
Jacobi's identity one easily shows that Lie sub-algebra $\left(
\operatorname*{Lie}\left(  V\right)  \right)  $ of $\mathfrak{g}$ generated by
$V$ may be described as;%
\[
\operatorname*{Lie}\left(  V\right)  =\operatorname*{span}\cup_{k=1}^{\infty
}\left\{  \operatorname{ad}_{v_{1}}\dots\operatorname{ad}_{v_{k-1}}v_{k}%
:v_{1},\dots,v_{k}\in V\right\}  ,
\]
where $\operatorname{ad}_{A}B:=\left[  A,B\right]  $ for all $A,B\in
\mathfrak{g}.$ As a consequence of this remark it follows that $F\left(
\mathbb{R}^{d}\right)  $ is a $\mathbb{N}_{0}$-graded Lie algebra with
\[
F\left(  \mathbb{R}^{d}\right)  =\oplus_{k=0}^{\infty}F_{k}\left(
\mathbb{R}^{d}\right)  \text{ where }F_{k}\left(  \mathbb{R}^{d}\right)
=F\left(  \mathbb{R}^{d}\right)  \cap\left[  \mathbb{R}^{d}\right]  ^{\otimes
k}\subset F\left(  \mathbb{R}^{d}\right)  .
\]
According to this grading, if $A\in F\left(  \mathbb{R}^{d}\right)  $ we let
$A_{k}\in F_{k}\left(  \mathbb{R}^{d}\right)  $ denote the projection of $A$
into $F_{k}\left(  \mathbb{R}^{d}\right)  .$
\end{remark}

The spaces $T\left(  \mathbb{R}^{d}\right)  $ and $F\left(  \mathbb{R}%
^{d}\right)  $ are infinite dimensional. We are going to be most interested in
the finite dimensional truncated versions of these algebras.

\begin{definition}
[Truncated Tensor Algebras]\label{def.10.10}Given $\kappa\in\mathbb{N},$ let
\[
T^{\left(  \kappa\right)  }\left(  \mathbb{R}^{d}\right)  :=\oplus
_{k=0}^{\kappa}\left[  \mathbb{R}^{d}\right]  ^{\otimes k}\subset T\left(
\mathbb{R}^{d}\right)
\]
which is algebra under the multiplication rule,%
\[
AB=\sum_{k=0}^{\kappa}\left(  AB\right)  _{k}=\sum_{k=0}^{\kappa}\sum
_{j=0}^{k}A_{j}\otimes B_{k-j}~\text{ }\forall~A,B\in T^{\left(
\kappa\right)  }\left(  \mathbb{R}^{d}\right)
\]
and a Lie algebra under the bracket operation, $\left[  A,B\right]  :=AB-BA$
for all $A,B\in T^{\left(  \kappa\right)  }\left(  \mathbb{R}^{d}\right)  .$
\end{definition}

\begin{notation}
\label{not.10.11}Let $\pi_{\leq\kappa}:T\left(  \mathbb{R}^{d}\right)
\rightarrow T^{\left(  \kappa\right)  }\left(  \mathbb{R}^{d}\right)  $ and
$\pi_{>\kappa}:=I_{T\left(  \mathbb{R}^{d}\right)  }-\pi_{\leq\kappa}:T\left(
\mathbb{R}^{d}\right)  \rightarrow\oplus_{k=\kappa+1}^{\infty}\left[
\mathbb{R}^{d}\right]  ^{\otimes k}$ be the projections associated to the
direct sum decomposition,
\[
T\left(  \mathbb{R}^{d}\right)  =T^{\left(  \kappa\right)  }\left(
\mathbb{R}^{d}\right)  \oplus\left(  \oplus_{k=\kappa+1}^{\infty}\left[
\mathbb{R}^{d}\right]  ^{\otimes k}\right)  .
\]
Further let
\begin{equation}
\mathfrak{g}^{\left(  \kappa\right)  }=\oplus_{k=1}^{\kappa}\left[
\mathbb{R}^{d}\right]  ^{\otimes k} \label{e.10.9}%
\end{equation}
which is a two sided ideal as well as a Lie sub-algebra of $T^{\left(
\kappa\right)  }\left(  \mathbb{R}^{d}\right)  .$
\end{notation}

With this notation the multiplication and Lie bracket on $T^{\left(
\kappa\right)  }\left(  \mathbb{R}^{d}\right)  $ may be described as,%
\[
AB=\pi_{\leq\kappa}\left(  A\otimes B\right)  \text{ and }\left[  A,B\right]
=\pi_{\leq\kappa}\left[  A,B\right]  _{\otimes}.
\]

\begin{notation}
[Induced Inner product]The usual dot product on $\mathbb{R}^{d}$ induces an
inner product, $\left\langle \cdot,\cdot,\right\rangle $ on $T^{\left(
\kappa\right)  }\left(  \mathbb{R}^{d}\right)  $ uniquely determined by
requiring $T^{\left(  \kappa\right)  }\left(  \mathbb{R}^{d}\right)
:=\oplus_{k=0}^{\kappa}\left[  \mathbb{R}^{d}\right]  ^{\otimes k}$ to be an
orthogonal direct sum decomposition, $\left\langle 1,1\right\rangle =1$ for
$1\in\left[  \mathbb{R}^{d}\right]  ^{\otimes0},$ and
\[
\left\langle v_{1}v_{2}\dots v_{k},w_{1}w_{2}\dots w_{k}\right\rangle
=\left\langle v_{1},w_{1}\right\rangle \left\langle v_{2},w_{2}\right\rangle
\dots\left\langle v_{k},w_{k}\right\rangle
\]
for any $v_{j},w_{j}\in\mathbb{R}^{d}$ and $1\leq k\leq\kappa.$ We let
$\left\vert A\right\vert :=\sqrt{\left\langle A,A\right\rangle }$ denote the
associated Hilbertian norm of $A\in T^{\left(  \kappa\right)  }\left(
\mathbb{R}^{d}\right)  .$
\end{notation}

It turns out to often be more convenient to measure the size of $A\in
\mathfrak{g}^{\left(  \kappa\right)  }$ using the following \textquotedblleft
homogeneous norms.\textquotedblright\

\begin{definition}
[Homogeneous norms]\label{def.10.13}For $A\in\mathfrak{g}^{\left(
\kappa\right)  }\subset T^{\left(  \kappa\right)  }\left(  \mathbb{R}%
^{d}\right)  ,$ let
\[
N\left(  A\right)  :=\max_{1\leq k\leq\kappa}\left\vert A_{k}\right\vert
^{1/k}.
\]

\end{definition}

\begin{definition}
[Free Nilpotent Lie Algebra]\label{not.10.14}The \textbf{step }$\kappa
$\textbf{ free Nilpotent Lie algebra} on $\mathbb{R}^{d}$ may then be realized
as the Lie sub-algebra, $F^{\left(  \kappa\right)  }\left(  \mathbb{R}%
^{d}\right)  ,$ of $\left(  T^{\left(  \kappa\right)  }\left(  \mathbb{R}%
^{d}\right)  ,\left[  \cdot,\cdot\right]  \right)  $ generated by
$\mathbb{R}^{d}\subset T^{\left(  \kappa\right)  }\left(  \mathbb{R}%
^{d}\right)  .$
\end{definition}

Again, a simple consequence of Remark \ref{rem.10.9} is that, as vector
spaces, $F^{\left(  \kappa\right)  }\left(  \mathbb{R}^{d}\right)  =\pi
_{\leq\kappa}\left(  F\left(  \mathbb{R}^{d}\right)  \right)  $ and
$F^{\left(  \kappa\right)  }\left(  \mathbb{R}^{d}\right)  $ is graded as%
\[
F^{\left(  \kappa\right)  }\left(  \mathbb{R}^{d}\right)  =\oplus
_{k=0}^{\kappa}F_{k}^{\left(  \kappa\right)  }\left(  \mathbb{R}^{d}\right)
\]
where
\[
F_{k}^{\left(  \kappa\right)  }\left(  \mathbb{R}^{d}\right)  :=F^{\left(
\kappa\right)  }\left(  \mathbb{R}^{d}\right)  \cap\left[  \mathbb{R}%
^{d}\right]  ^{\otimes k}\subset F^{\left(  \kappa\right)  }\left(
\mathbb{R}^{d}\right)  \text{ for }1\leq k\leq\kappa.
\]
It is not difficult to show (see see \cite[Section 1.2]{Driver2018}) using the
universal properties of the tensor algebra that the following notation is well defined.

\begin{notation}
\label{not.10.15}Given a dynamical system, $V:\mathbb{R}^{d}\rightarrow
\Gamma\left(  TM\right)  ,$ let $V^{\left(  \kappa\right)  }:F^{\left(
\kappa\right)  }\left(  \mathbb{R}^{d}\right)  \rightarrow\Gamma\left(
TM\right)  $ be the unique dynamical system such that
\[
V_{A}^{\left(  \kappa\right)  }:=L_{V_{w_{1}}}\dots L_{V_{w_{j-1}}}V_{w_{j}}%
\]
whenever $A=\operatorname{ad}_{w_{1}}\dots\operatorname{ad}_{w_{j-1}}w_{j}$
with $1\leq j\leq\kappa$ and $w_{i}\in\mathbb{R}^{d}.$ [For $A\in F^{\left(
\kappa\right)  }\left(  \mathbb{R}^{d}\right)  $ we will usually simply write
$V_{A}$ for $V_{A}^{\left(  \kappa\right)  }.]$
\end{notation}

\begin{remark}
\label{rem.10.16}It is \textbf{not} in general true that $V^{\left(
\kappa\right)  }:=V|_{F^{\left(  \kappa\right)  }\left(  \mathbb{R}%
^{d}\right)  }:F^{\left(  \kappa\right)  }\left(  \mathbb{R}^{d}\right)
\rightarrow\Gamma\left(  TM\right)  $ is a Lie algebra homomorphism. In order
for this to be true we must require that $L_{V_{a_{\kappa}}}\dots L_{V_{a_{1}%
}}V_{a_{0}}=0$ for all $\left\{  a_{j}\right\}  _{j=0}^{\kappa}\subset
\mathbb{R}^{d},$ i.e. $\left\{  V_{a}:a\in\mathbb{R}^{d}\right\}  $ should
generate a step-$\kappa$ nilpotent Lie sub-algebra of $\Gamma\left(
TM\right)  .$
\end{remark}

We now need to introduce a number of semi-norms on vector fields and dynamical
systems on $M.$

\begin{notation}
[Tensor Norms]\label{not.10.17}If $X\in\Gamma\left(  TM\right)  $ and $m\in
M,$ let
\begin{align*}
\left\vert X\right\vert _{m}  &  :=\left\vert X\left(  m\right)  \right\vert
_{g}\\
\left\vert \nabla X\right\vert _{m}  &  :=\sup_{\left\vert v_{m}\right\vert
=1}\left\vert \nabla_{v_{m}}X\right\vert _{g},\\
\left\vert \nabla^{2}X\right\vert _{m}  &  =\sup_{\left\vert v_{m}\right\vert
=1=\left\vert w_{m}\right\vert }\left\vert \nabla_{v_{m}\otimes w_{m}}%
^{2}X\right\vert _{g},\\
\left\vert R\left(  X,\cdot\right)  \right\vert _{m}  &  :=\sup_{\left\vert
v_{m}\right\vert =1=\left\vert w_{m}\right\vert }\left\vert R\left(  X\left(
m\right)  ,v_{m}\right)  w_{m}\right\vert _{g},\text{ and}\\
H_{m}\left(  X\right)   &  :=\left\vert \nabla^{2}X\right\vert _{m}+\left\vert
R\left(  X,\bullet\right)  \right\vert _{m}.
\end{align*}
We further let $\left\vert X\right\vert _{M}=\sup_{m\in M}\left\vert
X\right\vert _{m},\dots,H_{M}\left(  X\right)  :=\sup_{m\in M}H_{m}\left(
X\right)  .$
\end{notation}

\begin{definition}
[Dynamical system semi-norms]\label{def.10.18}When $V:\mathbb{R}%
^{d}\rightarrow\Gamma\left(  TM\right)  $ is a dynamical system of $\kappa
\in\mathbb{N},$ let%
\begin{align}
\left\vert V^{\left(  \kappa\right)  }\right\vert _{M}  &  :=\left\{
\left\vert V_{A}\right\vert _{M}:A\in F^{\left(  \kappa\right)  }\left(
\mathbb{R}^{d}\right)  \text{ with }\left\vert A\right\vert =1\right\}
,\label{e.10.10}\\
\left\vert \nabla V^{\left(  \kappa\right)  }\right\vert _{M}  &  :=\left\{
\left\vert \nabla V_{A}\right\vert _{M}:A\in F^{\left(  \kappa\right)
}\left(  \mathbb{R}^{d}\right)  \text{ with }\left\vert A\right\vert
=1\right\}  ,\text{ and}\label{e.10.11}\\
H_{M}\left(  V^{\left(  \kappa\right)  }\right)   &  :=\left\{  H_{M}\left(
V_{A}\right)  :A\in F^{\left(  \kappa\right)  }\left(  \mathbb{R}^{d}\right)
\text{ with }\left\vert A\right\vert =1\right\}  . \label{e.10.12}%
\end{align}
Any of these expressions are allowed to take plus infinity as a value.
\end{definition}

In order to introduce rough paths we need to go from Lie algebras to Lie
groups. Let $G^{\left(  \kappa\right)  }\left(  \mathbb{R}^{d}\right)
:=1+\mathfrak{g}^{\left(  \kappa\right)  }\subset T^{\left(  \kappa\right)
}\left(  \mathbb{R}^{d}\right)  $ which forms a group under the multiplication
rule of $T^{\left(  \kappa\right)  }\left(  \mathbb{R}^{d}\right)  .$ In fact,
$G^{\left(  \kappa\right)  }$ is a Lie group with Lie algebra,
$\operatorname*{Lie}\left(  G^{\left(  \kappa\right)  }\right)  =\mathfrak{g}%
^{\left(  \kappa\right)  }$ and the exponential map,
\[
\operatorname*{Lie}\left(  G^{\left(  \kappa\right)  }\right)  \ni
\xi\rightarrow e^{\xi}=\sum_{k=0}^{\kappa}\frac{\xi^{k}}{k!}\in G^{\left(
\kappa\right)  }\left(  \mathbb{R}^{d}\right)
\]
is a diffeomorphism where $\xi^{k}:=\pi_{\leq\kappa}\left(  \xi^{\otimes
k}\right)  $ inside of $T^{\left(  \kappa\right)  }\left(  \mathbb{R}%
^{d}\right)  .$ We will mostly only use the following subgroup of $G^{\left(
\kappa\right)  }\left(  \mathbb{R}^{d}\right)  .$

\begin{definition}
[Free Nilpotent Lie Groups]\label{not.10.19}For $\kappa\in\mathbb{N},$ let
$G_{\text{geo}}^{\left(  \kappa\right)  }\left(  \mathbb{R}^{d}\right)
\subset G^{\left(  \kappa\right)  }$ be the simply connected Lie subgroup of
$G^{\left(  \kappa\right)  }=1\oplus_{k=1}^{\kappa}\left[  \mathbb{R}%
^{d}\right]  ^{\otimes k}$ whose Lie algebra is $F^{\left(  \kappa\right)
}\left(  \mathbb{R}^{d}\right)  .$ This subgroup is a step-$\kappa$ (free)
nilpotent Lie group which we refer to as the \textbf{geometric sub-group }of
$G^{\left(  \kappa\right)  }.$
\end{definition}

It is well known as a consequence of the Baker-Campel-Dynken-Hausdorff formula
(see for example \cite[Proposition 3.12]{Driver2018}) that the exponential
map,%
\[
F^{\left(  \kappa\right)  }\left(  \mathbb{R}^{d}\right)  \ni\xi\rightarrow
e^{\xi}=\sum_{k=0}^{\kappa}\frac{\xi^{k}}{k!}\in G_{\text{geo}}^{\left(
\kappa\right)  }\left(  \mathbb{R}^{d}\right)  ,
\]
is a diffeomorphism. Furthermore that inverse, $\log,$ of this diffeomorphism
may be computed using%
\[
\log\left(  1+\xi\right)  =\sum_{k=1}^{\kappa}\frac{\left(  -1\right)  ^{k+1}%
}{k}\xi^{k}.
\]

\begin{definition}
[H\"{o}lder geometric rough paths]\label{def.10.20}For any $T\in\left(
0,\infty\right)  ,$ let%
\[
\Delta_{T}:=\left\{  \left(  s,t\right)  \in\left[  0,T\right]  ^{2}:0\leq
s\leq t\leq T\right\}  .
\]
Given $\alpha\in(\frac{1}{\kappa+1},\frac{1}{\kappa}],$ $X\in C\left(
\Delta_{T},G_{\text{geo}}^{\left(  \kappa\right)  }\left(  \mathbb{R}%
^{d}\right)  \right)  $ is an\textbf{ }$\alpha$\textbf{-H\"{o}lder geometric
rough path} if;
\end{definition}

\begin{enumerate}
\item $X_{s,s}=1$ for all $s\in\left[  0,T\right]  ,$

\item $X_{st}X_{tu}=X_{su}$ for all $0\leq s\leq t\leq u\leq T,$ and

\item there is a constant $C<\infty$ such that
\[
N\left(  X_{s,t}\right)  \leq C\left\vert t-s\right\vert ^{\alpha}\text{ for
all }\left(  s,t\right)  \in\Delta_{T}.
\]

\end{enumerate}

\begin{example}
[Smooth Rough paths]\label{ex.10.21}If $x\in C^{1}\left(  \left[  0,T\right]
,\mathbb{R}^{d}\right)  ,$ $X_{st}^{0}:=1,$ and for $1\leq k\leq\kappa,$%
\[
X_{st}^{k}:=\int_{s\leq\sigma_{1}\leq\sigma_{2}\leq\dots\leq\sigma_{k}\leq
t}dx\left(  \sigma_{1}\right)  \otimes dx\left(  \sigma_{2}\right)
\otimes\dots\otimes dx\left(  \sigma_{k}\right)
\]
for all $\left(  s,t\right)  \in\Delta_{T},$ then $X_{st}:=\sum_{k=0}^{\kappa
}X_{st}^{k}\in G_{geo}^{\left(  \kappa\right)  }$ is an $1$-H\"{o}lder
geometric rough path. In fact if $C=\max_{s\in\left[  0,T\right]  }\left\vert
\dot{x}\left(  t\right)  \right\vert ,$ then
\begin{align*}
\left\vert X_{st}^{k}\right\vert \leq &  \int_{s\leq\sigma_{1}\leq\sigma
_{2}\leq\dots\leq\sigma_{k}\leq t}\left\vert dx\left(  \sigma_{1}\right)
\right\vert \left\vert dx\left(  \sigma_{2}\right)  \right\vert \dots
\left\vert dx\left(  \sigma_{k}\right)  \right\vert \\
&  =\frac{1}{k!}\left(  \int_{s}^{t}\left\vert dx\left(  \sigma\right)
\right\vert \right)  ^{k}\leq\frac{C^{k}}{k!}\left\vert t-s\right\vert ^{k}.
\end{align*}
It is known that every $\alpha$-H\"{o}lder rough path is certain limit of a
sequence of such smooth rough paths, see \cite[Section 8.6]{Friz2010}.
\end{example}

\begin{example}
[Enhanced Brownian Motion]\label{ex.10.22}If $\left\{  B_{t}\right\}
_{t\geq0}$ is an $\mathbb{R}^{d}$ -valued Brownian motion, let $B_{s,t}%
:=B_{t}-B_{s},$ and
\[
B_{s,t}^{2}:=\int_{s}^{t}\left(  B_{\sigma}-B_{s}\right)  \otimes\delta
B_{\sigma}%
\]
where $\delta B$ denote the Stratonovich differential of $B.$ Then%
\[
\mathbf{B}_{s,t}:=1+B_{s,t}+B_{s,t}^{2}\in G_{geo}^{\left(  2\right)  }%
\]
is almost surely a $\alpha$-H\"{o}lder geometric rough path for any
$0<\alpha<\frac{1}{2}.$ For this example and other Gaussian process examples,
see \cite[Section 13.2 and Chapter 15]{Friz2010} and the references therein.
\end{example}

\begin{notation}
\label{not.10.23}Given a geometric rough path as in Definition \ref{def.10.20}%
, we extend $X_{s,t}$ to all $\left(  s,t\right)  \in\left[  0,T\right]  ^{2}$
by
\[
X_{s,t}:=X_{t,s}^{-1}\text{ when }0\leq t\leq s\leq T.
\]

\end{notation}

\begin{remark}
\label{rem.10.24}From items 1. and 2. of Definition \ref{def.10.20}, if we let
$g_{t}:=X_{0,t},$ then $g_{0}=1\in G_{\text{geo}}^{\left(  \kappa\right)
}\left(  \mathbb{R}^{d}\right)  $ and for $0\leq s\leq t\leq T,$
\[
g_{t}=X_{0,t}=X_{0,s}X_{st}=g_{s}X_{st}\implies X_{st}=g_{s}^{-1}g_{t}.
\]
Thus we see it is natural to extend the definition of $X$ using
\begin{equation}
X_{s,t}:=g_{s}^{-1}g_{t}\text{ for all }\left(  s,t\right)  \in\left[
0,T\right]  ^{2}. \label{e.10.13}%
\end{equation}
This extension satisfies, $X_{s,t}=X_{t,s}^{-1}$ for all $s,t\in\left[
0,T\right]  $ and hence is consistent with Notation \ref{not.10.23}. From Eq.
(\ref{e.10.13}), it is simple to verify $X_{s,t}X_{t,u}=X_{su}$ for all
$s,t,u\in\left[  0,T\right]  .$ Moreover, as a consequence of Corollary
\ref{cor.16.4}, we also have
\[
\left\vert X_{s,t}^{k}\right\vert =\left\vert X_{t,s}^{k}\right\vert \leq
C\left\vert t-s\right\vert ^{\alpha k}\text{ for }1\leq k\leq\kappa\text{ and
}\left(  s,t\right)  \in\left[  0,T\right]  ^{2}.
\]

\end{remark}

\begin{ass}
\label{ass.2}Throughout this paper we assume that $0<\alpha\leq1$ and
$\kappa\in\mathbb{N}$ always satisfy,
\begin{equation}
\theta:=\alpha\left(  \kappa+1\right)  >1. \label{e.10.14}%
\end{equation}

\end{ass}

\subsection{Statement of the Main Results\label{sec.10.3}}

For the rest of this paper let $X$ be a H\"{o}lder rough path as in Definition
\ref{def.10.20} and $V$ be a $d$-dimensional dynamical system on $M$ as in
Definition \ref{def.10.3}. When $d\left(  \cdot,\cdot\right)  $ the Riemannian
distance function associated to a Riemannian metric, $g,$ on $M$ and
$f,g:M\rightarrow M$ are two maps and $U$ is a subset of $M,$ we let%
\begin{equation}
d_{U}\left(  f,g\right)  :=\sup_{m\in U}d\left(  f\left(  m\right)  ,g\left(
m\right)  \right)  . \label{e.10.15}%
\end{equation}
We will mostly use this definition with $U=M.$

\begin{ass}
\label{ass.3}Throughout this paper we assume that $V_{A}\in\Gamma\left(
TM\right)  $ is complete for all $A\in F^{\left(  \kappa\right)  }\left(
\mathbb{R}^{d}\right)  .$
\end{ass}

A standard condition that guarantees a vector field, $Y\in\Gamma\left(
TM\right)  ,$ is complete is to assume there is a complete metric $g$ on $M$
such that
\begin{equation}
\sup_{m\in M}\frac{\left\vert Y\left(  m\right)  \right\vert _{g}}{1+d\left(
o,m\right)  }<\infty\label{e.10.16}%
\end{equation}
where $o$ is a fixed point in $M$ and $d$ is the length metric associated to
$g.$ See Lemma \ref{lem.15.3} and Examples \ref{ex.15.4} and \ref{ex.15.5} of
Appendix \ref{sec.15} for a review of this fact along with some extensions to
this type of result. It should also be noted that if $\left\vert \nabla
Y\right\vert _{M}<\infty,$ then Eq. (\ref{e.10.16}) holds, see for example
\cite[Lemma 2.10]{Driver2018}.

\begin{definition}
[Approximate flows]\label{def.10.25}If $X$ is an $\alpha$-H\"{o}lder rough
path and $V:\mathbb{R}^{d}\rightarrow\Gamma\left(  TM\right)  $ is dynamical
system satisfying Assumption \ref{ass.3}, let
\begin{equation}
\mu_{t,s}=\mu_{t,s}^{X}:=e^{V_{\log\left(  X_{st}\right)  }^{\left(
\kappa\right)  }}\in\mathrm{Diff}\left(  M\right)  \text{ }\forall~\left(
s,t\right)  \in\left[  0,T\right]  ^{2}. \label{e.10.17}%
\end{equation}

\end{definition}

The next simple proposition indicates the importance of $\mu_{t,s}.$

\begin{proposition}
[Nilpotent Flows]\label{pro.10.26}If $\left\{  V_{a}:a\in\mathbb{R}%
^{d}\right\}  \subset\Gamma\left(  TM\right)  $ generates a step-$\kappa$
Nilpotent Lie sub-algebra then $V^{\left(  \kappa\right)  }:F^{\left(
\kappa\right)  }\left(  \mathbb{R}^{d}\right)  \rightarrow\Gamma\left(
TM\right)  $ is a Lie-algebra homomorphism. Moreover, $\mu_{t,s}$ is
multiplicative (has the flow property),%
\begin{equation}
\mu_{t,s}\circ\mu_{s,r}=\mu_{t,r}\text{ }\forall~\left(  s,t\right)
\in\left[  0,T\right]  ^{2}. \label{e.10.18}%
\end{equation}

\end{proposition}

\begin{proof}
This follows from basic Lie theoretic considerations. Here is a proof based on
\cite[Corollary 4.7]{Driver2018} which under the given assumptions states that%
\[
e^{V_{B}^{\left(  \kappa\right)  }}\circ e^{V_{A}^{\left(  \kappa\right)  }%
}=e^{V_{\log\left(  e^{A}e^{B}\right)  }^{\left(  \kappa\right)  }}\text{ for
all }A,B\in F^{\left(  \kappa\right)  }\left(  \mathbb{R}^{d}\right)  .
\]
Taking $A=\log\left(  X_{r,s}\right)  $ and $B=\log\left(  X_{s,t}\right)  $
in this identity while using
\[
\log\left(  e^{A}e^{B}\right)  =\log\left(  X_{r,s}X_{s,t}\right)
=\log\left(  X_{r,t}\right)
\]
gives the flow identity in Eq. (\ref{e.10.18}).
\end{proof}

When $\left\{  V_{a}:a\in\mathbb{R}^{d}\right\}  $ does not generate a
step-$\kappa$ Nilpotent Lie sub-algebra, the flow property in Eq.
(\ref{e.10.18}) will no longer hold. Nevertheless, the following main theorem
of this paper gives necessary conditions on $V$ so that there is a unique
\textquotedblleft flow\textquotedblright\ on $M$ close to $\mu_{t,s}.$

\begin{theorem}
[Global Existence]\label{thm.10.27}Suppose that $V:\mathbb{R}^{d}%
\rightarrow\Gamma\left(  TM\right)  $ is a dynamical system and $g$ is a
complete metric on $M$ such that $\left\vert V^{\left(  \kappa\right)
}\right\vert _{M}+\left\vert \nabla V^{\left(  \kappa\right)  }\right\vert
_{M}<\infty.$ Then Assumption \ref{ass.3} holds (so $\mu_{t,s}$ in Definition
\ref{def.10.25} is well defined) and there exists a unique function
$\varphi\in C\left(  \left[  0,T\right]  ^{2}\times M,M\right)  $ such that;

\begin{enumerate}
\item $\varphi_{t,t}=Id$ for all $t\in\left[  0,T\right]  ,$

\item $\varphi_{t,s}\circ\varphi_{s,r}=\varphi_{t,r}$ $\forall~\left(
s,t\right)  \in\left[  0,T\right]  ^{2},$ and

\item there exists a constant $C<\infty$ such that
\begin{equation}
d_{M}\left(  \varphi_{t,s},\mu_{t,s}\right)  \leq C\left\vert t-s\right\vert
^{\theta}\text{ ~}\forall~\left(  s,t\right)  \in\left[  0,T\right]  ^{2},
\label{e.10.19}%
\end{equation}
where $\theta=\left(  \kappa+1\right)  \alpha>1$ as in Assumption \ref{ass.2}.
Moreover, $C\left(  K\right)  :=\sup_{\left(  s,t\right)  \in\left[
0,T\right]  ^{2}}\operatorname{Lip}_{K}\left(  \varphi_{t,s}\right)  <\infty$
for all compact subsets, $K,$ of $M.$\footnote{Items 1. and 2. imply that
$\varphi_{t,s}$ is invertible with inverse given by $\varphi_{s,t}$ for all
$\left(  s,t\right)  \in\left[  0,T\right]  ^{2}.$ Thus each $\varphi
_{t,s}:M\rightarrow M$ is a locally Lipshitz homeomorphism for all
$s,t\in\left[  0,T\right]  .$}
\end{enumerate}
\end{theorem}

Stochastic variants (of one kind or another) of the estimate in Eq.
(\ref{e.10.19}) occur frequently, for example see
\cite{BenArous1989,BenArous1988,Takanobu1988,Takanobu1990,Castell1993,Castell1996,Inahama2010a,Inahama2017}
for a few examples. Theorem \ref{thm.10.27} gives a notion of solution to
rough differential equations which is championed by Bailleul, see for example
\cite{Bailleul2014a,Bailleul2015a,Bailleul2018a,Bailleul2018}. However, it
should be noted that a precursor to this approximate flow notion of solution
already appeared in Davie in \cite{Davie2007} which gives a similar definition
at the level of paths. The next proposition shows that the notion of solution
to rough differential equations used in Theorem \ref{thm.10.27} gives a
solution in the sense of Davie. For further information in this direction the
reader is referred to \cite{Cass2016a} and \cite{Bailleul2018a}.

\begin{proposition}
[Path characterization of solutions]\label{pro.10.28}If we fix $m_{0}\in M$
and let $y_{t}:=\varphi_{t,0}\left(  m_{0}\right)  ,$ then $y$ satisfies
\[
\left\vert f\left(  y_{t}\right)  -\left(  V_{X_{s,t}}f\right)  \left(
y_{s}\right)  \right\vert \leq C\left(  f\right)  \left\vert t-s\right\vert
^{\theta}%
\]
where $f$ is a smooth function on $M,$ $\theta=\left(  \kappa+1\right)
\alpha>1$ as in Assumption \ref{ass.2}, and $C\left(  f\right)  $ depends on
the bounds on $f$ and its derivatives to order $2\kappa$ over a compact
neighborhood of $y_{\left[  0,T\right]  }.$
\end{proposition}

\begin{proof}
Let us shows $y$ satisfies the above estimate. To see this we observe that
$y_{t}=\varphi_{t,s}\left(  y_{s}\right)  $ and hence
\begin{align*}
f\left(  y_{t}\right)   &  =f\circ\varphi_{t,s}\left(  y_{s}\right)
=f\circ\mu_{t,s}\left(  y_{s}\right)  +O\left(  \left\vert t-s\right\vert
^{\theta}\right) \\
&  =f\circ e^{V_{\log\left(  X_{st}\right)  }^{\left(  \kappa\right)  }%
}\left(  y_{s}\right)  +O\left(  \left\vert t-s\right\vert ^{\theta}\right)
\end{align*}
and
\begin{align*}
f\circ e^{V_{\log\left(  X_{st}\right)  }^{\left(  \kappa\right)  }}\left(
y_{s}\right)   &  =\sum_{k=0}^{\kappa}\frac{1}{k!}\left(  V_{\log\left(
X_{st}\right)  }^{\left(  \kappa\right)  k}f\right)  \left(  y_{s}\right)
+O\left(  \left\vert t-s\right\vert ^{\theta}\right) \\
&  =\left(  \sum_{k=0}^{\kappa}\frac{1}{k!}V_{\log\left(  X_{st}\right)
}^{\left(  \kappa\right)  k}f\right)  \left(  y_{s}\right)  +O\left(
\left\vert t-s\right\vert ^{\theta}\right)  .
\end{align*}
Finally, $\sum_{k=0}^{\kappa}\frac{1}{k!}V_{\log\left(  X_{st}\right)  }%
^{k}=V_{B}$ where
\begin{align*}
B  &  =\sum_{k=0}^{\kappa}\frac{1}{k!}\left[  \log\left(  X_{st}\right)
\right]  ^{\otimes k}=\pi_{\leq\kappa}\sum_{k=0}^{\kappa}\frac{1}{k!}\left[
\log\left(  X_{st}\right)  \right]  ^{\otimes k}+O\left(  \left\vert
t-s\right\vert ^{\theta}\right) \\
&  =e^{\log\left(  X_{st}\right)  }+O\left(  \left\vert t-s\right\vert
^{\theta}\right)  =X_{st}+O\left(  \left\vert t-s\right\vert ^{\theta}\right)
\end{align*}
where $O\left(  \left\vert t-s\right\vert ^{\theta}\right)  \in\oplus
_{k=\kappa+1}^{2\kappa}\left[  \mathbb{R}^{d}\right]  ^{\otimes k}.$ Combining
these estimates shows,%
\[
f\left(  y_{t}\right)  =\left(  V_{X_{s,t}}f\right)  \left(  y_{s}\right)
+O\left(  \left\vert t-s\right\vert ^{\theta}\right)  .
\]

\end{proof}

Our next goal is to remove the hypothesis in Theorem \ref{thm.10.27} that
$\left\vert V^{\left(  \kappa\right)  }\right\vert _{M}<\infty.$ We begin with
the following simple lemma.

\begin{lemma}
\label{lem.10.29}Suppose $\left(  M,g,o\right)  $ is a pointed Riemannian
manifold, $\kappa\in\mathbb{N},$ and $V:\mathbb{R}^{d}\rightarrow\Gamma\left(
TM\right)  $ is a dynamical system on $M.$ If $\left\vert \nabla V^{\left(
\kappa\right)  }\right\vert _{M}<\infty$ and $C_{\kappa}\left(  V,o\right)
:=\max\left(  \left\vert \nabla V^{\left(  \kappa\right)  }\right\vert
_{M},\left\vert V^{\left(  \kappa\right)  }\left(  o\right)  \right\vert
\right)  <\infty,$ then%
\begin{align}
\left\vert V_{A}\left(  m\right)  \right\vert  &  \leq\left(  \left\vert
V^{\left(  \kappa\right)  }\left(  o\right)  \right\vert +d\left(  o,m\right)
\left\vert \nabla V^{\left(  \kappa\right)  }\right\vert _{M}\right)
\left\vert A\right\vert \nonumber\\
&  \leq C_{\kappa}\left(  V,o\right)  \left\vert A\right\vert \left(
1+d\left(  o,m\right)  \right)  \label{e.10.20}%
\end{align}
for all $A\in F^{\left(  \kappa\right)  }\left(  \mathbb{R}^{d}\right)  $ and
$m\in M.$
\end{lemma}

\begin{proof}
From \cite[Lemma 2.10]{Driver2018} with $X\left(  m\right)  :=V_{A}\left(
m\right)  $ and $p=o$ states that
\[
\left\vert \left\vert V_{A}\left(  m\right)  \right\vert -\left\vert
V_{A}\left(  o\right)  \right\vert \right\vert \leq\left\vert \nabla
V_{A}\right\vert _{M}d\left(  o,m\right)  .
\]
The estimates in Eq. (\ref{e.10.20}) are now an elementary consequence of the
previous inequality.
\end{proof}

\begin{notation}
\label{not.10.30}Let $\left(  M,o\right)  $ be a pointed manifold. To each
Riemannian metric, $g,$ on $M,$ let $\bar{g}$ be the continuous Riemannian
metric defined by%
\[
\bar{g}\left(  v,w\right)  =\left(  \frac{1}{1+d\left(  o,m\right)  }\right)
^{2}g\left(  v,w\right)  \text{ }\forall~v,w\in T_{m}M\text{ and }m\in M
\]
and let $\bar{d}$ be the length metric associated to the continuous metric.
\end{notation}

\begin{corollary}
\label{cor.10.31}Suppose that $V:\mathbb{R}^{d}\rightarrow\Gamma\left(
TM\right)  $ is a dynamical system satisfying Assumption \ref{ass.3} and there
exists a complete metric, $g,$ on $M$ such that $\left\vert \nabla V^{\left(
\kappa\right)  }\right\vert _{M}<\infty.$\footnote{From Lemma \ref{lem.10.29},
the assumption that $\left\vert \nabla V^{\left(  n\right)  }\right\vert
_{M}<\infty$ implies $V_{A}\left(  m\right)  $ has at most linear growth in
$m$ for all $A\in F^{\left(  n\right)  }\left(  \mathbb{R}^{d}\right)  .$}
Then there exists a unique function $\varphi\in C\left(  \left[  0,T\right]
^{2}\times M,M\right)  $ such that;

\begin{enumerate}
\item $\varphi_{t,t}=Id$ for all $t\in\left[  0,T\right]  ,$

\item $\varphi_{t,s}\circ\varphi_{s,r}=\varphi_{t,r}$ $\forall~0\leq r\leq
s\leq t\leq T,$ and

\item there exists a constant $C<\infty$ such that
\begin{align*}
\bar{d}_{M}\left(  \varphi_{t,s},\mu_{t,s}\right)   &  =\sup_{m\in M}\bar
{d}\left(  \varphi_{t,s}\left(  m\right)  ,\mu_{t,s}\left(  m\right)  \right)
\\
&  \leq C\left\vert t-s\right\vert ^{\theta}\text{ ~}\forall~\left(
s,t\right)  \in\left[  0,T\right]  ^{2},
\end{align*}
where $\bar{d}$ is as in Notation \ref{not.10.30}.
\end{enumerate}
\end{corollary}

\begin{proof}
This is an immediate consequence of Theorem \ref{thm.10.27} in conjunction
with Corollary \ref{cor.15.9} and Proposition \ref{pro.15.10} of Appendix
\ref{sec.15} below.
\end{proof}

\subsection{Examples\label{sec.10.4}}

\begin{example}
\label{ex.10.33}Let $M=G$ be a Lie group with $\mathfrak{g}:=T_{e}G$ being the
Lie algebra and define $V_{\xi}=\tilde{\xi}\in\Gamma\left(  TG\right)  $ for
all $\xi\in\mathfrak{g.}$ Then $V$ satisfies the hypothesis of Theorem
\ref{thm.10.27} relative to any left invariant Riemannian metric on $G.$
\end{example}

\begin{example}
\label{ex.10.34}Let $M=\mathbb{R}^{D}$ with the standard Euclidean Riemannian
metric. If $V_{i}\in\Gamma\left(  T\mathbb{R}^{D}\right)  $ are bounded vector
fields on $\mathbb{R}^{D}$ with bounded derivative to sufficiently high order,
then $V$ satisfies the hypothesis of Theorem \ref{thm.10.27}.
\end{example}

The proofs of the next proposition and corollary may be found in Section
\ref{sec.14} at the end of the paper.

\begin{proposition}
\label{pro.10.35}Suppose that $\left(  M^{d},g\right)  $ is a complete
Riemannian manifold, $TM$ is parallelizable, and $V:\mathbb{R}^{d}%
\rightarrow\Gamma\left(  TM\right)  $ is a dynamical system such that%
\begin{equation}
g\left(  V_{a}\left(  m\right)  ,V_{b}\left(  m\right)  \right)  =a\cdot
b\text{ for all }a,b\in\mathbb{R}^{d}\text{ and }m\in M. \label{e.10.21}%
\end{equation}
Further let $Q\left(  a,b\right)  \in C^{\infty}\left(  M,\mathbb{R}%
^{d}\right)  $ is determined by
\begin{equation}
\left[  V_{a},V_{b}\right]  =V_{Q\left(  a,b\right)  }\text{ for all }%
a,b\in\mathbb{R}^{d}. \label{e.10.22}%
\end{equation}
[Notice that $Q:\mathbb{R}^{d}\times\mathbb{R}^{d}\rightarrow C^{\infty
}\left(  M,\mathbb{R}^{d}\right)  $ is a skew-symmetric bilinear map.] If
$V_{a_{1}}\dots V_{a_{k}}Q\left(  a,b\right)  \in C^{\infty}\left(
M,\mathbb{R}^{d}\right)  $ is bounded for all $1\leq k\leq\kappa-1,$
$a,b\in\mathbb{R}^{d},$ and $a_{j}\in\mathbb{R}^{d},$ then $V$ satisfies the
hypothesis of Theorem \ref{thm.10.27}.
\end{proposition}

We end the introduction with some necessary conditions that the rough version
of Cartan's rolling map has global in time solutions. Let $\left(
M^{d},g\right)  $ be a Riemannian manifold of dimension $d,$ $O\left(
d\right)  $ be the Lie group of orthogonal $d\times d$ matrices, and
$so\left(  d\right)  $ be the Lie algebra of $O\left(  d\right)  $ consisting
of recall skew symmetric $d\times d$ - matrices. In order to describe Cartan's
rolling map we need to recall the following orthogonal frame bundle notations.

\begin{notation}
[Orthogonal frame bundle]\label{not.10.36}Let $\pi:O\left(  M\right)
\rightarrow M$ be the principal bundle of orthogonal frames, i.e.
\[
O\left(  M\right)  =\cup_{m\in M}O_{m}\left(  M\right)  =\cup_{m\in M}\left[
\pi^{-1}\left(  \left\{  m\right\}  \right)  \right]
\]
where for each $m\in M$ , $O_{m}\left(  M\right)  $ is the set of isometries,
$u:\mathbb{R}^{d}\rightarrow T_{m}M.$ To each $A\in so\left(  d\right)  ,$ let
$A^{\ast}\in\Gamma\left(  TO\left(  M\right)  \right)  $ be the
\textbf{vertical }vector field defined by%
\[
A^{\ast}\left(  u\right)  :=\frac{d}{dt}|_{0}ue^{tA}\text{ for all }u\in
O\left(  M\right)
\]
and to each $a\in\mathbb{R}^{d}$ let $B_{a}$ be the \textbf{horizontal vector
field }defined by%
\[
B_{a}\left(  u\right)  :=\frac{d}{dt}|_{0}\pt_{t}^{\nabla}\left(
\sigma\right)  u\in T_{u}O\left(  M\right)
\]
where $\pt_{t}^{\nabla}\left(  \sigma\right)  $ denote parallel translation
along any smooth path, $\sigma\left(  t\right)  \in M,$ such that $\dot
{\sigma}\left(  0\right)  =ua\in T_{m}M.$
\end{notation}

Given a smooth path $x\left(  t\right)  \in\mathbb{R}^{d}$ and $u_{o}\in
O_{m_{o}}\left(  M\right)  ,$ Cartan's rolling of $x$ onto $M$ is the path
$\sigma\left(  t\right)  =\pi\left(  u\left(  t\right)  \right)  \in M$ where
$u\left(  t\right)  \in O\left(  M\right)  $ satisfies the ODE
\begin{equation}
\dot{u}\left(  t\right)  =B_{\dot{x}\left(  t\right)  }\left(  u\left(
t\right)  \right)  \text{ with }u\left(  0\right)  =u_{o}\in O\left(
M\right)  . \label{e.10.23}%
\end{equation}
Thus we want to consider the existence of solutions to the rough version of
Eq. (\ref{e.10.23}). We will in fact consider the more general ODE for
$u\left(  t\right)  \in O\left(  M\right)  ;$
\begin{equation}
\dot{u}\left(  t\right)  =B_{\dot{x}\left(  t\right)  }\left(  u\left(
t\right)  \right)  +A^{\ast}\left(  t\right)  \left(  u\left(  t\right)
\right)  \text{ with }u\left(  0\right)  =u_{o}\in O\left(  M\right)  .
\label{e.10.24}%
\end{equation}
where now $\left(  x\left(  t\right)  ,A\left(  t\right)  \right)
\in\mathbb{R}^{d}\times so\left(  d\right)  $ is a given path.

\begin{corollary}
\label{cor.10.37}Suppose that $\left(  M,g\right)  $ is a complete Riemannian
manifold with bounded geometry to order $\kappa-1,$ i.e. $\nabla^{k}R$ is
bounded for $0\leq k\leq\kappa-1,$ then for any $\alpha$ - H\"{o}lder rough
path, $\mathbf{X}_{s,t}\in G_{geo}^{\left(  \kappa\right)  }\left(
\mathbb{R}^{d}\times so\left(  d\right)  \right)  $ with $\theta=\alpha\left(
\kappa+1\right)  >1,$ the rough version of Eq. (\ref{e.10.24}) has a unique
solution defined for all $t\in\left[  0,T\right]  .$
\end{corollary}

The proof of this corollary will amount to showing that the dynamical system,
$V:\mathbb{R}^{d}\times so\left(  d\right)  \rightarrow\Gamma\left(  TO\left(
M\right)  \right)  ,$ defined by
\begin{equation}
V_{\left(  a,A\right)  }\left(  u\right)  :=B_{a}\left(  u\right)  +A^{\ast
}\left(  u\right)  \label{e.10.25}%
\end{equation}
satisfies the hypothesis of Proposition \ref{pro.10.35} under the given
bounded geometry assumptions.

\subsection{Acknowledgments\label{sec.10.5}}

The author thanks Fabrice Baudoin and Thomas Cass for alerting him to the
papers of Brault and Lejay \cite{brault:hal-01716945,brault:hal-01839202} and
Weidner \cite{Weidner2018} respectively. I am also very grateful to Masha
Gordina for many illuminating conversations on this work and to her
hospitality and to that of the mathematics department at the University of
Connecticut where this work was started while the author was on sabbatical in
the Fall of 2017.

\section{A metric space almost multiplicative function theorem\label{sec.11}}

This section is devoted to a version of the Lyons' almost multiplicative
function theorem in the context of Bi-Lipschitz maps on an abstract complete
metric space which generalizes those in \cite[Theorem 2.1]{Bailleul2015a},
\cite{Feyel2008}, \cite{Feyel2006}, \cite{Gubinelli2004}, and \cite{Lyons1994}
in reverse chronological order. As mentioned in the introduction this topic is
also taken up in \cite{brault:hal-01716945,brault:hal-01839202} and
\cite{Weidner2018}. Ideas of this form for smooth function of time are already
prevalent in numerical and functional analysis literature, see for example the
review article \cite{Chorin1978}.

\subsection{Function Space Metrics\label{sec.11.1}}

\begin{remark}
\label{rem.11.1}Suppose that $\left(  X,\tau\right)  $ is a topological space,
$\left(  M,d\right)  $ is a complete metric space, and $\left\{
f_{n}\right\}  _{n=1}^{\infty}\subset C\left(  X,M\right)  $ satisfy
\[
\lim_{m,n\rightarrow\infty}\sup_{x\in X}d\left(  f_{n}\left(  x\right)
,f_{m}\left(  x\right)  \right)  =0.
\]
Then, as is well known, there exists $f\in C\left(  X,M\right)  $ such that%
\[
\lim_{n\rightarrow\infty}\sup_{x\in X}d\left(  f\left(  x\right)
,f_{n}\left(  x\right)  \right)  =0.
\]

\end{remark}

\begin{notation}
\label{not.11.2}When $\left(  M,d\right)  $ is a metric space let
$\operatorname{Homeo}\left(  M\right)  $ denote those $f\in C\left(
M,M\right)  $ which are homeomorphisms. Also for $f,g\in C\left(  M,M\right)
,$ let $d_{M}\left(  f,g\right)  :=\sup_{m\in M}d\left(  f\left(  m\right)
,g\left(  m\right)  \right)  $ as in Eq. (\ref{e.10.15}).
\end{notation}

\begin{definition}
\label{def.11.3}We say $f\in C\left(  M,M\right)  $ is \textbf{Lipschitz} if
there exists $K=K\left(  f\right)  <\infty$ such that
\[
d\left(  f\left(  m\right)  ,f\left(  m^{\prime}\right)  \right)  \leq
Kd\left(  m,m^{\prime}\right)  \text{ }\forall~m,m^{\prime}\in M.
\]
We denote the best such constant by $\operatorname{Lip}\left(  f\right)  ,$
i.e.
\[
\operatorname{Lip}\left(  f\right)  :=\sup_{m\neq m^{\prime}}\frac{d\left(
f\left(  m\right)  ,f\left(  m^{\prime}\right)  \right)  }{d\left(
m,m^{\prime}\right)  }.
\]
We will write $\operatorname{Lip}\left(  f\right)  =\infty$ if $f$ is not Lipschitz.
\end{definition}

\begin{remark}
\label{rem.11.4}If $f,g,g_{1},g_{2}\in C\left(  M,M\right)  ,$ then
\[
d_{M}\left(  g_{1}\circ f,g_{2}\circ f\right)  \leq d_{M}\left(  g_{1}%
,g_{2}\right)
\]
while for $m\in M,$%
\[
d\left(  f\circ g_{1}\left(  m\right)  ,f\circ g_{2}\left(  m\right)  \right)
\leq\operatorname{Lip}\left(  f\right)  d\left(  g_{1}\left(  m\right)
,g_{2}\left(  m\right)  \right)  \leq\operatorname{Lip}\left(  f\right)
d_{M}\left(  g_{1},g_{2}\right)
\]
and hence%
\[
d_{M}\left(  f\circ g_{1},f\circ g_{2}\right)  \leq\operatorname{Lip}\left(
f\right)  \cdot d_{M}\left(  g_{1},g_{2}\right)  .
\]
Also for $m,m^{\prime}\in M,$
\[
d\left(  f\circ g\left(  m\right)  ,f\circ g\left(  m^{\prime}\right)
\right)  \leq\operatorname{Lip}\left(  f\right)  d\left(  g\left(  m\right)
,g\left(  m^{\prime}\right)  \right)  \leq\operatorname{Lip}\left(  f\right)
\operatorname{Lip}\left(  g\right)  d\left(  m,m^{\prime}\right)
\]
from which it follows that
\[
\operatorname{Lip}\left(  f\circ g\right)  \leq\operatorname{Lip}\left(
f\right)  \operatorname{Lip}\left(  g\right)  .
\]

\end{remark}

More generally we have the following extensions of these results.

\begin{proposition}
\label{pro.11.5}Suppose that $f_{j},g_{j}\in C\left(  M,M\right)  $ for $1\leq
j\leq n,$ then%
\begin{align*}
d_{M}  &  \left(  f_{n}\circ\dots\circ f_{1},g_{n}\circ\dots\circ g_{1}\right)
\\
&  \leq\sum_{k=1}^{n}\left[  \operatorname{Lip}\left(  f_{n}\circ\dots\circ
f_{k+1}\right)  \wedge\operatorname{Lip}\left(  g_{n}\circ\dots\circ
g_{k+1}\right)  \right]  d_{M}\left(  f_{k},g_{k}\right)
\end{align*}
and in particular,
\[
d_{M}\left(  f_{1}\circ\dots\circ f_{n},g_{1}\circ\dots\circ g_{n}\right)
\leq\sum_{k=1}^{n}\left[  \operatorname{Lip}\left(  f_{n}\circ\dots\circ
f_{k+1}\right)  \right]  d_{M}\left(  f_{k},g_{k}\right)  .
\]
The convention used above is that $\operatorname{Lip}\left(  f_{n}\circ
\dots\circ f_{k+1}\right)  \equiv1$ when $k=n.${}
\end{proposition}

\begin{proof}
The proof is by induction on $n$ where in the case $n=1,$ there is nothing to
prove. The induction step is as follows;%
\begin{align*}
d_{M}  &  \left(  f_{n}\circ\dots\circ f_{1},g_{n}\circ\dots\circ g_{1}\right)
\\
&  \leq d_{M}\left(  f_{n}\circ\dots\circ f_{2}\circ f_{1},f_{n}\circ
\dots\circ f_{2}\circ g_{1}\right) \\
&  \qquad\qquad+d_{M}\left(  f_{n}\circ\dots\circ f_{2}\circ g_{1},g_{n}%
\circ\dots\circ g_{2}\circ g_{1}\right) \\
&  \leq\operatorname{Lip}\left(  f_{n}\circ\dots\circ f_{2}\right)
d_{M}\left(  f_{1},g_{1}\right)  +d_{M}\left(  f_{n}\circ\dots\circ
f_{2},g_{n}\circ\dots\circ g_{2}\right)  .
\end{align*}
By interchanging the roles of $f$ and $g$ we further can show,%
\begin{align*}
d_{M}\left(  f_{n}\circ\dots\circ f_{1},g_{n}\circ\dots\circ g_{1}\right)
\leq &  \operatorname{Lip}\left(  g_{n}\circ\dots\circ g_{2}\right)
d_{M}\left(  f_{1},g_{1}\right) \\
&  +d_{M}\left(  f_{n}\circ\dots\circ f_{2},g_{n}\circ\dots\circ g_{2}\right)
\end{align*}
which combined with the previous inequality allows us to conclude that
\begin{align*}
d_{M}  &  \left(  f_{n}\circ\dots\circ f_{1},g_{n}\circ\dots\circ g_{1}\right)
\\
&  \leq\operatorname{Lip}\left(  f_{n}\circ\dots\circ f_{2}\right)
\wedge\operatorname{Lip}\left(  g_{n}\circ\dots\circ g_{2}\right)
d_{M}\left(  f_{1},g_{1}\right) \\
&  \qquad\qquad+d_{M}\left(  f_{n}\circ\dots\circ f_{2},g_{n}\circ\dots\circ
g_{2}\right)
\end{align*}
and the induction step is complete.
\end{proof}

\subsection{Oriented Partitions\label{sec.11.2}}

For $s,t\in\left[  0,T\right]  ,$ let
\[
J\left(  s,t\right)  :=\left[  \min\left(  s,t\right)  ,\max\left(
s,t\right)  \right]
\]
be the interval between $s$ and $t.$

\begin{definition}
[Oriented Partitions]\label{def.11.7}If $\left(  s,t\right)  \in\left[
0,T\right]  ^{2}$ with $s\neq t,$ let $\mathcal{P}\left(  s,t\right)  $ denote
the \textbf{oriented partitions}, $\pi,$ of $J\left(  s,t\right)  $ where
$\pi\in\mathcal{P}\left(  s,t\right)  $ iff $\pi=\left(  s_{k}\right)
_{k=0}^{n}=\left(  s_{0},\dots,s_{n}\right)  $ is an ordered subset of
$J\left(  s,t\right)  $ such that%
\begin{align}
s  &  =s_{0}<s_{1}<\dots<s_{n}=t\text{ when }t>s\text{ and}\nonumber\\
t  &  =s_{n}<s_{n-1}<\dots<s_{0}=s\text{ when }s>t. \label{e.11.1}%
\end{align}
We further let $\left\{  \pi\right\}  :=\left\{  s_{k}:0\leq k\leq n\right\}
$ be the unordered points in $\pi$ and $\#\left(  \pi\right)  =n$ which is the
same as the number of connected components of $J\left(  s,t\right)
\setminus\left\{  \pi\right\}  .$ [We will often abuse notation and simply
refer to an element, $\pi\in\mathcal{P}\left(  s,t\right)  ,$ as an oriented
partition of $J\left(  s,t\right)  $ or more simply as a partition of
$J\left(  s,t\right)  .$]
\end{definition}

\begin{notation}
\label{not.11.8}For $\varepsilon\in(0,1/2]$ and $s,t\in\left[  0,T\right]  ,$
let $J_{\varepsilon}\left(  s,t\right)  $ be the \textbf{middle }%
$1-2\varepsilon$\textbf{ subinterval} of $J\left(  s,t\right)  $ defined by
\[
J_{\varepsilon}\left(  s,t\right)  :=J\left(  s+\varepsilon\left(  t-s\right)
,t-\varepsilon\left(  t-s\right)  \right)  .
\]
[If we let $J_{\varepsilon}=J_{\varepsilon}\left(  0,1\right)  =\left[
\varepsilon,1-\varepsilon\right]  ,$ then we may express $J_{\varepsilon
}\left(  s,t\right)  $ as $J_{\varepsilon}\left(  s,t\right)  :=s+\left(
t-s\right)  J_{\varepsilon}.$]\begin{figure}[ptbh]
\centering
\par
\psize{3.5in} %
\executeiffilenewer{\GraphicsDirectorymideps.svg}{\GraphicsDirectorymideps.pdf}%
{inkscape -z -D --file=\GraphicsDirectorymideps.svg --export-pdf=\GraphicsDirectorymideps.pdf --export-latex}%
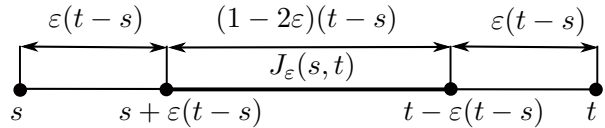%
 \caption{The middle $\left(  1-2\varepsilon
\right)  $ fraction subinterval when $s<t.$ The figure for $s>t$ is similar.}%
\label{fig.1}%
\end{figure}
\end{notation}

\begin{definition}
[$\varepsilon$-special partitions]\label{def.11.9}Let $s,t\in\left[
0,T\right]  $ and $\varepsilon\in(0,1/2]$, $\pi\in\mathcal{P}\left(
s,t\right)  $ be a partition of $J\left(  s,t\right)  .$ We say a point
$u\in\pi$ is $\varepsilon$\textbf{-special }if $u\in J_{\varepsilon}\left(
s,t\right)  \cap\pi$ and then we define the notion of $\pi$ being an
$\varepsilon$\textbf{-special} partition inductively on $n=\#\left(
\pi\right)  \geq1$ as follows.

\begin{enumerate}
\item If $n=1$ then $\pi$ is $\varepsilon$ special for any choice of
$0<\varepsilon\leq1/2.$

\item Assuming that $\varepsilon$\textbf{-special }has been defined for
partitions with $\#\left(  \pi\right)  =n\ $for some $n\geq2$ and suppose
$\pi=\left(  s_{0},s_{1},\dots,s_{n+1}\right)  \in\mathcal{P}\left(
s,t\right)  $ with $\#\left(  \pi\right)  =n+1,$ then $\pi$ is $\varepsilon
$\textbf{-special }if there exists a $1\leq p\leq n$ such that: 1)
$u:=s_{p}\in\pi\cap J_{\varepsilon}\left(  s,t\right)  ,$ 2) $\pi_{\leq
u}:=\left(  s_{0},s_{1},\dots,s_{p}\right)  $ is an $\varepsilon$-special
partition of $J\left(  s,u\right)  ,$ and 3) $\pi_{\geq u}:=\left(
s_{p},s_{p+1},\dots,s_{n+1}\right)  $ is an $\varepsilon$-special partitions
of $J\left(  u,t\right)  .$
\end{enumerate}
\end{definition}

\begin{notation}
[Uniform Partitions]\label{not.11.10}For $\left(  s,t\right)  \in\left[
0,T\right]  ^{2}$ and $n\in\mathbb{N},$ let $\pi^{n}\left(  s,t\right)
=\left(  s_{0},\dots,s_{n}\right)  \in\mathcal{P}\left(  s,t\right)  $ where%
\begin{equation}
s_{i}:=s+\frac{i}{n}\left(  t-s\right)  \text{ for }0\leq i\leq n
\label{e.11.2}%
\end{equation}
be the uniform partition of $J\left(  s,t\right)  $ with $n$ equal
subdivisions. We also let $\pi^{\left(  n\right)  }\left(  s,t\right)
:=\pi^{2^{n}}\left(  s,t\right)  =\left(  s_{0},\dots,s_{2^{n}}\right)
\in\mathcal{P}\left(  s,t\right)  $ where
\begin{equation}
s_{i}:=s+i2^{-n}\left(  t-s\right)  \text{ for }0\leq i\leq2^{n}
\label{e.11.3}%
\end{equation}
so that $\pi^{\left(  n\right)  }\left(  s,t\right)  $ is the uniform
partition of $J\left(  s,t\right)  $ with $2^{n}$-subdivisions.
\end{notation}

\begin{example}
\label{ex.11.11}For $\left(  s,t\right)  \in\left[  0,T\right]  ^{2}$ and
$n\in\mathbb{N}$ the partition $\pi^{\left(  n\right)  }\left(  s,t\right)  $
in Eq. (\ref{e.11.3}) $\varepsilon=1/2$-special. There are also many
sub-partitions of $\pi^{\left(  n\right)  }\left(  s,t\right)  $ which are
still $1/2$-special. For example if $n\geq2,$ then
\[
\pi=\left\{  s\right\}  \cup\left\{  s+\left(  t-s\right)  k2^{-n}:2^{\left(
n-1\right)  }\leq k\leq2^{n}\right\}
\]
is still $1/2$ uniform.
\end{example}

\begin{lemma}
\label{lem.11.12}For $\left(  s,t\right)  \in\left[  0,T\right]  ^{2}$ and
$n\in\mathbb{N},$ the uniform partition, $\pi^{n}=\pi^{n}\left(  s,t\right)
,$ of $J\left(  s,t\right)  $ is $\varepsilon=1/3$ -special for any
$n\in\mathbb{N}.$
\end{lemma}

\begin{proof}
The proof is by induction on $n=\#\left(  \pi\right)  .$ For $n=1,$ $\pi
^{1}=\left\{  s,t\right\}  $ and there is nothing to prove. For $n=2$ the
interior point of $\pi^{2}$ is $u=s+\left(  t-s\right)  /2$ is the mid-point
of $\pi^{2}$ and hence $\pi^{2}$ is $1/2$-special. For $n=3,$ we let
$u=s+\left(  t-s\right)  /3\in J_{1/3}\left(  s,t\right)  $ so that $u$ is
$1/3$-special point and so by $n=1$ and $n=2$ cases already discussed,
$\pi^{3}$ is $1/3$-special. For the induction step, suppose we have shown, for
some $n\in\mathbb{N},$ that every uniform partition, $\pi$ of any compact
interval $\#\left(  \pi\right)  \leq n$ is $1/3$-special. Then for $\pi
=\pi^{n+1}\left(  s,t\right)  \in\mathcal{P}\left(  s,t\right)  ,$ we let%
\[
u=\left\{
\begin{array}
[c]{ccc}%
s+\left(  t-s\right)  \frac{1}{2} & \text{if} & n\text{ is odd}\\
s+\left(  t-s\right)  \frac{n/2}{n+1} & \text{if} & n\text{ is even}%
\end{array}
\right.  \in\pi^{n+1}.
\]
So when $n$ is odd $u$ is $1/2$-special and when $n$ is even, $u$ is
$\varepsilon_{n}:=\frac{1}{2}\frac{n}{n+1}$-special. As $\frac{d}{dn}\frac
{n}{n+1}=\frac{1}{\left(  n+1\right)  ^{2}}>0$ it follows that $\varepsilon
_{n}\geq\varepsilon_{2}=1/3$ so that in all cases $u$ is at least
$1/3$-special. The result now follows by applying the induction hypothesis
applied to remaining uniform partitions of $J\left(  s,u\right)  $ and
$J\left(  u,t\right)  $ respectively.
\end{proof}

\begin{definition}
\label{def.11.13}For $\theta>1$ and $\varepsilon\in(0,1/2],$ let%
\[
\gamma\left(  \varepsilon,\theta\right)  :=\max_{x\in J_{\varepsilon}\left(
0,1\right)  }\left[  x^{\theta}+\left(  1-x\right)  ^{\theta}\right]
=\varepsilon^{\theta}+\left(  1-\varepsilon\right)  ^{\theta}<1.
\]

\end{definition}

The point of $\varepsilon$-special partitions is that in the arguments below
we will often arrive at an estimate for a quantity, $Q,$ of the form
\[
Q\leq k\left[  \left\vert t-u\right\vert ^{\theta}+\left\vert u-s\right\vert
^{\theta}\right]  =k\left[  \beta^{\theta}+\left(  1-\beta\right)  ^{\theta
}\right]  \left\vert t-s\right\vert ^{\theta}\text{ }%
\]
where $u\in\pi\setminus\left\{  s,t\right\}  $ and $\beta:=\left\vert
t-u\right\vert /\left\vert t-s\right\vert .$ If we now further assume that
$u\in\pi\cap J_{\varepsilon}\left(  s,t\right)  $ then we know $\beta\in
J_{\varepsilon}\left(  0,1\right)  $ and we will have
\[
Q\leq k\max_{\beta\in J_{\varepsilon}\left(  0,1\right)  }\left[
\beta^{\theta}+\left(  1-\beta\right)  ^{\theta}\right]  \left\vert
t-s\right\vert ^{\theta}=k\gamma\left(  \varepsilon,\theta\right)  \left\vert
t-s\right\vert ^{\theta}.
\]
The pre-factor, $\gamma\left(  \varepsilon,\theta\right)  ,$ being less than
one will play a crucial role in the arguments to follow.

\subsection{Approximate Flows\label{sec.11.3}}

\begin{definition}
[Pre-flows]\label{def.11.14}A \textbf{pre-flow }is a function, $\left(
s,t,m\right)  \rightarrow\mu_{t,s}\left(  m\right)  $ in $C\left(  \left[
0,T\right]  ^{2}\times M,M\right)  ,$ such that $\mu_{t,s}\in
\operatorname{Homeo}\left(  M\right)  ,$ $\mu_{s,t}=\mu_{t,s}^{-1},$ and
$\mu_{t,t}=Id_{M}$ for all $s,t\in\left[  0,T\right]  .$
\end{definition}

\begin{notation}
\label{not.11.15}Suppose that $\mu\in C\left(  \left[  0,T\right]  ^{2}\times
M,M\right)  ,$ $s,t\in\left[  0,T\right]  ,$ and $\pi=\left(  s_{0}%
,\dots,s_{n}\right)  \in\mathcal{P}\left(  s,t\right)  $ is an (oriented)
partition $J\left(  s,t\right)  .$ For any $0\leq l<k\leq n,$ let
\begin{equation}
\mu_{s_{k},s_{l}}^{\pi}:=\mu_{s_{k},s_{kn-1}}\circ\dots\circ\mu_{s_{l+2}%
,s_{l+1}}\circ\mu_{s_{l+1},s_{l}} \label{e.11.4}%
\end{equation}
and in particular,%
\begin{equation}
\mu_{t,s}^{\pi}=\mu_{s_{n},s_{n-1}}\circ\dots\circ\mu_{s_{2},s_{1}}\circ
\mu_{s_{1},s_{0}}. \label{e.11.5}%
\end{equation}

\end{notation}

\begin{notation}
\label{not.11.16}For $n\in\mathbb{N}$ and $\left(  s,t\right)  \in\left[
0,T\right]  ^{2}$ let
\[
\mu_{t,s}^{n}:=\mu_{t,s}^{\pi^{n}\left(  s,t\right)  }\text{ and }\mu
_{t,s}^{\left(  n\right)  }=\mu^{\pi^{\left(  n\right)  }\left(  s,t\right)  }%
\]
where $\pi^{n}\left(  s,t\right)  $ and $\pi^{\left(  n\right)  }\left(
s,t\right)  =\pi^{2^{n}}\left(  s,t\right)  $ are the uniform partition of
$J\left(  s,t\right)  $ as in Notation \ref{not.11.10}.
\end{notation}

The next lemma is a fairly direct extension of Lemma 2.4 in
\cite{Bailleul2015a}.

\begin{lemma}
[Local Trotter bounds]\label{lem.11.17}Let $\mu\in C\left(  \left[
0,T\right]  ^{2}\times M,M\right)  $ be a pre-flow as in Definition
\ref{def.11.14}. Assume there exists $\theta>1,$ $c<\infty,$ a continuous
increasing function, $k:\left[  -T,T\right]  \rightarrow\lbrack0,\infty)$ such
that $k\left(  0\right)  :=\lim_{t\rightarrow0}k\left(  t\right)  =0$ and for
all $\left(  s,t\right)  \in\left[  0,T\right]  ^{2},$
\begin{equation}
\operatorname{Lip}\left(  \mu_{t,s}\right)  \leq\left(  1+k\left(  t-s\right)
\right)  \text{ and }\sup_{u\in J\left(  s,t\right)  }d_{M}\left(  \mu
_{tu}\circ\mu_{us},\mu_{ts}\right)  \leq c\left\vert t-s\right\vert ^{\theta}.
\label{e.11.6}%
\end{equation}
Let $\varepsilon\in(0,1/2]$ be given and for $\delta\in(0,T],$ let $k^{\ast
}\left(  \delta\right)  :=\max\left(  k\left(  \delta\right)  ,k\left(
-\delta\right)  \right)  .$ If $\delta>0$ is chosen so that%
\begin{equation}
p\left(  \delta\right)  :=\gamma\left(  \theta,\varepsilon\right)  +k^{\ast
}\left(  \delta\right)  <1 \label{e.11.7}%
\end{equation}
and $L\geq\frac{c}{1-p\left(  \delta\right)  },$ then for all $s,t\in\left[
0,T\right]  $ with $\left\vert t-s\right\vert \leq\delta$ and any
$\varepsilon$-special partition $\left(  \pi\right)  $ of $J\left(
s,t\right)  ,$
\begin{equation}
d_{M}\left(  \mu_{t,s}^{\pi},\mu_{t,s}\right)  \leq L\left\vert t-s\right\vert
^{\theta}. \label{e.11.8}%
\end{equation}

\end{lemma}

\begin{proof}
The proof is by induction on $r\left(  \pi\right)  =\#\left(  \pi\right)  -1.$
When $r=0,$ $\mu_{t,s}^{\pi}=\mu_{t,s}$ and there is nothing to prove. If
$r=1,$ then $\pi=\left(  s,u,t\right)  $ if $s<t$ or $\pi=\left(
t,u,s\right)  $ if $t<s$ for some $u\in J_{\varepsilon}\left(  s,t\right)  .$
It then follows by assumption that
\[
d_{M}\left(  \mu_{t,s}^{\pi},\mu_{t,s}\right)  =d_{M}\left(  \mu_{t,u}\circ
\mu_{u,s},\mu_{t,s}\right)  \leq c\left\vert t-s\right\vert ^{\theta}%
\]
from which it follows that we are going to need to choose $L\geq c.$

Suppose that the estimate in Eq. (\ref{e.11.8}) is known to hold for all
$s,t\in\left[  0,T\right]  $ and all $\varepsilon$-special partitions $\pi
\in\mathcal{P}\left(  s,t\right)  $ with $r\left(  \pi\right)  \leq r_{0}$ for
some $r_{0}\in\mathbb{N}$ and suppose that $\pi$ is an $\varepsilon$-special
partition of $J\left(  s,t\right)  $ with $\#\left(  \pi\right)  =r_{0}+1.$ By
definition of $\pi$ being $\varepsilon$-special, there exists $u\in\pi\cap
J_{\varepsilon}\left(  s,t\right)  $ such that $\pi\cap J\left(  s,u\right)  $
and $\pi\cap J\left(  u,t\right)  $ are $\varepsilon$-special partitions of
$J\left(  s,u\right)  $ and $J\left(  u,t\right)  $ respectively. Using
$\mu_{t,s}^{\pi}=\mu_{t,u}^{\pi}\circ\mu_{u,s}^{\pi},$ the triangle
inequality, the assumptions, and the induction hypothesis it follows that
\begin{align}
d_{M}\left(  \mu_{t,s}^{\pi},\mu_{t,s}\right)  =  &  d_{M}\left(  \mu
_{t,u}^{\pi}\circ\mu_{u,s}^{\pi},\mu_{t,s}\right) \nonumber\\
\leq &  d_{M}\left(  \mu_{t,u}^{\pi}\circ\mu_{u,s}^{\pi},\mu_{t,u}\circ
\mu_{u,s}^{\pi}\right)  +d_{M}\left(  \mu_{t,u}\circ\mu_{u,s}^{\pi},\mu
_{t,u}\circ\mu_{u,s}\right) \nonumber\\
&  \qquad+d_{M}\left(  \mu_{t,u}\circ\mu_{u,s},\mu_{t,s}\right)  .\nonumber\\
\leq &  d_{M}\left(  \mu_{t,u}^{\pi},\mu_{t,u}\right)  +\operatorname{Lip}%
\left(  \mu_{t,u}\right)  d_{M}\left(  \mu_{u,s}^{\pi},\mu_{u,s}\right)
\nonumber\\
&  \qquad+d_{M}\left(  \mu_{t,u}\circ\mu_{u,s},\mu_{t,s}\right) \nonumber\\
\leq &  L\left\vert t-u\right\vert ^{\theta}+\left(  1+k\left(  t-u\right)
\right)  L\left\vert u-s\right\vert ^{\theta}+c\left\vert t-s\right\vert
^{\theta}\label{e.11.9}\\
&  =\left[  \left\vert t-u\right\vert ^{\theta}+\left\vert u-s\right\vert
^{\theta}\right]  L+k\left(  t-u\right)  \left\vert u-s\right\vert ^{\theta
}L+c\left\vert t-s\right\vert ^{\theta}\nonumber\\
&  \leq\left[  \gamma\left(  \theta,\varepsilon\right)  +k\left(  t-s\right)
\right]  L\left\vert t-s\right\vert ^{\theta}+c\left\vert t-s\right\vert
^{\theta}.\nonumber
\end{align}
We now choose $\delta>0$ so that $p\left(  \delta\right)  :=\gamma\left(
\theta,\varepsilon\right)  +k^{\ast}\left(  \delta\right)  <1$ and then
choosing $L$ sufficiently large (i.e. $L\geq\frac{c}{1-p\left(  \delta\right)
}),$ such that $p\left(  \delta\right)  L+c\leq L,$ it will follow for
$\left(  s,t\right)  \in\left[  0,T\right]  ^{2}$ with $\left\vert
t-s\right\vert <\delta,$ that%
\[
d_{M}\left(  \mu_{t,s}^{\pi},\mu_{t,s}\right)  \leq\left(  p\left(
\delta\right)  L+c\right)  \left\vert t-s\right\vert ^{\theta}\leq L\left\vert
t-s\right\vert ^{\theta}%
\]
which completes the inductive step.
\end{proof}

The next theorem extends the previous lemma by removing the restriction on $s$
and $t$ being close to one another.

\begin{theorem}
[Global product bounds]\label{thm.11.18}Let $\mu\in C\left(  \left[
0,T\right]  ^{2}\times M,M\right)  $ be as in Lemma \ref{lem.11.17}, i.e.
$\mu$ satisfies the estimates in Eq. (\ref{e.11.6}). Then for all
$\varepsilon\in(0,1/2]$ there exists $L_{\varepsilon}<\infty$ such that%
\begin{equation}
d_{M}\left(  \mu_{t,s}^{\pi},\mu_{t,s}\right)  \leq L_{\varepsilon}\left\vert
t-s\right\vert ^{\theta}\text{ for all }s,t\in\left[  0,T\right]  .
\label{e.11.10}%
\end{equation}

\end{theorem}

\begin{proof}
Let $K:=\max_{\left\vert t\right\vert \leq T}k\left(  t\right)  ,$
$\delta=\delta\left(  \varepsilon\right)  $ be as in Lemma \ref{lem.11.17},
$s,t\in\left[  0,T\right]  $ with $\alpha:=t-s\neq0,$ and suppose that $\pi$
is an $\varepsilon$-special partition of $J\left(  s,t\right)  .$ Since $\pi$
is $\varepsilon$-special, there exists $u\in J_{\varepsilon}\left(
s,t\right)  \cap\pi,$ i.e. $u\in\pi$ and $u$ lies between $s+\varepsilon
\alpha$ and $t-\varepsilon\alpha$ and therefore,%
\begin{align*}
\left\vert t-u\right\vert  &  \leq\left\vert t-\left(  s+\varepsilon
\alpha\right)  \right\vert =\left\vert \alpha\right\vert \left(
1-\varepsilon\right)  \text{ and}\\
\left\vert s-u\right\vert  &  \leq\left\vert s-\left(  t-\varepsilon
\alpha\right)  \right\vert =\left\vert \alpha\right\vert \left(
1-\varepsilon\right)  .
\end{align*}

Let us now assume that $\left\vert t-s\right\vert \left(  1-\varepsilon
\right)  \leq\delta,$ i.e.
\[
\left\vert \alpha\right\vert =\left\vert t-s\right\vert \leq\delta
_{1}:=\left(  1-\varepsilon\right)  ^{-1}\delta.
\]
Then for $u\in J_{\varepsilon}\left(  s,t\right)  \cap\pi$ as above we will
have $\left\vert t-u\right\vert \leq\delta$ and $\left\vert s-u\right\vert
\leq\delta$ and hence by the triangle inequality and Lemma \ref{lem.11.17},
\begin{align*}
d\left(  \mu_{t,s}^{\pi},\mu_{t,s}\right)   &  =d\left(  \mu_{t,u}^{\pi}%
\circ\mu_{u,s}^{\pi},\mu_{t,s}\right) \\
&  \leq d\left(  \mu_{t,u}^{\pi}\circ\mu_{u,s}^{\pi},\mu_{t,u}\circ\mu
_{u,s}\right)  +d\left(  \mu_{t,u}\circ\mu_{u,s},\mu_{t,s}\right) \\
&  \leq\operatorname{Lip}\left(  \mu_{t,u}\right)  d\left(  \mu_{u,s}^{\pi
},\mu_{u,s}\right)  +d\left(  \mu_{t,u}^{\pi},\mu_{t,u}\right)  +d\left(
\mu_{t,u}\circ\mu_{u,s},\mu_{t,s}\right) \\
&  \leq\left(  1+K\right)  L\left\vert u-s\right\vert ^{\theta}+L\left\vert
t-u\right\vert ^{\theta}+c\left\vert t-s\right\vert ^{\theta}\leq
L_{1}\left\vert t-s\right\vert ^{\theta}%
\end{align*}
where $L_{1}=\left(  2+K\right)  L+c.$

We may now repeat this same argument with $\delta$ replaced by $\delta_{1}$ to
find, for $s,t\in\left[  0,T\right]  $ with
\[
\left\vert t-s\right\vert \leq\delta_{2}:=\left(  1-\varepsilon\right)
^{-1}\delta_{2}=\left(  1-\varepsilon\right)  ^{-2}\delta,
\]
that%
\begin{align*}
d\left(  \mu_{t,s}^{\pi},\mu_{t,s}\right)   &  \leq\operatorname{Lip}\left(
\mu_{t,u}\right)  d\left(  \mu_{u,s}^{\pi},\mu_{u,s}\right)  +d\left(
\mu_{t,u}^{\pi},\mu_{t,u}\right)  +d\left(  \mu_{t,u}\circ\mu_{u,s},\mu
_{t,s}\right) \\
&  \leq\left(  1+K\right)  L_{1}\left\vert u-s\right\vert ^{\theta}%
+L_{1}\left\vert t-u\right\vert ^{\theta}+c\left\vert t-s\right\vert ^{\theta
}\leq L_{2}\left\vert t-s\right\vert ^{\theta}%
\end{align*}
where $L_{2}:=\left(  2+K\right)  L_{1}+c.$ Hence, by induction, if
$s,t\in\left[  0,T\right]  $ with $\left\vert t-s\right\vert \leq\delta
_{n}=\left(  1-\varepsilon\right)  ^{-n}\delta,$ then
\[
d\left(  \mu_{t,s}^{\pi},\mu_{t,s}\right)  \leq L_{n}\left\vert t-s\right\vert
^{\theta}%
\]
where $L_{n}$ is defined inductively by $L_{n}=\left(  2+K\right)  L_{n-1}+c$
with $L_{0}=L.$ To complete the proof we need only use this estimate with $n$
sufficiently large so that $\delta_{n}=\left(  1-\varepsilon\right)
^{-n}\delta\geq T.$
\end{proof}

\begin{lemma}
\label{lem.11.19}Suppose $\mu\in C\left(  \left[  0,T\right]  ^{2}\times
M,M\right)  $ is a pre-semi group and for some $\varepsilon\in(0,1/2],$ there
exists a $\delta>0$ and $C<\infty$ such that $\operatorname{Lip}\left(
\mu_{t,s}^{\pi}\right)  \leq C$ for all $\left(  s,t\right)  \in\left[
0,T\right]  ^{2}$ with $\left\vert t-s\right\vert \leq\delta$ and
$\varepsilon$-special partitions, $\pi,$ of $J\left(  s,t\right)  .$ Then
there exists a constant $C^{\prime}<\infty$ such that $\operatorname{Lip}%
\left(  \mu_{t,s}^{\pi}\right)  \leq C^{\prime}$ for all $\left(  s,t\right)
\in\left[  0,T\right]  ^{2}$ and all $\varepsilon$-special partitions, $\pi,$
of $J\left(  s,t\right)  ,$ i.e. we may drop the restriction that $\left\vert
t-s\right\vert \leq\delta.$
\end{lemma}

\begin{proof}
The proof is very similar to the proof of Theorem \ref{thm.11.18}. Let
$\delta_{1}=\left(  1-\varepsilon\right)  ^{-1}\delta$ and $\left(
s,t\right)  \in\left[  0,T\right]  ^{2}$ with $\left\vert t-s\right\vert
\leq\delta_{1}$ and $\pi$ be an $\varepsilon$-special partition of $J\left(
s,t\right)  .$ As in the proof of Theorem \ref{thm.11.18}, for $u\in
J_{\varepsilon}\left(  s,t\right)  \cap\pi$ (which exists as $\pi$ is
$\varepsilon$-special) we have both $\left\vert u-s\right\vert \leq\delta$ and
$\left\vert t-u\right\vert \leq\delta$ and hence%
\[
\operatorname{Lip}\left(  \mu_{t,s}^{\pi}\right)  =\operatorname{Lip}\left(
\mu_{t,u}^{\pi}\circ\mu_{u,s}^{\pi}\right)  \leq\operatorname{Lip}\left(
\mu_{t,u}^{\pi}\right)  \cdot\operatorname{Lip}\left(  \mu_{u,s}^{\pi}\right)
\leq C^{2}.
\]
We may now repeat this procedure with $\delta_{2}=\left(  1-\varepsilon
\right)  ^{-1}\delta_{1}=\left(  1-\varepsilon\right)  ^{-2}\delta$ and
$\left(  s,t\right)  \in\left[  0,T\right]  ^{2}$ with $\left\vert
t-s\right\vert \leq\delta_{2}$ and $\pi$ is an $\varepsilon$-special partition
of $J\left(  s,t\right)  $ in order to find $\operatorname{Lip}\left(
\mu_{t,s}^{\pi}\right)  \leq C^{4}.$ Continuing this way inductively, if
$\left\vert t-s\right\vert \leq\delta_{n}=\left(  1-\varepsilon\right)
^{-n}\delta$ and $\pi$ is an $\varepsilon$-special partition of $J\left(
s,t\right)  ,$ then
\[
\operatorname{Lip}\left(  \mu_{t,s}^{\pi}\right)  \leq C^{2^{n}}<\infty.
\]
It then follows that $\operatorname{Lip}\left(  \mu_{t,s}^{\pi}\right)  \leq
C^{\prime}$ for all $\left(  s,t\right)  \in\left[  0,T\right]  ^{2}$ where
$C^{\prime}:=C^{2^{n}}$ provided we choose $n\in\mathbb{N}$ so that
$\delta_{n}=\left(  1-\varepsilon\right)  ^{-n}\delta\geq T.$
\end{proof}

The next two corollaries are easy consequences of Proposition \ref{pro.11.5}.

\begin{corollary}
\label{cor.11.20}Suppose that $\mu_{t,s}$ and $\nu_{t,s}$ are two
pre-semigroups on $M,$ see Definition \ref{def.11.14}. If $\pi=\left\{
s_{k}\right\}  _{k=0}^{n}$ is a partition of $J\left(  s,t\right)  ,$ then
\[
d\left(  \mu_{t,s}^{\pi},\nu_{t,s}^{\pi}\right)  \leq\sum_{k=1}^{n}\left(
\operatorname{Lip}\left(  \mu_{t,s_{k}}^{\pi}\right)  \wedge\operatorname{Lip}%
\left(  \nu_{t,s_{k}}^{\pi}\right)  \right)  d\left(  \mu_{s_{k},s_{k-1}}%
,\nu_{s_{k},s_{k-1}}\right)  .
\]

\end{corollary}

\begin{corollary}
\label{cor.11.21}Let $\pi=\left(  s_{0},\dots,s_{n}\right)  \in\mathcal{P}%
\left(  s,t\right)  $ be a partition of $J\left(  s,t\right)  $ and to each
$1\leq k\leq n$ let $\Lambda_{k}\in\mathcal{P}\left(  s_{k-1},s_{k}\right)  $
be a partition of $J\left(  s_{k-1},s_{k}\right)  $ and let $\pi^{\ast}%
\in\mathcal{P}\left(  s,t\right)  $ be the unique oriented partition of
$J\left(  s,t\right)  $ such that $\left\{  \pi^{\ast}\right\}  =\cup
_{k=1}^{n}\left\{  \Lambda_{k}\right\}  .$ Then
\begin{equation}
d\left(  \mu_{t,s}^{\pi^{\ast}},\mu_{t,s}^{\pi}\right)  \leq\sum_{k=1}%
^{n}\operatorname{Lip}\left(  \mu_{t,s_{k}}^{\pi}\right)  d\left(  \mu
_{s_{k},s_{k-1}},\mu_{s_{k},s_{k-1}}^{\Lambda_{k}}\right)  . \label{e.11.11}%
\end{equation}

\end{corollary}

\begin{proof}
Again Eq. (\ref{e.11.11}) follows from Proposition \ref{pro.11.5} along with
the following two identities,%
\begin{align*}
\mu_{t,s}^{\pi}  &  =\mu_{s_{n},s_{n-1}}\circ\dots\circ\mu_{s_{2},s_{1}}%
\circ\mu_{s_{1},s_{0}}\text{ and }\\
\mu_{t,s}^{\pi^{\ast}}  &  =\mu_{s_{n},s_{n-1}}^{\Lambda_{n}}\circ\dots
\circ\mu_{s_{2},s_{1}}^{\Lambda_{2}}\circ\mu_{s_{1},s_{0}}^{\Lambda_{1}}.
\end{align*}

\end{proof}

\begin{corollary}
\label{cor.11.22}Let $\pi=\left(  s_{0},\dots,s_{n}\right)  \in\mathcal{P}%
\left(  s,t\right)  $ be an oriented partition of $J\left(  s,t\right)  $ and
\[
\pi^{\ast}=\left(  s_{0},s_{1}^{\ast},s_{1},s_{2}^{\ast},s_{2},\dots
,s_{n-1},s_{n}^{\ast},s_{n}\right)
\]
where $s_{k}^{\ast}$ is chosen arbitrarily to lie between $s_{k-1}$ and
$s_{k}$ for $1\leq k\leq n.$ Then
\[
d\left(  \mu_{t,s}^{\pi^{\ast}},\mu_{t,s}^{\pi}\right)  \leq\sum_{k=1}%
^{n}\operatorname{Lip}\left(  \mu_{t,s_{k}}^{\pi}\right)  d\left(  \mu
_{s_{k},s_{k-1}},\mu_{s_{k},s_{k}^{\ast}}\circ\mu_{s_{k}^{\ast},s_{k-1}%
}\right)  .
\]

\end{corollary}

\begin{proof}
This is a very special case of Corollary \ref{cor.11.21} where $\Lambda
_{k}:=\left(  s_{k-1},s_{k}^{\ast},s_{k}\right)  $ for $1\leq k\leq n.$
\end{proof}

\begin{corollary}
\label{cor.11.23}Suppose $\mu\in C\left(  \left[  0,T\right]  ^{2}\times
M,M\right)  $ is a pre-flow and there exists $C<\infty$ such that
$\operatorname{Lip}\left(  \mu_{t,s}^{\pi}\right)  \leq C$ and and $L<\infty$
such that Eq. (\ref{e.11.10}) holds for all $\left(  s,t\right)  \in\left[
0,T\right]  ^{2}$ and all $\frac{1}{3}$-special partitions of $J\left(
s,t\right)  .$ Let $n\in\mathbb{N}$ and $\pi_{n}:=\pi^{n}\left(  s,t\right)  $
be an oriented uniform partition of $J\left(  s,t\right)  $ and $\pi_{n}%
^{\ast}:=\cup_{k=1}^{n}\Lambda_{k}$ where for each $1\leq k\leq n,$
$\Lambda_{k}=\pi^{n_{k}}\left(  s_{k-1},s_{k}\right)  \in\mathcal{P}\left(
s_{k-1},s_{k}\right)  $ is a uniform partition of $J\left(  s_{k-1}%
,s_{k}\right)  $ with any number of subdivisions, $n_{k},$ which can depend on
$k.$ Then%
\begin{equation}
d\left(  \mu_{t,s}^{\pi_{n}^{\ast}},\mu_{t,s}^{n}\right)  =d\left(  \mu
_{t,s}^{\pi_{n}^{\ast}},\mu_{t,s}^{\pi_{n}}\right)  \leq LC\left\vert
t-s\right\vert ^{\theta}\left(  \frac{1}{n}\right)  ^{\theta-1}
\label{e.11.12}%
\end{equation}

\end{corollary}

\begin{proof}
By Corollary \ref{cor.11.21} along with Eq. (\ref{e.11.10}) and the assumption
that $\sup_{s,t\in\left[  0,1\right]  }\operatorname{Lip}\left(  \mu
_{t,s}^{\pi_{n}}\right)  \leq C,$ we find
\begin{align*}
d\left(  \mu_{t,s}^{\pi_{n}^{\ast}},\mu_{t,s}^{\pi_{n}}\right)   &  \leq
\sum_{k=1}^{n}\operatorname{Lip}\left(  \mu_{t,s_{k}}^{\pi_{n}}\right)
d\left(  \mu_{s_{k},s_{k-1}},\mu_{s_{k},s_{k-1}}^{\Lambda_{k}}\right) \\
&  \leq\sum_{k=1}^{n}\operatorname{Lip}\left(  \mu_{t,s_{k}}^{\pi_{n}}\right)
L\left\vert s_{k}-s_{k-1}\right\vert ^{\theta}\\
&  \leq LC\sum_{k=1}^{n}\left\vert \frac{t-s}{n}\right\vert ^{\theta
}=LC\left\vert t-s\right\vert ^{\theta}\left(  \frac{1}{n}\right)  ^{\theta
-1}.
\end{align*}

\end{proof}

\begin{corollary}
\label{cor.11.24}If $\mu\in C\left(  \left[  0,T\right]  ^{2}\times
M,M\right)  $ is a pre-flow satisfying the assumptions in Corollary
\ref{cor.11.23} and (as usual) $\mu_{t,s}^{\left(  n\right)  }:=\mu
_{t,s}^{2^{n}},$ then
\begin{equation}
\lim_{n,k\rightarrow\infty}\sup_{s,t\in\left[  0,T\right]  }d\left(  \mu
_{t,s}^{\left(  n\right)  },\mu_{t,s}^{\left(  k\right)  }\right)  =0,
\label{e.11.13}%
\end{equation}
i.e. $\left\{  \mu_{t,s}^{\left(  n\right)  }\left(  m\right)  \right\}
_{n=1}^{\infty}$ is Cauchy uniformly in $\left(  t,s,m\right)  \in\left[
0,T\right]  ^{2}\times M.$
\end{corollary}

\begin{proof}
By applying Eq. (\ref{e.11.12}) with $n$ replaced by $2^{n}$ and $\pi_{2^{n}%
}^{\ast}:=\pi^{2^{n+1}}\left(  s,t\right)  \in\mathcal{P}\left(  s,t\right)  $
allows us to conclude,%
\[
\sup_{s,t\in\left[  0,T\right]  }d\left(  \mu_{t,s}^{\left(  n\right)  }%
,\mu_{t,s}^{\left(  n+1\right)  }\right)  \leq LCT^{\theta}\left(  \frac
{1}{2^{n}}\right)  ^{\theta-1}.
\]
As $\theta>1,$ the right member of this equation is summable over
$n\in\mathbb{N}$ which is sufficient to prove Eq. (\ref{e.11.13}).
\end{proof}

We now come to the main result of this section which can be viewed as a
generalization of \cite[Theorem 2.1]{Bailleul2015a} to complete metric spaces.
The history of such results in the context of rough paths starts with Lyon's
\textquotedblleft almost multiplicative function\textquotedblright\ theorem
(see \cite{Lyons1994} and \cite[Theorem 3.2.1, p. 41]{Lyons2002}). Successive
generalizations / alternative formulations of Lyon's result may be found in
Gubinelli'\cite{Gubinelli2004}, Feyel and de la Pradelle \cite{Feyel2006},
Feyel, de la Pradelle and Mokobodski \cite{Feyel2008}, and likely many other references.

\begin{theorem}
[AMF theorem -- Existence of $\varphi$]\label{thm.11.25}Suppose, as in
Corollary \ref{cor.11.23}, that $\mu\in C\left(  \left[  0,T\right]
^{2}\times M,M\right)  $ is a pre-flow and there exists $C<\infty$ such that
$\operatorname{Lip}\left(  \mu_{t,s}^{\pi}\right)  \leq C$ and $L<\infty$ such
that Eq. (\ref{e.11.10}) holds for all $\left(  s,t\right)  \in\left[
0,T\right]  ^{2}$ and all $\frac{1}{3}$-special partitions of $J\left(
s,t\right)  .$ If we further assume that $\left(  M,d\right)  $ is a complete
metric space, then there exists a unique $\varphi\in C\left(  \left[
0,T\right]  ^{2}\times M,M\right)  $ such that;

\begin{enumerate}
\item $\varphi_{t,t}=Id_{M}$ for all $t\in\left[  0,T\right]  ,$

\item $\varphi_{t,s}\circ\varphi_{s,r}=\varphi_{t,r}$ for all $s,t,r\in\left[
0,T\right]  ,$ and

\item $\varphi$ is close to $\mu$, i.e.%
\begin{equation}
d\left(  \varphi_{t,s},\mu_{t,s}\right)  \leq L\left\vert t-s\right\vert
^{\theta}\text{ for all }s,t\in\left[  0,T\right]  . \label{e.11.14}%
\end{equation}
Note that items 1. and 2. above implies that $\varphi_{t,s}\in
\operatorname{Homeo}\left(  M\right)  $ for all $s,t\in\left[  0,T\right]  .$
It is also true that $\varphi_{t,s}:M\rightarrow M$ is Lipschitz with%
\begin{equation}
\operatorname{Lip}\left(  \varphi_{t,s}\right)  \leq C\text{ for all }%
s,t\in\left[  0,T\right]  . \label{e.11.15}%
\end{equation}

\end{enumerate}
\end{theorem}

\begin{proof}
By Corollary \ref{cor.11.24} and the completeness of $M,$ there exists
$\varphi\in C\left(  \left[  0,T\right]  ^{2}\times M,M\right)  $ such that
\[
\lim_{n\rightarrow\infty}\sup_{s,t\in\left[  0,T\right]  }d\left(  \mu
_{t,s}^{\left(  n\right)  },\varphi_{t,s}\right)  =0.
\]
From Eq. (\ref{e.11.10}) we have
\[
d_{M}\left(  \mu_{t,s}^{\left(  n\right)  },\mu_{t,s}\right)  \leq L\left\vert
t-s\right\vert ^{\theta}.
\]
Passing to the limit as $n\rightarrow\infty$ in this inequality then gives Eq.
(\ref{e.11.14}). Similarly, if $x,y\in M,$ then%
\[
d\left(  \varphi_{t,s}\left(  y\right)  ,\varphi_{t,s}\left(  x\right)
\right)  =\lim_{n\rightarrow\infty}d\left(  \mu_{t,s}^{\left(  n\right)
}\left(  y\right)  ,\mu_{t,s}^{\left(  n\right)  }\left(  x\right)  \right)
\leq\liminf_{n\rightarrow\infty}\operatorname{Lip}\left(  \mu_{t,s}^{\left(
n\right)  }\right)  \leq C
\]
which gives Eq. (\ref{e.11.15}). To finish the proof we must show $\varphi$ is
multiplicative and $\varphi$ is unique which we do in four parts.

(\textbf{1) }$\varphi_{t,s}=\varphi_{t,u}\circ\varphi_{us}$ when $u$
is\textbf{ between }$s$\textbf{ }and $t.$ Let $s,t\in\left[  0,T\right]  $ and
suppose that $u=s_{a}:=s+\left(  t-s\right)  a2^{-p}$ for some $p\in
\mathbb{N}$ and $0<a<2^{p}.$ Let $b=2^{p}-a,$ $n\in\mathbb{N}$ and consider
the $2^{-n}$ uniform partitions
\begin{align*}
\pi_{<} &  =\pi^{2^{n}}\left(  s,u\right)  =\left(  s_{<}\left(  k\right)
\right)  _{k=0}^{n}\in\mathcal{P}\left(  s,u\right)  \text{ and}\\
\pi_{>} &  =\pi^{2^{n}}\left(  u,t\right)  =\left(  s_{>}\left(  k\right)
\right)  _{k=0}^{n}\in\mathcal{P}\left(  u,t\right)
\end{align*}
where
\begin{align*}
s_{<}\left(  k\right)   &  =s+\left(  u-s\right)  \frac{k}{2^{n}}=s+\left(
t-s\right)  \frac{ak}{2^{p+n}}\text{ and}\\
s_{>}\left(  k\right)   &  =u+\left(  t-u\right)  \frac{k}{2^{n}}=u+\left(
t-s\right)  \frac{bk}{2^{p+n}}.
\end{align*}
We further divide each subinterval of $\pi_{<}$ $\left(  \pi_{>}\right)  $
into $a$ $\left(  b\right)  $ equal pieces, that is for each $1\leq k\leq
2^{n},$ let
\begin{align*}
\Lambda_{k} &  :=\left(  s+\left(  t-s\right)  \frac{a\left(  k-1\right)
+j}{2^{p+n}}\right)  _{j=0}^{a}\in\mathcal{P}\left(  s_{<}\left(  k-1\right)
,s_{<}\left(  k\right)  \right)  \text{ and }\\
\Lambda_{k}^{\prime} &  :=\left(  u+\left(  t-s\right)  \frac{b\left(
k-1\right)  +j}{2^{p+n}}\right)  _{j=0}^{b}\in\mathcal{P}\left(  s_{>}\left(
k-1\right)  ,s_{>}\left(  k\right)  \right)  .
\end{align*}
If $\pi_{>}^{\ast}$ and $\pi_{<}^{\ast}$ are constructed as in Corollary
\ref{cor.11.23},\footnote{Another way to view this construction is to start
with the uniform partition, $\pi:=\pi^{2^{p}}\left(  s,t\right)  ,$ which is
the \textquotedblleft join\textquotedblright\ of the uniform partitions
$\pi_{<}=\pi^{a}\left(  s,u\right)  \in\mathcal{P}\left(  s,u\right)  $ and
$\pi_{>}=\pi^{b}\left(  u,t\right)  \in\mathcal{P}\left(  u,t\right)  .$ We
further subdivide each partition interval of $\pi_{<}$ and $\pi_{>}$ into
$2^{n}$ pieces in order to construct the uniform partitions, $\pi_{<}^{\ast
}=\pi^{a2^{n}}\left(  s,u\right)  $ and $\pi_{>}^{\ast}=\pi^{b2^{n}}\left(
u,t\right)  .$ The join of these two partition is then $\pi^{\ast}=\pi
^{2^{p}2^{n}}\left(  s,t\right)  =\pi^{2^{p+n}}\left(  s,t\right)  .$ We may
also commute the order of these construction and realize $\pi_{>}^{\ast}$ and
$\pi_{<}^{\ast}$ as refinements of $\pi^{2^{n}}\left(  s,u\right)  $ and
$\pi^{2^{n}}\left(  u,t\right)  $ where by we split each partition interval of
$\pi^{2^{n}}\left(  s,u\right)  $ ($\pi^{2^{n}}\left(  u,t\right)  )$ into $a$
$\left(  b=2^{p}-a\right)  $ oriented uniform partitions. Note that size of
each of these subdivisions are $\left(  t-s\right)  a2^{-p}/2^{n}/a=\left(
t-s\right)  2^{-\left(  p+n\right)  }$ and $\left(  t-s\right)  b2^{-p}%
/2^{n}/b=\left(  t-s\right)  2^{-\left(  p+n\right)  }$ respectively.} then
according to Corollary \ref{cor.11.23},%
\[
d\left(  \mu_{t,\mu}^{\pi_{>}^{\ast}},\mu_{t,u}^{\pi_{>}}\right)  \leq
CL\left(  \frac{1}{2^{n}}\right)  ^{\theta-1}\left\vert t-u\right\vert
^{\theta}%
\]
and
\[
d\left(  \mu_{u,s}^{\pi_{<}^{\ast}},\mu_{u,s}^{\pi_{<}}\right)  \leq CL\left(
\frac{1}{2^{n}}\right)  ^{\theta-1}\left\vert u-s\right\vert ^{\theta}.
\]
Upon noting that $\mu_{t,\mu}^{\pi_{>}^{\ast}}\circ\mu_{u,s}^{\pi_{<}^{\ast}%
}=\mu_{t,s}^{\left(  n+p\right)  },$ $\mu_{t,\mu}^{\pi_{>}}=\mu_{t,u}^{\left(
n\right)  },$ and $\mu_{u,s}^{\pi_{<}}=\mu_{u,s}^{\left(  n\right)  }$ we
find,
\begin{align*}
d\left(  \mu_{t,s}^{\left(  n+p\right)  },\mu_{t,u}^{\left(  n\right)  }%
\circ\mu_{u,s}^{\left(  n\right)  }\right)  = &  d\left(  \mu_{t,\mu}^{\pi
_{>}^{\ast}}\circ\mu_{u,s}^{\pi_{<}^{\ast}},\mu_{t,\mu}^{\pi_{>}}\circ
\mu_{u,s}^{\pi_{<}}\right)  \\
\leq &  \operatorname{Lip}\left(  \mu_{t,\mu}^{\pi_{>}}\right)  d\left(
\mu_{u,s}^{\pi_{<}^{\ast}},\mu_{u,s}^{\pi_{<}}\right)  +d\left(  \mu_{t,\mu
}^{\pi_{>}^{\ast}},\mu_{t,\mu}^{\pi_{>}}\right)  \\
\leq &  C^{2}L\left(  \frac{1}{2^{n}}\right)  ^{\theta-1}\left\vert
u-s\right\vert ^{\theta}+CL\left(  \frac{1}{2^{n}}\right)  ^{\theta
-1}\left\vert t-u\right\vert ^{\theta}\\
&  \longrightarrow0\text{ as }n\rightarrow\infty.
\end{align*}
Since also, $\lim_{n\rightarrow\infty}d\left(  \mu_{t,s}^{\left(  n+p\right)
},\varphi_{t,s}\right)  =0,$ and
\begin{align*}
d\left(  \mu_{t,u}^{\left(  n\right)  }\circ\mu_{u,s}^{\left(  n\right)
},\varphi_{t,u}\circ\varphi_{us}\right)   &  \leq\operatorname{Lip}\left(
\mu_{t,u}^{\left(  n\right)  }\right)  d\left(  \mu_{u,s}^{\left(  n\right)
},\varphi_{us}\right)  +d\left(  \mu_{t,u}^{\left(  n\right)  },\varphi
_{t,u}\right)  \\
&  \leq Cd\left(  \mu_{u,s}^{\left(  n\right)  },\varphi_{us}\right)
+d\left(  \mu_{t,u}^{\left(  n\right)  },\varphi_{t,u}\right)  \rightarrow
0\text{ as }n\rightarrow\infty,
\end{align*}
we conclude that
\begin{align*}
d\left(  \varphi_{t,s},\varphi_{t,u}\circ\varphi_{us}\right)   &  \leq
d\left(  \mu_{t,s}^{\left(  n+p\right)  },\varphi_{t,s}\right)  +d\left(
\mu_{t,s}^{\left(  n+p\right)  },\mu_{t,u}^{\left(  n\right)  }\circ\mu
_{u,s}^{\left(  n\right)  }\right)  \\
&  \quad+d\left(  \mu_{t,u}^{\left(  n\right)  }\circ\mu_{u,s}^{\left(
n\right)  },\varphi_{t,u}\circ\varphi_{us}\right)  \longrightarrow0\text{ as
}n\rightarrow\infty,
\end{align*}
i.e. $\varphi_{t,s}=\varphi_{t,u}\circ\varphi_{us}.$ As this identity holds
for a dense set of $u$ between $s$ and $t$ and since $\varphi$ is continuous
in all of its variables, we conclude that $\varphi_{t,s}=\varphi_{t,u}%
\circ\varphi_{us}$ for all $u$ between $s$ and $t.$%

(\textbf{2) }$\varphi_{t,s}\in\operatorname{Homeo}\left(  M\right)  $\textbf{
and }$\varphi_{t,s}^{-1}=\varphi_{s,t}$\textbf{ for all }$s,t\in\left[
0,T\right]  .$ For each $n\in\mathbb{N}$ it is easy to verify that
\[
\mu_{t,s}^{2^{n}}\circ\mu_{s,t}^{2^{n}}=Id_{M}=\mu_{s,t}^{2^{n}}\circ\mu
_{t,s}^{2^{n}}.
\]
Passing to the limit as $n\rightarrow\infty$ in this identity then shows
\[
\varphi_{t,s}\circ\varphi_{s,t}=Id_{M}=\varphi_{s,t}\circ\varphi_{t,s}%
\]
which completes the proof of this step.

(\textbf{3) }$\varphi$\textbf{ is multiplicative. }Let $\alpha_{t}%
:=\varphi_{0,t}\in\operatorname{Homeo}\left(  M\right)  $ for all $t\in\left[
0,T\right]  .$ Making use of items 1. and 2., if $0\leq s\leq t\leq\not T  ,$
then $\alpha_{t}=\varphi_{t,s}\circ\alpha_{s}$ and hence $\varphi_{t,s}%
=\alpha_{t}\circ\alpha_{s}^{-1}.$ By interchanging the roles of $s$ and $t,$
if $0\leq t\leq s\leq T,$ then $\varphi_{s,t}=\alpha_{s}\circ\alpha_{t}^{-1}.$
Taking the inverse of this equation shows (again using item 2.) that
$\varphi_{t,s}=\varphi_{s,t}^{-1}=\alpha_{t}\circ\alpha_{s}^{-1}.$ Thus for
all $s,t\in\left[  0,T\right]  $ we have $\varphi_{t,s}=\alpha_{t}\circ
\alpha_{s}^{-1}$ and from this identity, if $r,s,t$ are arbitrary points in
$\left[  0,T\right]  ,$ then
\[
\varphi_{t,s}\circ\varphi_{s,r}=\alpha_{t}\circ\alpha_{s}^{-1}\circ\alpha
_{s}\circ\alpha_{r}^{-1}=\alpha_{t}\circ\alpha_{r}^{-1}=\varphi_{t,r}.
\]

(\textbf{4) Uniqueness of }$\varphi.$ Suppose that $\psi\in C\left(  \left[
0,T\right]  ^{2}\times M,M\right)  $ also satisfies items 1.--3. in the
statement of the theorem. Then for $s,t\in\left[  0,T\right]  $ and $u$
between $s,t$ we have%
\begin{align}
d\left(  \psi_{t,s},\varphi_{t,s}\right)   &  =d\left(  \psi_{t,u}\circ
\psi_{u,s},\varphi_{t,u}\circ\varphi_{u,s}\right) \nonumber\\
&  \leq d\left(  \psi_{t,u}\circ\psi_{u,s},\varphi_{t,u}\circ\psi
_{u,s}\right)  +d\left(  \varphi_{t,u}\circ\psi_{u,s},\varphi_{t,u}%
\circ\varphi_{u,s}\right) \nonumber\\
&  \leq d\left(  \psi_{t,u},\varphi_{t,u}\right)  +\operatorname{Lip}\left(
\varphi_{t,u}\right)  d\left(  \psi_{u,s},\varphi_{u,s}\right) \nonumber\\
&  \leq d\left(  \psi_{t,u},\varphi_{t,u}\right)  +2LC\left\vert
s-u\right\vert ^{\theta}. \label{e.11.16}%
\end{align}
We now let $\pi_{n}=\pi^{n}\left(  s,t\right)  =\left(  s_{k}\right)
_{k=0}^{n}\in\mathcal{P}\left(  s,t\right)  $ be the oriented uniform
partition of $J\left(  s,t\right)  .$ Taking $u=s_{1}$ in Eq. (\ref{e.11.16})
shows,
\[
d\left(  \psi_{t,s},\varphi_{t,s}\right)  \leq d\left(  \psi_{t,s_{1}}%
,\varphi_{t,s_{1}}\right)  +2LC\left\vert \frac{t-s}{n}\right\vert ^{\theta}.
\]
Iterating this procedure then shows,
\begin{align*}
d\left(  \psi_{t,s},\varphi_{t,s}\right)   &  \leq d\left(  \psi_{t,s_{2}%
},\varphi_{t,s_{2}}\right)  +2LC\left\vert \frac{t-s}{n}\right\vert ^{\theta
}+2LC\left\vert \frac{t-s}{n}\right\vert ^{\theta}\\
&  \vdots\\
&  \leq n\cdot2LC\left\vert \frac{t-s}{n}\right\vert ^{\theta}=2LC\left\vert
t-s\right\vert ^{\theta}\cdot\frac{1}{n^{\theta-1}}\rightarrow0\text{ as
}n\rightarrow\infty
\end{align*}
which shows $\psi_{t,s}=\varphi_{t,s}.$
\end{proof}

We may improve on the uniqueness proof of $\varphi$ in order to give a
continuity statement of $\varphi$ relative to $\mu.$

\begin{theorem}
[Continuity theorem]\label{thm.11.26}Suppose that $\mu_{t,s},~\nu_{t,s}%
\in\operatorname{Homeo}\left(  M\right)  $ are two approximate flows
satisfying the assumptions of Theorem \ref{thm.11.25} and $\varphi_{t,s}%
,~\psi_{t,s}\in\operatorname{Homeo}\left(  M\right)  $ are the unique
multiplicative functions such that%
\[
d\left(  \varphi_{t,s},\mu_{t,s}\right)  \leq L\left\vert t-s\right\vert
^{\theta}\text{ and }d\left(  \psi_{t,s},\nu_{t,s}\right)  \leq L\left\vert
t-s\right\vert ^{\theta}\text{ for all }s,t\in\left[  0,T\right]  .
\]
If there exists $\varepsilon>0$ and $0<\alpha<1$ such that
\[
d\left(  \mu_{t,s},\nu_{t,s}\right)  \leq\varepsilon\left\vert t-s\right\vert
^{\alpha}\text{ for all }s,t\in\left[  0,T\right]  ,
\]
then
\begin{equation}
d\left(  \varphi_{t,s},\psi_{t,s}\right)  \leq\left\{
\begin{array}
[c]{ccc}%
3\varepsilon\cdot\left\vert t-s\right\vert ^{\alpha} & \text{if} & \left\vert
t-s\right\vert \leq\left(  \frac{\varepsilon}{L}\right)  ^{1/\beta}\\
6CL^{\frac{1-\alpha}{\beta}}\varepsilon^{\frac{\theta-1}{\theta-\alpha}%
}\left\vert t-s\right\vert  & \text{if} & \left\vert t-s\right\vert >\left(
\frac{\varepsilon}{L}\right)  ^{1/\beta}.
\end{array}
\right.  \label{e.11.17}%
\end{equation}
If we further assume that $0\leq\varepsilon<1,$ then
\begin{equation}
d\left(  \varphi_{t,s},\psi_{t,s}\right)  \leq\max\left(  6CL^{\frac{1-\alpha
}{\beta}}T^{1-\alpha},3\right)  \varepsilon^{\frac{\theta-1}{\theta-\alpha}%
}\left\vert t-s\right\vert ^{\alpha}\text{ }\forall~s,t\in\left[  0,T\right]
\label{e.11.18}%
\end{equation}
and in the special case where $\beta=1$ (i.e. $\theta=1+\alpha),$ this
inequality becomes,%
\begin{equation}
d\left(  \varphi_{t,s},\psi_{t,s}\right)  \leq\max\left(  6C\left(  LT\right)
^{1-\alpha},3\right)  \varepsilon^{\alpha}\left\vert t-s\right\vert ^{\alpha
}~\forall~s,t\in\left[  0,T\right]  . \label{e.11.19}%
\end{equation}

\end{theorem}

\begin{proof}
Let $\beta:=\theta-\alpha>0.$ From the triangle inequality and the given
estimates it follows that%
\begin{align*}
d\left(  \varphi_{t,s},\psi_{t,s}\right)   &  \leq d\left(  \varphi_{t,s}%
,\mu_{t,s}\right)  +d\left(  \mu_{t,s},\nu_{t,s}\right)  +\text{ }d\left(
\nu_{t,s},\psi_{t,s}\right) \\
&  \leq2L\left\vert t-s\right\vert ^{\theta}+\varepsilon\left\vert
t-s\right\vert ^{\alpha}=\left[  2L\left\vert t-s\right\vert ^{\beta
}+\varepsilon\right]  \left\vert t-s\right\vert ^{\alpha}.
\end{align*}
We now define $\delta>0$ so that $L\delta^{\beta}=\varepsilon,$ i.e.
$\delta=\left(  \varepsilon/L\right)  ^{1/\beta}.$ Then from the above
inequality,%
\begin{equation}
d\left(  \varphi_{t,s},\psi_{t,s}\right)  \leq3\varepsilon\cdot\left\vert
t-s\right\vert ^{\alpha}\text{ when }\left\vert t-s\right\vert \leq
\delta=\left(  \frac{\varepsilon}{L}\right)  ^{1/\beta}. \label{e.11.20}%
\end{equation}
If $0\leq s<t\leq T$ with $\left\vert t-s\right\vert >\delta,$ choose
$n\in\mathbb{N}$ and $0\leq r<\delta$ so that $\left\vert t-s\right\vert
=n\delta+r$ and then let $s_{k}=s+k\delta$ for $0\leq k\leq n$ and
$s_{n+1}=t=s_{n}+r.$ By the multiplicative property of $\varphi$ and $\psi$
and Corollary \ref{cor.11.20}, it follows that%
\begin{align*}
d\left(  \varphi_{t,s},\psi_{t,s}\right)   &  =d\left(  \varphi_{s_{n+1}%
,s_{n}}\circ\dots\varphi_{s_{1},s_{0}},\psi_{s_{n+1},s_{n}}\circ\dots
\psi_{s_{1},s_{0}}\right) \\
&  \leq C\sum_{k=1}^{n+1}\operatorname{Lip}\left(  \varphi_{s,s_{k}}\right)
d\left(  \varphi_{s_{k},s_{k-1}},\psi_{s_{k},s_{k-1}}\right)  \leq
3C\varepsilon\left[  n\cdot\delta^{\alpha}+r^{\alpha}\right] \\
&  \leq3C\left(  n+1\right)  \varepsilon\delta^{\alpha}.
\end{align*}
Since
\[
1\leq n=\frac{\left\vert t-s\right\vert -r}{\delta}\leq\frac{\left\vert
t-s\right\vert }{\delta}%
\]
we find,%
\begin{align}
d\left(  \varphi_{t,s},\psi_{t,s}\right)   &  \leq3C\left(  \frac{\left\vert
t-s\right\vert }{\delta}+1\right)  \varepsilon\delta^{\alpha}=3C\left(
\left\vert t-s\right\vert +\delta\right)  \varepsilon\delta^{\alpha
-1}\nonumber\\
&  \leq6CL^{\frac{1-\alpha}{\beta}}\left\vert t-s\right\vert \varepsilon
^{\frac{\alpha-1}{\beta}+1}=6CL^{\frac{1-\alpha}{\beta}}\left\vert
t-s\right\vert \varepsilon^{\frac{\alpha+\beta-1}{\beta}}\nonumber\\
&  =6CL^{\frac{1-\alpha}{\beta}}\left\vert t-s\right\vert \varepsilon^{\left(
\theta-1\right)  /\beta}=6CL^{\frac{1-\alpha}{\beta}}\left\vert t-s\right\vert
\varepsilon^{\frac{\theta-1}{\theta-\alpha}} \label{e.11.21}%
\end{align}
which along with Eq. (\ref{e.11.20}) completes the proof of Eq. (\ref{e.11.17}%
). Using
\[
\left\vert t-s\right\vert =\left\vert t-s\right\vert ^{1-\alpha}\left\vert
t-s\right\vert ^{\alpha}\leq T^{1-\alpha}\left\vert t-s\right\vert ^{\alpha},
\]
in Eq. (\ref{e.11.21}) implies,%
\begin{equation}
d\left(  \varphi_{t,s},\psi_{t,s}\right)  \leq6CL^{\frac{1-\alpha}{\beta}%
}T^{1-\alpha}\varepsilon^{\frac{\theta-1}{\theta-\alpha}}\left\vert
t-s\right\vert ^{\alpha}\text{ when }\left\vert t-s\right\vert >\delta.
\label{e.11.22}%
\end{equation}
Since $\left(  \theta-1\right)  /\left(  \theta-\alpha\right)  <1,$ if
$0\leq\varepsilon<1,$ then $\varepsilon\leq\varepsilon^{\left(  \theta
-1\right)  /\left(  \theta-\alpha\right)  }$ and so the estimates in Eq.
(\ref{e.11.18}) follows directly from Eqs. (\ref{e.11.20}) and (\ref{e.11.22}).
\end{proof}

\begin{notation}
[Mesh size]\label{not.11.27}If $s,t\in\left[  0,T\right]  $ and $\pi=\left(
s_{0},\dots,s_{n}\right)  \in\mathcal{P}\left(  s,t\right)  ,$ let $\left\vert
\pi\right\vert $ be the \textbf{mesh size }of $\pi$ defined by
\begin{equation}
\left\vert \pi\right\vert :=\max_{1\leq k\leq n}\left\vert s_{k}%
-s_{k-1}\right\vert . \label{e.11.23}%
\end{equation}

\end{notation}

\begin{proposition}
\label{pro.11.28}Continuing the assumptions and notation in Theorem
\ref{thm.11.25}, then for $s,t\in\left[  0,T\right]  $%
\[
\varphi_{t,s}=\lim_{\pi\in\mathcal{P}\left(  s,t\right)  :\left\vert
\pi\right\vert \rightarrow0}\mu_{t,s}^{\pi}\text{ uniformly in }\left(
s,t\right)  \in\left[  0,T\right]  ^{2}.
\]
More precisely, if $s,t\in\left[  0,T\right]  $ and $\pi\in\mathcal{P}\left(
s,t\right)  ,$ then
\begin{equation}
d\left(  \varphi_{t,s},\mu_{t,s}^{\pi}\right)  \leq CL\left\vert
\pi\right\vert ^{\theta-1}\left\vert t-s\right\vert . \label{e.11.24}%
\end{equation}

\end{proposition}

\begin{proof}
Let $\pi=\left(  s_{0},\dots,s_{n}\right)  \in\mathcal{P}\left(  s,t\right)  $
be an oriented partition of $J\left(  s,t\right)  .$ By Corollary
\ref{cor.11.20}, the estimates in Eqs. (\ref{e.11.14}) and (\ref{e.11.15}),
and the fact that $\varphi_{t,s_{k}}^{\pi}=\varphi_{t,s_{k}}$ we find,%
\begin{align*}
d\left(  \mu_{t,s}^{\pi},\varphi_{ts}\right)   &  =d\left(  \mu_{t,s}^{\pi
},\varphi_{ts}^{\pi}\right)  \leq\sum_{k=1}^{n}\operatorname{Lip}\left(
\varphi_{t,s_{k}}^{\pi}\right)  d\left(  \mu_{s_{k},s_{k-1}},\varphi
_{s_{k},s_{k-1}}\right) \\
&  =\sum_{k=1}^{n}\operatorname{Lip}\left(  \varphi_{t,s_{k}}\right)  d\left(
\mu_{s_{k},s_{k-1}},\varphi_{s_{k},s_{k-1}}\right)  \leq C\sum_{k=1}%
^{n}L\left\vert s_{k}-s_{k-1}\right\vert ^{\theta}\\
&  \leq CL\left\vert \pi\right\vert ^{\theta-1}\sum_{k=1}^{n}\left\vert
s_{k}-s_{k-1}\right\vert =CL\left\vert \pi\right\vert ^{\theta-1}\left\vert
t-s\right\vert
\end{align*}
which proves Eq. (\ref{e.11.24}).
\end{proof}

In order to apply the above results in our context we will need to construct
appropriate approximate flows for our rough ordinary differential equations
and we will need to verify the hypothesis above hold for these approximate
flows. The trickiest point is the verification of the hypothesis that
$\operatorname{Lip}\left(  \mu_{t,s}^{\pi}\right)  \leq C\left(
\varepsilon\right)  $ when $\pi$ is an $\varepsilon$-special partition of
$J\left(  s,t\right)  .$ In the next section, we pause to recall some results
from \cite{Driver2018} that are needed in the sequel. In the following Section
\ref{sec.13}, we will give the proof of the main theorem (Theorem
\ref{thm.10.27}) of this paper.

\section{Estimates from \cite{Driver2018}\label{sec.12}}

\begin{notation}
\label{not.12.1}Given two functions, $f\left(  x\right)  $ and $g\left(
x\right)  ,$ depending on some parameters indicated by $x,$ we write $f\left(
x\right)  \lesssim g\left(  x\right)  $ if there exists a constant, $C\left(
\kappa\right)  ,$ only possibly depending on $\kappa$ so that $f\left(
x\right)  \leq C\left(  \kappa\right)  g\left(  x\right)  $ for the allowed
values of $x.$ Similarly we write $f\left(  x\right)  \asymp g\left(
x\right)  $ if both $f\left(  x\right)  \lesssim g\left(  x\right)  $ and
$g\left(  x\right)  \lesssim f\left(  x\right)  $ hold.
\end{notation}

\begin{notation}
\label{not.12.2}For $\lambda\geq0$ and $m,n\in\mathbb{N}$ with $m<n,$ let
\begin{align*}
Q_{[m,n]}\left(  \lambda\right)   &  :=\max\left\{  \lambda^{k}:k\in
\mathbb{N}\cap\left[  m,n\right]  \right\}  =\max\left\{  \lambda^{m}%
,\lambda^{n}\right\}  \text{ and}\\
Q_{(m,n]}\left(  \lambda\right)   &  =Q_{[m+1,n]}\left(  \lambda\right)
:=\max\left\{  \lambda^{k}:k\in\mathbb{N}\cap(m,n]\right\}  =\max\left\{
\lambda^{m+1},\lambda^{n}\right\}  .
\end{align*}

\end{notation}

For $m,n\in\mathbb{N}$ with $m<n,$ it is not difficult to verify%
\[
Q_{(m,n]}\left(  \lambda+\mu\right)  \asymp Q_{(m,n]}\left(  \lambda\right)
+Q_{(m,n]}\left(  \mu\right)  \text{ for }\mu,\lambda\geq0.
\]

\begin{theorem}
[{\cite[Corollary 4.15]{Driver2018}}]\label{thm.12.3}If $A,B\in F^{\left(
\kappa\right)  }\left(  \mathbb{R}^{d}\right)  ,$ then (with $N\left(
\cdot\right)  $ as in Definition \ref{def.10.13})%
\[
d_{M}\left(  e^{V_{B}},Id_{M}\right)  \leq\left\vert V^{\left(  \kappa\right)
}\right\vert \left\vert B\right\vert \leq\left\vert V^{\left(  \kappa\right)
}\right\vert Q_{\left[  1,\kappa\right]  }\left(  N\left(  B\right)  \right)
\]
and%
\[
d_{M}\left(  e^{V_{B}}\circ e^{V_{A}},e^{V_{\log\left(  e^{A}e^{B}\right)  }%
}\right)  \leq\mathcal{K}_{0}\cdot N\left(  A\right)  N\left(  B\right)
Q_{\left[  \kappa-1,2\left(  \kappa-1\right)  \right]  }\left(  N\left(
A\right)  +N\left(  B\right)  \right)
\]%
\[
d_{M}\left(  e^{V_{B}}\circ e^{V_{A}},e^{V_{\log\left(  e^{A}e^{B}\right)  }%
}\right)  \leq\mathcal{K}_{0}\cdot Q_{(\kappa,2\kappa]}\left(  N\left(
A\right)  +N\left(  B\right)  \right)  .
\]
where for a suitable, $k<\infty,$
\[
\mathcal{K}_{0}=ke^{k\left\vert \nabla V^{\left(  \kappa\right)  }\right\vert
Q_{\left[  1,\kappa\right]  }\left(  N\left(  A\right)  \vee N\left(
B\right)  \right)  }\cdot\left\vert V^{\left(  \kappa\right)  }\right\vert
_{M}\left\vert \nabla V^{\left(  \kappa\right)  }\right\vert _{M}%
\]

\end{theorem}

We further will make use the following Riemannian metric on $TM.$

\begin{definition}
[Riemannian metric on $TM$]\label{def.12.4}Given a smooth curve, $v\left(
t\right)  \in TM$ with $v\left(  t\right)  \in T_{\sigma\left(  t\right)  }M$
where $\sigma=\pi\circ v$ is a smooth curve on $M,$ let
\[
g^{TM}\left(  \dot{v}\left(  0\right)  ,\dot{v}\left(  0\right)  \right)
:=\left\vert \dot{\sigma}\left(  0\right)  \right\vert _{g}^{2}+\left\vert
\nabla_{t}v\left(  t\right)  |_{t=0}\right\vert _{g}^{2}%
\]
where $\nabla_{t}v\left(  t\right)  $ is the covariant derivative of $v$ along
$\sigma.$ We further let $d^{TM}:TM\times TM\rightarrow\lbrack0,\infty)$ be
the associated length metric associate to $g^{TM}.$ [See \cite[Section
5]{Driver2018} for a more detailed discussion of the definitions and
properties of $g^{TM}$ and $d^{TM}.]$
\end{definition}

\begin{notation}
\label{not.12.5}For $f,g\in C^{\infty}\left(  M,M\right)  $ and a subset
$U\subset M,$ let
\[
d_{U}^{TM}\left(  f_{\ast},g_{\ast}\right)  =\sup_{m\in U}\sup\left\{
d^{TM}\left(  f_{\ast}v_{m},g_{\ast}v_{m}\right)  :v_{m}\in T_{m}M\text{ with
}\left\vert v_{m}\right\vert =1\right\}  ,
\]
where $f_{\ast}:TM\rightarrow TM$ denotes the differential of $f,$ i.e.
\[
f_{\ast}\dot{\sigma}\left(  0\right)  :=\frac{d}{dt}|_{0}f\left(
\sigma\left(  t\right)  \right)  \in T_{\sigma\left(  0\right)  }M.
\]

\end{notation}

The relationship between $\operatorname{Lip}\left(  f\right)  $ and $f_{\ast}$
is (see \cite[Lemma 2.9]{Driver2018})
\begin{equation}
\operatorname{Lip}\left(  f\right)  =\left\vert f_{\ast}\right\vert _{M}%
:=\sup\left\{  \left\vert f_{\ast}v\right\vert :v\in TM\text{ with }\left\vert
v\right\vert =1\right\}  . \label{e.12.1}%
\end{equation}

\begin{theorem}
[{\cite[Corollary 8.5]{Driver2018}}]\label{thm.12.6}If $A,B\in F^{\left(
\kappa\right)  }\left(  \mathbb{R}^{d}\right)  ,$ then%
\[
d_{M}^{TM}\left(  e_{\ast}^{V_{B}},Id_{TM}\right)  \leq\left[  \left\vert
V^{\left(  \kappa\right)  }\right\vert _{M}+\left\vert \nabla V^{\left(
\kappa\right)  }\right\vert _{M}e^{\left\vert \nabla V^{\left(  \kappa\right)
}\right\vert _{M}\left\vert B\right\vert }\right]  \left\vert B\right\vert ,
\]
and there exists $\mathcal{K}_{1}$ such that%
\[
d_{M}^{TM}\left(  \left[  e^{V_{B}}\circ e^{V_{A}}\right]  _{\ast},e_{\ast
}^{V_{\log\left(  e^{A}e^{B}\right)  }}\right)  \leq\mathcal{K}_{1}\cdot
N\left(  A\right)  N\left(  B\right)  Q_{(\kappa-1,2\left(  \kappa-1\right)
]}\left(  N\left(  A\right)  +N\left(  B\right)  \right)  .
\]
where
\[
\mathcal{K}_{1}=\mathcal{K}_{1}\left(  \left\vert V^{\left(  \kappa\right)
}\right\vert _{M},\left\vert \nabla V^{\left(  \kappa\right)  }\right\vert
_{M},H_{M}\left(  V^{\left(  \kappa\right)  }\right)  ,N\left(  A\right)  \vee
N\left(  B\right)  \right)  .
\]

\end{theorem}

In the proof of the next Corollary \ref{cor.13.5} we will use the following
two properties of $d^{TM};$%
\begin{align}
\left\vert \left\vert v\right\vert -\left\vert w\right\vert \right\vert  &
\leq d^{TM}\left(  v,w\right)  ~\forall\text{ }v,w\in TM,\text{ and}%
\label{e.12.2}\\
d^{TM}\left(  \lambda v,\lambda w\right)   &  \leq\left(  \lambda\vee1\right)
d^{TM}\left(  v,w\right)  \text{ }\forall\text{ }\lambda\geq0\text{ and
}v,w\in TM, \label{e.12.3}%
\end{align}
see \cite[Proposition 5.9]{Driver2018} and \cite[Theorem 5.11]{Driver2018}
respectively. We also need the following basic estimates for complete vector
field, $X\in\Gamma\left(  TM\right)  .$

\begin{proposition}
[{\cite[Corollary 2.27]{Driver2018}}]\label{pro.12.7}If $X\in\Gamma\left(
TM\right)  $ is complete, then%
\begin{equation}
\operatorname{Lip}\left(  e^{X}\right)  =\left\vert e_{\ast}^{X}\right\vert
_{M}\leq e^{\left\vert \nabla X\right\vert _{M}}. \label{e.12.4}%
\end{equation}

\end{proposition}

\begin{proposition}
[{\cite[Corollary 6.4]{Driver2018}}]\label{pro.12.8}If $X\in\Gamma\left(
TM\right)  $ is complete, then
\begin{equation}
d^{TM}\left(  e_{\ast}^{X}v_{m},e_{\ast}^{X}w_{p}\right)  \leq e^{2\left\vert
\nabla X\right\vert _{M}}\left[  1+H_{M}\left(  X\right)  \left\vert
w_{p}\right\vert \right]  d^{TM}\left(  v_{m},w_{p}\right)  \label{e.12.5}%
\end{equation}
where $H_{M}\left(  X\right)  $ is as in Notation \ref{not.10.17}.
\end{proposition}

\section{Proof of Theorem \ref{thm.10.27}\label{sec.13}}

In this section we fix $\alpha\in(\frac{1}{\kappa+1},\frac{1}{\kappa}]$,
$0<T<\infty,$ an $\alpha$-H\"{o}lder geometric rough path $X_{s,t}\in
G_{geo}^{\left(  \kappa\right)  }\left(  \mathbb{R}^{d}\right)  ,$ and
$\mu_{t,s}:=e^{V_{\log\left(  X_{st}\right)  }}\in\mathrm{Diff}\left(
M\right)  .$ Note that by definition of a $\alpha$-H\"{o}lder rough path,
there exists $c<\infty$ such that
\[
N\left(  X_{s,t}\right)  \leq c\left\vert t-s\right\vert ^{\alpha}\text{ for
all }\left(  s,t\right)  \in\left[  0,T\right]  ^{2}.
\]
We also let $\theta:=\alpha\left(  \kappa+1\right)  >1$ as in Eq.
(\ref{e.10.14}).

\begin{theorem}
\label{thm.13.1}Suppose that $V:\mathbb{R}^{d}\rightarrow\Gamma\left(
TM\right)  $ is a dynamical system satisfying Assumption \ref{ass.3} and $g$
is a (not necessarily complete) metric on $M$ such that $\left\vert V^{\left(
\kappa\right)  }\right\vert _{M}+\left\vert \nabla V^{\left(  \kappa\right)
}\right\vert _{M}<\infty$ (see Definition \ref{def.10.18}). Then
\[
d_{M}\left(  \mu_{t,s},Id_{M}\right)  \leq\left\vert V^{\left(  \kappa\right)
}\right\vert _{M}\left\vert \log\left(  X_{s,t}\right)  \right\vert \leq
c\left\vert V^{\left(  \kappa\right)  }\right\vert _{M}Q_{\left[
1,\kappa\right]  }\left(  c\left\vert t-s\right\vert ^{\alpha}\right)
\]
and there exists a constant, $k=k\left(  \kappa\right)  <\infty$ such that%
\begin{align*}
&  d_{M}\left(  \mu_{t,s}\circ\mu_{s,r},\mu_{t,r}\right) \\
&  \leq ke^{k\left\vert \nabla V^{\left(  \kappa\right)  }\right\vert
Q_{\left[  1,\kappa\right]  }\left(  2kc_{\kappa}\left\vert t-r\right\vert
^{\alpha}\right)  }\cdot\left\vert V^{\left(  \kappa\right)  }\right\vert
_{M}\left\vert \nabla V^{\left(  \kappa\right)  }\right\vert _{M}%
Q_{(\kappa,2\kappa]}\left(  2kc_{\kappa}\left\vert t-r\right\vert ^{\alpha
}\right)  .
\end{align*}
\textbf{Note well: }curvature does not enter these bounds!
\end{theorem}

\begin{proof}
Let $A=\log\left(  X_{r,s}\right)  $ and $B=\log\left(  X_{s,t}\right)  $ in
which case
\[
\Gamma\left(  A,B\right)  =\log\left(  e^{A}e^{B}\right)  =\log\left(
X_{r,s}X_{st}\right)  =\log\left(  X_{r,t}\right)  ,
\]%
\[
N\left(  A\right)  =N\left(  \log\left(  X_{r,s}\right)  \right)  \leq
c_{\kappa}\cdot N\left(  X_{r,s}\right)  \leq c_{\kappa}\cdot k\left\vert
s-r\right\vert ^{\alpha},
\]
and%
\[
N\left(  A\right)  \vee N\left(  B\right)  \leq kc_{\kappa}\left\vert
t-r\right\vert ^{\alpha}.
\]
Thus as an application of Theorem \ref{thm.12.3},%
\begin{align*}
d_{M}\left(  \mu_{t,s},Id_{M}\right)   &  \leq\left\vert V^{\left(
\kappa\right)  }\right\vert \left\vert \log\left(  X_{s,t}\right)  \right\vert
\leq\left\vert V^{\left(  \kappa\right)  }\right\vert Q_{\left[
1,\kappa\right]  }\left(  N\left(  \log\left(  X_{s,t}\right)  \right)
\right) \\
&  \leq\left\vert V^{\left(  \kappa\right)  }\right\vert Q_{\left[
1,\kappa\right]  }\left(  c_{\kappa}N\left(  X_{s,t}\right)  \right)
\leq\left\vert V^{\left(  \kappa\right)  }\right\vert Q_{\left[
1,\kappa\right]  }\left(  c_{\kappa}k\left\vert t-s\right\vert ^{\alpha
}\right)
\end{align*}
and%
\begin{align*}
d_{M}  &  \left(  \mu_{t,s}\circ\mu_{s,r},\mu_{t,r}\right)  =d_{M}\left(
e^{V_{B}}\circ e^{V_{A}},e^{V_{\Gamma\left(  A,B\right)  }}\right) \\
&  \leq ke^{k\left\vert \nabla V^{\left(  \kappa\right)  }\right\vert
Q_{\left[  1,\kappa\right]  }\left(  N\left(  A\right)  \vee N\left(
B\right)  \right)  }\cdot\left\vert V^{\left(  \kappa\right)  }\right\vert
_{M}\left\vert \nabla V^{\left(  \kappa\right)  }\right\vert _{M}%
Q_{(\kappa,2\kappa]}\left(  N\left(  A\right)  \vee N\left(  B\right)  \right)
\\
&  \leq ke^{k\left\vert \nabla V^{\left(  \kappa\right)  }\right\vert
Q_{\left[  1,\kappa\right]  }\left(  2kc_{\kappa}\left\vert t-r\right\vert
^{\alpha}\right)  }\cdot\left\vert V^{\left(  \kappa\right)  }\right\vert
_{M}\left\vert \nabla V^{\left(  \kappa\right)  }\right\vert _{M}%
Q_{(\kappa,2\kappa]}\left(  2kc_{\kappa}\left\vert t-r\right\vert ^{\alpha
}\right)  .
\end{align*}

\end{proof}

\begin{notation}
\label{not.13.2}Let $\Gamma_{c}\left(  TM\right)  $ denote the smooth
compactly supported vector fields on $M.$
\end{notation}

\begin{theorem}
[Localizing Approximates]\label{thm.13.3}Let us continue the notation and
assumptions in Theorem \ref{thm.13.1}. Assume that $g$ is a complete
Riemannian metric on $M$ (i.e. $\left(  M,d\right)  $ is complete) and $K$ is
a compact subset of $M,$ i.e. $K$ is closed and bounded. Then there exists
$\tilde{V}:\mathbb{R}^{d}\rightarrow\Gamma_{c}\left(  TM\right)  $ such that
$\tilde{\mu}_{t,s}^{n}=\mu_{t,s}^{n}$ on $K$ for all $n\in\mathbb{N}$ and
$\left(  s,t\right)  \in\left[  0,T\right]  ^{2}.$
\end{theorem}

\begin{proof}
Let $\theta:=\left(  \kappa+1\right)  \alpha>1$ as in Assumption \ref{ass.2}.
Let $K_{1}$ be the compact subset of $M$ containing $K$ defined by
\[
K_{1}:=\left\{  \mu_{t,s}\left(  m\right)  :\left(  s,t\right)  \in\left[
0,T\right]  ^{2}\text{ and }m\in K\right\}  .
\]
By Theorem \ref{thm.13.1}, and Theorem \ref{thm.11.18} there exists $C<\infty$
such that
\[
d_{M}\left(  \mu_{t,s}^{n},\mu_{t,s}\right)  \leq C\left\vert t-s\right\vert
^{\theta}\text{ for all }\left(  s,t\right)  \in\left[  0,T\right]  ^{2}.
\]
Now let
\[
K_{2}:=\left\{  m\in M:d\left(  m,K_{1}\right)  \leq C\left[  T^{\theta}%
\vee1\right]  \right\}
\]
which is again closed and bounded and hence compact. Moreover we have
$\mu_{t,s}^{n}\left(  m\right)  \in K_{2}$ for all $m\in K,$ $\left(
s,t\right)  \in\left[  0,T\right]  ^{2},$ and $n\in\mathbb{N}.$ Lastly let
\[
R:=\max_{\left(  s,t\right)  \in\left[  0,T\right]  ^{2}}\left\vert
V_{\log\left(  X_{s,t}\right)  }\right\vert _{M}\text{ and }K_{3}:=\left\{
m\in M:d\left(  m,K_{2}\right)  \leq R\right\}  .
\]
We then have $\mu_{t,s}\left(  K_{2}\right)  \subset K_{3}$ for all $\left(
s,t\right)  \in\left[  0,T\right]  ^{2}.$ Let $\varphi\in C_{c}^{\infty
}\left(  M,\left[  0,1\right]  \right)  $ such that $\varphi=1$ on a
neighborhood of $K_{3}$ and define $\tilde{V}:=\varphi V.$ We will finish the
proof by showing $\tilde{\mu}_{t,s}^{n}=\mu_{t,s}^{n}$ on $K$ for all
$n\in\mathbb{N}$ and $\left(  s,t\right)  \in\left[  0,T\right]  ^{2}.$

If $m\in K_{2}$ and $0\leq\tau\leq1,$ then
\[
d\left(  e^{\tau V_{\log\left(  X_{st}\right)  }}\left(  m\right)  ,m\right)
\leq\tau\left\vert V_{\log\left(  X_{s,t}\right)  }\right\vert _{M}\leq R
\]
and so $e^{\tau V_{\log\left(  X_{st}\right)  }}\left(  m\right)  \in K_{3}$
for all $0\leq\tau\leq1.$ Therefore,%
\begin{align*}
\frac{d}{d\tau}e^{\tau V_{\log\left(  X_{st}\right)  }}\left(  m\right)   &
=V_{\log\left(  X_{st}\right)  }\left(  e^{\tau V_{\log\left(  X_{st}\right)
}}\left(  m\right)  \right) \\
&  =\tilde{V}_{\log\left(  X_{st}\right)  }\left(  e^{\tau V_{\log\left(
X_{st}\right)  }}\left(  m\right)  \right)  \text{ for }0\leq\tau\leq1.
\end{align*}
From this we conclude that
\[
e^{\tau V_{\log\left(  X_{st}\right)  }}\left(  m\right)  =e^{\tau\tilde
{V}_{\log\left(  X_{st}\right)  }}\left(  m\right)  \text{ }\forall~m\in
K_{2}\text{ and }0\leq\tau\leq1
\]
and in particular this implies that $\tilde{\mu}_{t,s}=\mu_{t,s}$ on $K_{2}$
for all $\left(  s,t\right)  \in\left[  0,T\right]  ^{2}.$ This shows
$\tilde{\mu}_{t,s}^{n}=\mu_{t,s}^{n}$ on $K$ for all $\left(  s,t\right)
\in\left[  0,T\right]  ^{2}$ when $n=1.$ We now finish proof by induction on
$n.$ If $m\in K,$ $n\in\mathbb{N},$ and
\[
u:=t-\frac{1}{n+1}\left(  t-s\right)  ,
\]
then $\mu_{u,s}^{n}\left(  m\right)  \in K_{2}$ and so using the induction
hypothesis and the fact that $\tilde{\mu}_{t,s}=\mu_{t,s}$ on $K_{2}$ we find,%
\begin{align*}
\tilde{\mu}_{t,s}^{n+1}\left(  m\right)   &  =\tilde{\mu}_{t,u}\circ\tilde
{\mu}_{u,s}^{n}\left(  m\right)  =\tilde{\mu}_{t,u}\circ\mu_{u,s}^{n}\left(
m\right) \\
&  =\mu_{t,u}\circ\mu_{u,s}^{n}\left(  m\right)  =\mu_{t,s}^{n+1}\left(
m\right)  .
\end{align*}

\end{proof}

To prove Theorem \ref{thm.10.27}, we will (after \textquotedblleft
localization\textquotedblright) apply Theorem \ref{thm.11.25} with $\mu_{t,s}$
as in Definition \ref{def.10.25}. Hence to carry out this proof we must verify
that $\mu_{t,s}$ in Definition \ref{def.10.25} satisfies the hypotheses in
Theorem \ref{thm.10.27}. The most challenging aspect of the proof is to
control the size of $\operatorname{Lip}\left(  \mu_{t,s}^{\pi}\right)  $ over
all $1/3$ - special partitions $\pi$ of $J\left(  s,t\right)  .$ Following
Ballieul, we will do this by proving a version of Lemma \ref{lem.11.17} with
$\mu_{t,s}$ replaced by $\left(  \mu_{t,s}\right)  _{\ast}.$

\begin{theorem}
\label{thm.13.4}Let $V:\mathbb{R}^{d}\rightarrow\Gamma\left(  TM\right)  $ be
a dynamical system such that
\[
\left\vert V^{\left(  \kappa\right)  }\right\vert _{M}+\left\vert \nabla
V^{\left(  \kappa\right)  }\right\vert _{M}+H_{M}\left(  V^{\left(
\kappa\right)  }\right)  <\infty,
\]
let $1/\left(  \kappa+1\right)  <\alpha\leq1/\kappa,$ $X_{s,t}\in
G_{geo}^{\left(  \kappa\right)  }\left(  \mathbb{R}^{d}\right)  $ be an
$\alpha$-H\"{o}lder geometric rough path in $\mathbb{R}^{d},$ and for $\left(
s,t\right)  \in\left[  0,T\right]  ^{2},$ let
\[
\mu_{t,s}:=e^{V_{\log\left(  X_{st}\right)  }}\in\mathrm{Diff}\left(
M\right)
\]
as in Definition \ref{def.10.25}. Then $\mu_{t,s}$ is a good approximate flow,
i.e.%
\begin{align*}
d_{M}\left(  \mu_{t,s\ast},Id_{TM}\right)   &  \leq C\left(  \left\vert
V^{\left(  \kappa\right)  }\right\vert _{M},\left\vert \nabla V^{\left(
\kappa\right)  }\right\vert _{M}\right)  \left\vert t-s\right\vert ^{\alpha
}\text{ and}\\
d_{M}^{TM}\left(  \mu_{t,s\ast}\mu_{s,r\ast},\mu_{t,r\ast}\right)   &
\leq\mathcal{K}\left(  \left\vert V^{\left(  \kappa\right)  }\right\vert
_{M},\left\vert \nabla V^{\left(  \kappa\right)  }\right\vert _{M}%
,H_{M}\left(  V^{\left(  \kappa\right)  }\right)  \right)  \left\vert
t-r\right\vert ^{\theta},
\end{align*}
where $\theta=\left(  \kappa+1\right)  \alpha$ as Assumption \ref{ass.2}. The
constant $C$ and $\mathcal{K}$ also depend on $\sup_{\left(  s,t\right)
\in\left[  0,T\right]  ^{2}}\left\vert \log\left(  X_{s,t}\right)  \right\vert
.$ We suppress this dependence as we view the geometric rough path, $X,$ as
being fixed.
\end{theorem}

\begin{proof}
As in the proof of Theorem \ref{thm.13.1}, let $A=\log\left(  X_{r,s}\right)
$ and $B=\log\left(  X_{s,t}\right)  ,$ so that $\Gamma\left(  A,B\right)
=\log\left(  X_{r,t}\right)  ,$
\[
\mu_{t,s\ast}=e_{\ast}^{V_{B}},\quad\mu_{s,r\ast}=e_{\ast}^{V_{A}}%
,\quad\text{and }\mu_{t,r\ast}=e_{\ast}^{V_{\Gamma\left(  A,B\right)  }}.
\]
Thus by Theorem \ref{thm.12.6},
\begin{align*}
d_{M}^{TM}\left(  \mu_{t,s\ast},Id_{TM}\right)   &  \leq\left(  \left\vert
V^{\left(  \kappa\right)  }\right\vert _{M}+\left\vert \nabla V^{\left(
\kappa\right)  }\right\vert _{M}e^{\left\vert \nabla V^{\left(  \kappa\right)
}\right\vert _{M}^{2}\left\vert \log\left(  X_{s,t}\right)  \right\vert
}\right)  \cdot\left\vert \log\left(  X_{s,t}\right)  \right\vert \\
&  \leq C\left(  \left\vert V^{\left(  \kappa\right)  }\right\vert
_{M},\left\vert \nabla V^{\left(  \kappa\right)  }\right\vert _{M}\right)
\left\vert t-s\right\vert ^{\alpha}%
\end{align*}
and
\[
d_{M}^{TM}\left(  \mu_{t,s\ast}\mu_{s,r\ast},\mu_{t,r\ast}\right)
\leq\mathcal{K}\cdot\left(  N\left(  A\right)  \vee N\left(  B\right)
\right)  ^{\kappa+1}\leq\mathcal{K}\left\vert t-r\right\vert ^{\theta}.
\]

\end{proof}

\begin{corollary}
\label{cor.13.5}Under the same assumptions of Theorem \ref{thm.13.4}, if
$\varepsilon\in(0,1/2]$ is given, there exists a $\delta>0$ and $L<\infty$
such that for any $\left(  s,t\right)  \in\left[  0,T\right]  ^{2}$ with
$\left\vert t-s\right\vert \leq\delta$ and any $\varepsilon$-special
partitions, $\pi,$ of $J\left(  s,t\right)  ,$%
\begin{equation}
d_{M}^{TM}\left(  \mu_{t,s\ast}^{\pi},\mu_{t,s\ast}\right)  \leq L\left\vert
t-s\right\vert ^{\theta}. \label{e.13.1}%
\end{equation}
where $\theta=\left(  \kappa+1\right)  \alpha>1$ as in Eq. (\ref{e.10.14}).
\end{corollary}

\begin{proof}
When $\pi=\left\{  s<u<t\right\}  ,$ Eq. (\ref{e.13.1}) holds with $L=C$ as in
Theorem \ref{thm.13.4}. Thus we will need to take $L\geq c.$ The proof will
now be completed by induction on $n=\#\left(  \pi\right)  .$ Let $\pi$ be and
$\varepsilon$-special partition of $J\left(  s,t\right)  $ and $u\in
J_{\varepsilon}\left(  s,t\right)  \cap\pi$ which exists as $\pi$ is
$\varepsilon$-special. Then for a unit vector, $v_{m}\in T_{m}M,$
\begin{align}
d^{TM}  &  \left(  \mu_{t,s\ast}^{\pi}v_{m},\mu_{t,s\ast}v_{m}\right)
\nonumber\\
&  =d^{TM}\left(  \mu_{t,u\ast}^{\pi}\mu_{u,s\ast}^{\pi}v_{m},\mu_{t,s\ast
}v_{m}\right) \nonumber\\
&  \leq d^{TM}\left(  \mu_{t,u\ast}^{\pi}\mu_{u,s\ast}^{\pi}v_{m},\mu
_{t,u\ast}\mu_{u,s\ast}^{\pi}v_{m}\right) \nonumber\\
&  \quad+d^{TM}\left(  \mu_{t,u\ast}\mu_{u,s\ast}^{\pi}v_{m},\mu_{t,u\ast}%
\mu_{u,s\ast}v_{m}\right)  +d^{TM}\left(  \mu_{t,u\ast}\mu_{u,s\ast}v_{m}%
,\mu_{t,s\ast}v_{m}\right)  . \label{e.13.2}%
\end{align}
By Theorem \ref{thm.13.4},
\[
d_{M}^{TM}\left(  \mu_{t,u\ast}\mu_{u,s\ast}v_{m},\mu_{t,s\ast}v_{m}\right)
\leq C\left\vert t-s\right\vert ^{\theta}.
\]
By Eq. (\ref{e.12.3}) and the induction hypothesis,
\begin{align*}
d_{M}^{TM}\left(  \mu_{t,u\ast}^{\pi}\mu_{u,s\ast}^{\pi}v_{m},\mu_{t,u\ast}%
\mu_{u,s\ast}^{\pi}v_{m}\right)   &  \leq\left(  1\vee\left\vert \mu_{u,s\ast
}^{\pi}v_{m}\right\vert \right)  d_{M}^{TM}\left(  \mu_{t,u\ast}^{\pi}%
,\mu_{t,u\ast}\right) \\
&  \leq\left(  1\vee\left\vert \mu_{u,s\ast}^{\pi}v_{m}\right\vert \right)
L\left\vert t-u\right\vert ^{\theta}.
\end{align*}
Now by Eq. (\ref{e.12.2}), Theorem \ref{thm.13.4}, and the induction
hypothesis,
\begin{align}
\left\vert \mu_{u,s\ast}^{\pi}v_{m}\right\vert  &  \leq\left\vert
v_{m}\right\vert +d^{TM}\left(  \mu_{u,s\ast}^{\pi}v_{m},v_{m}\right)
\nonumber\\
&  \leq1+d^{TM}\left(  \mu_{u,s\ast}^{\pi}v_{m},\mu_{u,s\ast}v_{m}\right)
+d^{TM}\left(  \mu_{u,s\ast}v_{m},v_{m}\right) \nonumber\\
&  \leq1+L\left\vert u-s\right\vert ^{\theta}+\tilde{C}\left\vert
u-s\right\vert ^{\alpha} \label{e.13.3}%
\end{align}
which combined with the previous estimate shows%
\[
d^{TM}\left(  \mu_{t,u\ast}^{\pi}\mu_{u,s\ast}^{\pi}v_{m},\mu_{t,u\ast}%
\mu_{u,s\ast}^{\pi}v_{m}\right)  \leq\left(  1+\tilde{C}\left\vert
u-s\right\vert ^{\alpha}+L\left\vert u-s\right\vert ^{\theta}\right)
L\left\vert t-u\right\vert ^{\theta}.
\]
Thus we have shown so far that%
\begin{align}
d^{TM}\left(  \mu_{t,s\ast}^{\pi}v_{m},\mu_{t,s\ast}v_{m}\right)  \leq &
\left(  1+\tilde{C}\left\vert u-s\right\vert ^{\alpha}+L\left\vert
u-s\right\vert ^{\theta}\right)  L\left\vert t-u\right\vert ^{\theta
}+C\left\vert t-s\right\vert ^{\theta}\nonumber\\
&  +d^{TM}\left(  \mu_{t,u\ast}\mu_{u,s\ast}^{\pi}v_{m},\mu_{t,u\ast}%
\mu_{u,s\ast}v_{m}\right)  . \label{e.13.4}%
\end{align}
By Proposition \ref{pro.12.8} with $X=V_{\log\left(  X_{u,t}\right)  },$
$v_{m}\rightarrow\mu_{u,s\ast}^{\pi}v_{m}$ and $w_{p}\rightarrow\mu_{u,s\ast
}v_{m}$ we have%
\begin{align*}
d^{TM}  &  \left(  \mu_{t,u\ast}\mu_{u,s\ast}^{\pi}v_{m},\mu_{t,u\ast}%
\mu_{u,s\ast}v_{m}\right) \\
&  \leq e^{2\left\vert \nabla V_{\log\left(  X_{u,t}\right)  }\right\vert
_{M}}\left[  1+H_{M}\left(  V_{\log\left(  X_{u,t}\right)  }\right)
\left\vert \mu_{u,s\ast}v_{m}\right\vert \right]  d^{TM}\left(  \mu_{u,s\ast
}^{\pi}v_{m},\mu_{u,s\ast}v_{m}\right) \\
&  \leq e^{c\left\vert u-t\right\vert ^{\alpha}}\left[  1+c\left\vert
u-t\right\vert ^{\alpha}\left\vert \mu_{u,s\ast}v_{m}\right\vert \right]
d^{TM}\left(  \mu_{u,s\ast}^{\pi}v_{m},\mu_{u,s\ast}v_{m}\right) \\
&  \leq e^{c\left\vert u-t\right\vert ^{\alpha}}\left[  1+c\left\vert
u-t\right\vert ^{\alpha}\left\vert \mu_{u,s\ast}v_{m}\right\vert \right]
L\left\vert u-s\right\vert ^{\theta}\\
&  \leq e^{c\left\vert u-t\right\vert ^{\alpha}}\left[  1+c\left\vert
u-t\right\vert ^{\alpha}\left(  1+\tilde{C}\left\vert u-s\right\vert ^{\alpha
}\right)  \right]  L\left\vert u-s\right\vert ^{\theta}\\
&  \leq\left(  1+\hat{c}\left\vert t-s\right\vert ^{\alpha}\right)
L\left\vert u-s\right\vert ^{\theta}%
\end{align*}
wherein we have used Eq. (\ref{e.12.2}) along with Theorem \ref{thm.13.4} to
conclude%
\[
\left\vert \mu_{u,s\ast}v_{m}\right\vert \leq\left\vert v_{m}\right\vert
+d\left(  \mu_{u,s\ast}v_{m},v_{m}\right)  =1+\tilde{C}\left\vert
u-s\right\vert ^{\alpha}.
\]
Combining this result with Eq. (\ref{e.13.4}) shows
\begin{align*}
d^{TM}  &  \left(  \mu_{t,s\ast}^{\pi}v_{m},\mu_{t,s\ast}v_{m}\right) \\
\leq &  \left(  1+\tilde{C}\left\vert u-s\right\vert ^{\alpha}+L\left\vert
u-s\right\vert ^{\theta}\right)  L\left\vert t-u\right\vert ^{\theta}\\
&  +C\left\vert t-s\right\vert ^{\theta}+\left(  1+\hat{c}\left\vert
t-s\right\vert ^{\alpha}\right)  L\left\vert u-s\right\vert ^{\theta}\\
\leq &  \gamma\left(  \varepsilon,\theta\right)  L\left\vert t-s\right\vert
^{\theta}+\left(  \tilde{C}\left\vert u-s\right\vert ^{\alpha}+L\left\vert
u-s\right\vert ^{\theta}\right)  L\left\vert t-u\right\vert ^{\theta}\\
&  +C\left\vert t-s\right\vert ^{\theta}+\hat{c}\left\vert t-s\right\vert
^{\alpha}L\left\vert u-s\right\vert ^{\theta}\\
\leq &  \left(  \left[  \gamma\left(  \varepsilon,\theta\right)  +\left(
\tilde{C}+\hat{c}\right)  \delta^{\alpha}+\delta^{\theta}L\right]  L+C\right)
\left\vert t-s\right\vert ^{\theta}\\
=  &  \left(  \left[  \gamma\left(  \varepsilon,\theta\right)  +C^{\prime
}\delta^{\alpha}+\delta^{\theta}L\right]  L+C\right)  \left\vert
t-s\right\vert ^{\theta}.
\end{align*}

We now choose $\delta>0$ so small that $\gamma\left(  \varepsilon
,\theta\right)  +C^{\prime}\delta^{\alpha}<1,$ for example, choose $\delta$ so
that%
\[
\gamma\left(  \varepsilon,\theta\right)  +C^{\prime}\delta^{\alpha}\leq
\frac{1+\gamma\left(  \varepsilon,\theta\right)  }{2}:=\alpha\left(
\varepsilon,\theta\right)  <1.
\]
Then we want to choose $L$ so that%
\[
\alpha\left(  \varepsilon,\theta\right)  L+2\delta^{\theta}L^{2}+C\leq L\iff
L\geq\frac{1}{1-\alpha\left(  \varepsilon,\theta\right)  }\left[
C+2\delta^{\theta}L^{2}\right]  .
\]
We do this by requiring
\[
L=\frac{2}{1-\alpha\left(  \varepsilon,\theta\right)  }C
\]
and then shrink $\delta$ so that $2\delta^{\theta}L^{2}\leq C.$ For these
choices we have
\[
\alpha\left(  \varepsilon,\theta\right)  L_{1}+2\delta^{\theta}L^{2}+C\leq L
\]
and the induction step is complete.
\end{proof}

\begin{corollary}
[Lip-bounds]\label{cor.13.6}Under the same assumptions of Theorem
\ref{thm.13.4}, to each $\varepsilon\in(0,1/2],$ there exists $C=C\left(
\varepsilon\right)  <\infty$ such that $\operatorname{Lip}\left(  \mu
_{t,s}^{\pi}\right)  \leq C<\infty$ for all $\left(  s,t\right)  \in\left[
0,T\right]  ^{2}$ and all $\varepsilon$-special partitions, $\pi,$ of
$J\left(  s,t\right)  .$
\end{corollary}

\begin{proof}
By Eq. (\ref{e.12.2}) and Corollary \ref{cor.13.5} if we assume that $\left(
t-s\right)  \leq\delta,$ then (using Eq. (\ref{e.12.1}))
\begin{align*}
\operatorname{Lip}\left(  \mu_{t,s}^{\pi}\right)   &  =\left\vert \mu
_{t,s\ast}^{\pi}\right\vert _{M}\leq\left\vert \mu_{t,s\ast}\right\vert
_{M}+d_{M}^{TM}\left(  \mu_{t,s\ast}^{\pi},\mu_{t,s\ast}\right) \\
&  \leq\left\vert \mu_{t,s\ast}\right\vert _{M}+L\left\vert t-s\right\vert
^{\theta}.
\end{align*}
Moreover by Proposition \ref{pro.12.7},
\[
\left\vert \mu_{t,s\ast}\right\vert _{M}\leq e^{\left\vert \nabla
V_{\log\left(  X_{s,t}\right)  }\right\vert _{M}}\leq e^{C\left\vert
t-s\right\vert ^{\alpha}}.
\]
Thus it follows that
\[
\operatorname{Lip}\left(  \mu_{t,s}^{\pi}\right)  \leq e^{C\delta^{\alpha}%
}+L\delta^{\theta}<\infty.
\]
We may now remove the restriction that $\left(  t-s\right)  \leq\delta$ by an
application of Lemma \ref{lem.11.19}.
\end{proof}

We are now ready to prove Theorem \ref{thm.10.27} which we split into two
parts; 1) existence is subsection \ref{sec.13.1} and 2) uniqueness in
subsection \ref{sec.13.2}.

\begin{theorem}
[Global Existence]\label{thm.13.8}Suppose that $V:\mathbb{R}^{d}%
\rightarrow\Gamma\left(  TM\right)  $ is a dynamical system satisfying
Assumption \ref{ass.3} and there exists a complete metric, $g,$ on $M$ such
that $\left\vert V^{\left(  \kappa\right)  }\right\vert _{M}+\left\vert \nabla
V^{\left(  \kappa\right)  }\right\vert _{M}<\infty.$ Then there exists a
unique function $\varphi\in C\left(  \left[  0,T\right]  ^{2}\times
M,M\right)  $ such that;

\begin{enumerate}
\item $\varphi_{t,t}=Id$ for all $t\in\left[  0,T\right]  ,$

\item $\varphi_{t,s}\circ\varphi_{s,r}=\varphi_{t,r}$ $\forall~0\leq r\leq
s\leq t\leq T,$ and

\item there exists a constant $C<\infty$ such that
\[
d_{M}\left(  \varphi_{t,s},\mu_{t,s}\right)  \leq C\left\vert t-s\right\vert
^{\theta}\text{ ~}\forall~\left(  s,t\right)  \in\left[  0,T\right]  ^{2}.
\]
Moreover, $C\left(  K\right)  :=\sup_{\left(  s,t\right)  \in\left[
0,T\right]  ^{2}}\operatorname{Lip}_{K}\left(  \varphi_{t,s}\right)  <\infty$
for all compact subsets, $K,$ of $M.$
\end{enumerate}
\end{theorem}

\subsection{Existence proof for Theorem \ref{thm.10.27}\label{sec.13.1}}

Suppose that $V:\mathbb{R}^{d}\rightarrow\Gamma\left(  TM\right)  $ is
dynamical system of complete vector fields and $g$ is a Riemannian metric on
$M$ so that $\left\vert V^{\left(  \kappa\right)  }\right\vert _{M}+\left\vert
\nabla V^{\left(  \kappa\right)  }\right\vert _{M}<\infty.$ As usual let
\[
\mu_{t,s}:=e^{V_{\log\left(  X_{s,t}\right)  }}.
\]
By Proposition \ref{pro.12.7} we know that
\begin{align*}
\operatorname{Lip}\left(  \mu_{t,s}\right)   &  =\left\vert e_{\ast}%
^{V_{\log\left(  X_{s,t}\right)  }}\right\vert _{M}\leq e^{\left\vert \nabla
V_{\log\left(  X_{s,t}\right)  }\right\vert _{M}}\leq e^{\left\vert \nabla
V^{\left(  \kappa\right)  }\right\vert _{M}\left\vert \log\left(
X_{s,t}\right)  \right\vert }\\
&  \leq e^{c\left\vert \nabla V^{\left(  \kappa\right)  }\right\vert
_{M}\left\vert t-s\right\vert ^{\alpha}}=1+O\left(  \left\vert t-s\right\vert
^{\alpha}\right)  .
\end{align*}
From Theorem \ref{thm.13.1} we know%
\[
d_{M}\left(  \mu_{t,s}\circ\mu_{s,r},\mu_{t,r}\right)  \leq\mathcal{K}%
\left\vert t-r\right\vert ^{\alpha}\text{ for all }0\leq r\leq s\leq t\leq T.
\]
where $\mathcal{K=K}\left(  \left\vert V^{\left(  \kappa\right)  }\right\vert
_{M},\left\vert \nabla V^{\left(  \kappa\right)  }\right\vert _{M}\right)  .$
We can now apply the Trotter approximation bounds in Lemma \ref{lem.11.17}
along with the large-time extension Theorem \ref{thm.11.18} to conclude, for
ever $\varepsilon\in(0,\frac{1}{2}],$ there exists and $L_{\varepsilon}%
<\infty$ such that%
\begin{equation}
d\left(  \mu_{t,s}^{\pi},\mu_{t,s}\right)  \leq L_{\varepsilon}\left\vert
t-s\right\vert ^{\theta}\text{ for all }\left(  s,t\right)  \in\left[
0,T\right]  ^{2}. \label{e.13.5}%
\end{equation}

Let us now further assume that $g$ is a complete metric and choose compact
subsets, $\left\{  K_{N}\right\}  _{N=1}^{\infty},$ of $M$ so that
$K_{N}\subset K_{N+1}^{o}$ for all $N\in\mathbb{N}$ and $K_{N}\uparrow M$ as
$N\uparrow\infty.$ By Theorem \ref{thm.13.3} there exists $V^{N}\in\Gamma
_{c}\left(  TM\right)  $ so that if $\mu_{N,t,s}:=e^{V_{\log\left(
X_{s,t}\right)  }^{N}},$ then
\[
\mu_{N,t,s}^{n}=\mu_{t,s}^{n}\text{ on }K_{N}\text{ for all }n\in\mathbb{N}.
\]
We further have by Corollary \ref{cor.13.6} that to each $\varepsilon
\in(0,1/2],$ there exists $C=C_{N}\left(  \varepsilon\right)  <\infty$ such
that $\operatorname{Lip}\left(  \mu_{N,t,s}^{\pi}\right)  \leq C<\infty$ for
all $\left(  s,t\right)  \in\left[  0,T\right]  ^{2}$ and all $\varepsilon
$-special partitions, $\pi,$ of $J\left(  s,t\right)  .$ These assertions
verify the assumptions of Theorem \ref{thm.11.25} from which it follows that
$\varphi_{N,t,s}=\lim_{n\rightarrow\infty}\mu_{N,t,s}^{2^{n}}$ exists
uniformly on $M$ and moreover, $\varphi_{N,t,s}\in C\left(  M,M\right)  $ is a
semi-group such that%
\[
d\left(  \varphi_{N,t,s},\mu_{N,t,s}\right)  \leq L_{1/3}\left(  N\right)
\left\vert t-s\right\vert ^{\theta}.
\]
If $\hat{N}>N,$ then we will have
\[
\varphi_{\hat{N},t,s}=\lim_{n\rightarrow\infty}\mu_{\hat{N},t,s}^{2^{n}}%
=\lim_{n\rightarrow\infty}\mu_{t,s}^{2^{n}}=\lim_{n\rightarrow\infty}%
\mu_{N,t,s}^{2^{n}}=\varphi_{N,t,s}\text{ on }K_{N}.
\]
This shows that it is well defined to put $\varphi_{t,s}:=\varphi_{N,t,s}$ on
$K_{N}$ and moreover we have in fact shown that $\varphi_{t,s}=\lim
_{n\rightarrow\infty}\mu_{t,s}^{2^{n}}$ uniformly on compact subsets of $M.$
Passing to the limit in Eq. (\ref{e.13.5}) shows that
\[
d\left(  \varphi_{t,s},\mu_{t,s}\right)  \leq L_{1/3}\left\vert t-s\right\vert
^{\theta}\text{ for all }\left(  s,t\right)  \in\left[  0,T\right]  ^{2}.
\]
Choose $\hat{N}$ sufficiently large such that $\cup_{0\leq s,t\leq T}%
\varphi_{N,t,s}\left(  K_{N}\right)  $ is a compact subset of $K_{\hat{N}}%
^{o}$ so that if $m\in K_{N},$ then
\begin{align*}
\varphi_{t,s}\circ\varphi_{s,r}\left(  m\right)   &  =\varphi_{t,s}%
\circ\varphi_{N,s,r}\left(  m\right)  =\varphi_{\hat{N},t,s}\circ
\varphi_{N,s,r}\left(  m\right) \\
&  =\varphi_{\hat{N},t,s}\circ\varphi_{\hat{N},s,r}\left(  m\right)
=\varphi_{\hat{N},t,r}\left(  m\right)  =\varphi_{t,r}\left(  m\right)  .
\end{align*}
As $N\in\mathbb{N}$ was arbitrary and $K_{N}\uparrow M$ as $N\uparrow\infty$
we conclude that $\varphi_{t,s}\circ\varphi_{s,r}=\varphi_{t,r}.$ Thus we have
proved the existence assertion in Theorem \ref{thm.10.27}.

\subsection{Uniqueness proof for Theorem \ref{thm.10.27}\label{sec.13.2}}

One way to prove uniqueness is to show, with $\psi_{t,s}$ satisfies items
1.-3. of Theorem \ref{thm.10.27}, for $m\in M$ that $x_{t}:=\psi_{t,0}\left(
m\right)  $ (or more generally $x_{t}:=\psi_{t,s}\left(  m\right)  $ for
$t\geq s)$ satisfies the RDE as described in Proposition \ref{pro.10.28} and
that solutions in this sense are unique. This would then show that $\psi
_{t,s}\left(  m\right)  =\varphi_{t,s}\left(  m\right)  $ where $\psi$ is any
other solution as described in Theorem \ref{thm.10.27}. We wish to avoid
developing the path-wise notion of solutions in this paper and so we will try
to use a variant of proof given in Theorem \ref{thm.11.25} instead.

Let $\left(  s,t,m\right)  \rightarrow\psi_{t,s}\left(  m\right)  $ be a
continuous function which satisfies 1.-3. of Theorem \ref{thm.10.27} and
$\varphi_{t,s}$ be the solution we have already constructed. Then by the
triangle inequality we know that
\[
d_{M}\left(  \varphi_{t,s},\psi_{t,s}\right)  \leq2C\left\vert t-s\right\vert
^{\theta}\text{ ~}\forall~\left(  s,t\right)  \in\left[  0,T\right]  ^{2}.
\]
Given a compact subset, $K\subset M,$ let
\[
K^{\prime}=\left[  \cup_{\left(  s,t\right)  \in\left[  0,T\right]  ^{2}}%
\psi_{t,s}\left(  K\right)  \right]  \cup\left[  \cup_{\left(  s,t\right)
\in\left[  0,T\right]  ^{2}}\varphi_{t,s}\left(  K\right)  \right]
\]
which is still compact since it is the union of two compact sets. Let
$c:=c\left(  K^{\prime}\right)  :=\sup_{\left(  s,t\right)  \in\left[
0,T\right]  ^{2}}\operatorname{Lip}_{K^{\prime}}\left(  \varphi_{s,t}\right)
<\infty.$ Next suppose $\left(  s,t\right)  \in\left[  0,T\right]  ^{2},$
$n\in\mathbb{N},$ and $s_{k}:=s+k\frac{t-s}{n}$ for $0\leq k\leq n.$ We then
find, using the notation in Eq. (\ref{e.10.15}) with $U=K$ and $U=\psi
_{s_{1},s_{0}}\left(  K\right)  ),$ that%
\begin{align*}
d_{K}\left(  \psi_{t,s},\varphi_{t,s}\right)   &  =d_{K}\left(  \psi_{t,s_{1}%
}\circ\psi_{s_{1},s_{0}},\varphi_{t,s_{1}}\circ\varphi_{s_{1},s_{0}}\right) \\
&  \leq d_{K}\left(  \psi_{t,s_{1}}\circ\psi_{s_{1},s_{0}},\varphi_{t,s_{1}%
}\circ\psi_{s_{1},s_{0}}\right)  +d_{K}\left(  \varphi_{t,s_{1}}\circ
\psi_{s_{1},s_{0}},\varphi_{t,s_{1}}\circ\varphi_{s_{1},s_{0}}\right) \\
&  \leq d_{\psi_{s_{1},s_{0}}\left(  K\right)  }\left(  \psi_{t,s_{1}}%
,\varphi_{t,s_{1}}\right)  +cd_{K}\left(  \psi_{s_{1},s_{0}},\varphi
_{s_{1},s_{0}}\right) \\
&  \leq d_{\psi_{s_{1},s_{0}}\left(  K\right)  }\left(  \psi_{t,s_{1}}%
,\varphi_{t,s_{1}}\right)  +c2C\left\vert s_{1}-s_{0}\right\vert ^{\theta}.
\end{align*}
Similarly,%
\begin{align*}
d  &  _{\psi_{s_{1},s_{0}}\left(  K\right)  }\left(  \psi_{t,s_{1}}%
,\varphi_{t,s_{1}}\right) \\
&  =d_{\psi_{s_{1},s_{0}}\left(  K\right)  }\left(  \psi_{t,s_{2}}\circ
\psi_{s_{2},s_{1}},\varphi_{t,s_{2}}\circ\varphi_{s_{2},s_{1}}\right) \\
&  \leq d_{\psi_{s_{1},s_{0}}\left(  K\right)  }\left(  \psi_{t,s_{2}}%
\circ\psi_{s_{2},s_{1}},\varphi_{t,s_{2}}\circ\psi_{s_{2},s_{1}}\right)
+d_{\psi_{s_{1},s_{0}}\left(  K\right)  }\left(  \varphi_{t,s_{2}}\circ
\psi_{s_{2},s_{1}},\varphi_{t,s_{2}}\circ\varphi_{s_{2},s_{1}}\right) \\
&  \leq d_{\psi_{s_{2},s_{1}}\left(  \psi_{s_{1},s_{0}}\left(  K\right)
\right)  }\left(  \psi_{t,s_{2}},\varphi_{t,s_{2}}\right)  +cd_{K}\left(
\psi_{s_{2},s_{1}},\varphi_{s_{2},s_{1}}\right) \\
&  \leq d_{\psi_{s_{2},s_{0}}\left(  K\right)  }\left(  \psi_{t,s_{2}}%
,\varphi_{t,s_{2}}\right)  +c2C\left\vert s_{2}-s_{1}\right\vert ^{\theta}%
\end{align*}
and continuing in this way, we may show by induction that
\[
d_{K}\left(  \psi_{t,s},\varphi_{t,s}\right)  \leq d_{\psi_{s_{m},s_{0}%
}\left(  K\right)  }\left(  \psi_{t,s_{m}},\varphi_{t,s_{m}}\right)
+c2C\sum_{k=1}^{m}\left\vert s_{k}-s_{k-1}\right\vert ^{\theta}.
\]
Taking $m=n$ and then letting $n\rightarrow\infty$ in the previous equation
shows,%
\begin{align*}
d_{K}\left(  \psi_{t,s},\varphi_{t,s}\right)  \leq &  c2C\sum_{k=1}%
^{n}\left\vert s_{k}-s_{k-1}\right\vert ^{\theta}\\
&  =c\left(  K^{\prime}\right)  2Cn\cdot\left\vert \frac{t-s}{n}\right\vert
^{\theta}\rightarrow0\text{ as }n\rightarrow\infty.
\end{align*}
This shows $\psi_{t,s}=\varphi_{t,s}$ on $K$ and as $K$ was arbitrary, the
uniqueness proof is complete.

\section{Proofs of Proposition \ref{pro.10.35} and Corollary \ref{cor.10.37}%
\label{sec.14}}

Recall that if $\left(  M,g\right)  $ is a Riemannian manifold, then the
Koszul formula for the Levi-Civita Covariant derivative is;%
\begin{align}
2g\left(  \nabla_{X}Y,Z\right)   &  =X\left[  g\left(  Y,Z\right)  \right]
+Y\left[  g\left(  X,Z\right)  \right]  -Z\left[  g\left(  X,Y\right)  \right]
\nonumber\\
&  +g\left(  \left[  X,Y\right]  ,Z\right)  -g\left(  \left[  X,Z\right]
,Y\right)  -g\left(  \left[  Y,Z\right]  ,X\right)  . \label{e.14.1}%
\end{align}

\begin{corollary}
\label{cor.14.1}Suppose that $\left(  M^{d},g\right)  $ is a Riemannian
manifold such that $TM$ is parallelizable and $V:\mathbb{R}^{d}\rightarrow
\Gamma\left(  TM\right)  $ is a dynamical system satisfying Eq. (\ref{e.10.21}%
) as in Proposition \ref{pro.10.35} and $Q\left(  a,b\right)  \in C^{\infty
}\left(  M,\mathbb{R}^{d}\right)  $ is determined by $\left[  V_{a}%
,V_{b}\right]  =V_{Q\left(  a,b\right)  }$ for all $a,b\in\mathbb{R}^{d},$ see
Eq. (\ref{e.10.22}). Then the Levi-Civita covariant derivative, $\nabla,$ on
$TM$ satisfies
\begin{equation}
\nabla_{V_{a}}V_{b}=\frac{1}{2}V_{Q\left(  a,b\right)  -Q_{a}%
^{\operatorname{tr}}b-Q_{b}^{\operatorname{tr}}a}, \label{e.14.2}%
\end{equation}
where $Q_{a}\in C^{\infty}\left(  M,\operatorname*{End}\left(  \mathbb{R}%
^{d}\right)  \right)  $ is defined by $Q_{a}b=Q\left(  a,b\right)  $ and
$Q_{a}^{\operatorname{tr}}$ denotes the matrix transpose of $Q_{a}.$
\end{corollary}

\begin{proof}
Since $g\left(  V_{a},V_{b}\right)  =a\cdot b$ is constant on $M$ for each
$a,b\in\mathbb{R}^{d},$ the Koszul formula in Eq. (\ref{e.14.1}) shows for
$a,b,c\in\mathbb{R}^{d}$ that,%
\begin{align*}
2g\left(  \nabla_{V_{a}}V_{b},V_{c}\right)   &  =g\left(  \left[  V_{a}%
,V_{b}\right]  ,V_{c}\right)  -g\left(  \left[  V_{a},V_{c}\right]
,V_{b}\right)  -g\left(  \left[  V_{b},V_{c}\right]  ,V_{a}\right) \\
&  =g\left(  V_{Q\left(  a,b\right)  },V_{c}\right)  -g\left(  V_{Q\left(
a,c\right)  },V_{b}\right)  -g\left(  V_{Q\left(  b,c\right)  },V_{a}\right)
\\
&  =Q\left(  a,b\right)  \cdot c-Q\left(  a,c\right)  \cdot b-Q\left(
b,c\right)  \cdot a\\
&  =Q\left(  a,b\right)  \cdot c-c\cdot Q_{a}^{\operatorname{tr}}b-c\cdot
Q_{b}^{\operatorname{tr}}a\\
&  =g\left(  V_{Q\left(  a,b\right)  -Q_{a}^{\operatorname{tr}}b-Q_{b}%
^{\operatorname{tr}}a},V_{c}\right)
\end{align*}
from which Eq. (\ref{e.14.2}) follows.
\end{proof}

\begin{proof}
[Proof of Proposition \ref{pro.10.35}]If $f\in C^{\infty}\left(
M,\mathbb{R}^{d}\right)  $ is a bounded function such that $V_{a_{1}}\dots
V_{a_{k}}f$ is bounded for all $a_{j}\in\mathbb{R}^{d}$ and $k\in
\mathbb{N}_{0},$ then for $a\in\mathbb{R}^{d}$ we have
\[
\left[  V_{a},V_{f}\right]  =V_{Q\left(  a,f\right)  }+V_{V_{a}f}%
=V_{V_{a}f+Q\left(  a,f\right)  }.
\]
The assumptions given in Proposition \ref{pro.10.35} imply that $g:=V_{a}%
f+Q\left(  a,f\right)  \in C^{\infty}\left(  M,\mathbb{R}^{d}\right)  $ is a
bounded function such that $V_{a_{1}}\dots V_{a_{k}}g$ is bounded for all
$a_{j}\in\mathbb{R}^{d}$ and $1\leq k\leq\kappa-1.$ We also have from
Corollary \ref{cor.14.1} that
\[
\nabla_{V_{a}}V_{f}=V_{V_{a}f}+\frac{1}{2}V_{Q\left(  a,f\right)
-Q_{a}^{\operatorname{tr}}f-Q_{f}^{\operatorname{tr}}a}=V_{g}%
\]
where
\[
g:=V_{a}f+\frac{1}{2}\left(  Q\left(  a,f\right)  -Q_{a}^{\operatorname{tr}%
}f-Q_{f}^{\operatorname{tr}}a\right)
\]
is again a bounded function such that $V_{a_{1}}\dots V_{a_{k}}g$ is bounded
for all $a_{j}\in\mathbb{R}^{d}$ and $1\leq k\leq\kappa-1.$ Thus it follows by
induction that $L_{V_{a_{1}}}L_{V_{a_{2}}}\dots L_{V_{a_{k-1}}}V_{a_{k}}$ is a
bounded vector field such that $\nabla_{V_{a}}\left(  L_{V_{a_{1}}}%
L_{V_{a_{2}}}\dots L_{V_{a_{k-1}}}V_{a_{k}}\right)  $ is bounded for all
$1\leq k\leq\kappa$ and $a,a_{j}\in\mathbb{R}^{d}.$
\end{proof}

\begin{proof}
[Proof of Corollary \ref{cor.10.37}]We begin by recalling some frame bundle
basics which may be found in \cite[Chapter III]{Kobayashi1963}. Associate to a
tensor field, $G\in\Gamma\left(  T^{\ast}\left(  M\right)  ^{\otimes k}%
\otimes\operatorname*{End}\left(  TM\right)  \right)  ,$ is the function
\[
\tilde{G}:O\left(  M\right)  \rightarrow Hom\left(  \left[  \mathbb{R}%
^{d}\right]  ^{\otimes k},\operatorname*{End}\left(  \mathbb{R}^{d}\right)
\right)
\]
defined by%
\[
\tilde{G}\left(  u\right)  \left[  a_{1}\otimes\dots\otimes a_{k}\right]
=u^{-1}G\left(  ua_{1}\otimes\dots\otimes ua_{k}\right)  u.
\]
If $A\in so\left(  d\right)  ,$ then a chain rule computation shows,
\begin{equation}
\left(  A^{\ast}\tilde{G}\right)  \left(  u\right)  =-\operatorname{ad}%
_{A}\tilde{G}\left(  u\right)  +\tilde{G}\left(  u\right)  \mathcal{D}_{A}
\label{e.14.3}%
\end{equation}
where%
\begin{equation}
\mathcal{D}_{A}\left[  a_{1}\otimes\dots\otimes a_{k}\right]  :=\sum_{j=1}%
^{k}a_{1}\otimes\dots\otimes a_{j-1}\otimes Aa_{j}\otimes a_{j+1}\dots\otimes
a_{k}. \label{e.14.4}%
\end{equation}
Similarly using the definition of $\nabla G$ one also shows $B_{\left(
\cdot\right)  }\tilde{G}=\widetilde{\nabla G}.$ The dynamical system,
$V:\mathbb{R}^{d}\times so\left(  d\right)  \rightarrow\Gamma\left(  TO\left(
M\right)  \right)  $ as in Eq. (\ref{e.10.25}), trivializes $TO\left(
M\right)  $ and we let $g^{O\left(  M\right)  }$ to be the unique Riemannian
metric on $O\left(  M\right)  $ such that $V$ is isometric when $so\left(
d\right)  $ is equipped with the Hilbert-Schmidt inner product defined by
\[
A\cdot C:=\operatorname{tr}\left(  A^{\ast}C\right)  =-\operatorname{tr}%
\left(  AC\right)  .
\]

Again referring to \cite[Chapter III]{Kobayashi1963}, if $A,C\in so\left(
d\right)  $ and $a,c\in\mathbb{R}^{d},$ then the following commutator formulas
hold;%
\[
\left[  A^{\ast},C^{\ast}\right]  =\left[  A,C\right]  ^{\ast},\text{ }\left[
A^{\ast},B_{a}\right]  =B_{Aa},\text{ and }\left[  B_{a},B_{c}\right]
=-\tilde{R}_{a,c}^{\ast},
\]
where as above, $\tilde{R}_{a,c}\left(  u\right)  :=\tilde{R}\left(  u\right)
a\otimes c=u^{-1}R\left(  ua,uc\right)  u.$ Using these commutator formulas it
follows that $\left[  V_{\left(  a,A\right)  },V_{\left(  c,C\right)
}\right]  =V_{Q_{\left(  \left(  a,A\right)  ,\left(  c,C\right)  \right)  }}$
where
\[
Q_{\left(  \left(  a,A\right)  ,\left(  c,C\right)  \right)  }:=\left(
Ac-Ca,\left[  A,C\right]  \right)  -\left(  0,\tilde{R}_{a,c}\right)  .
\]
Moreover, from the derivative formulas in Eqs. (\ref{e.14.3}) and
(\ref{e.14.4}), it follows that%
\begin{equation}
V_{\left(  a_{1},A_{1}\right)  }\dots V_{\left(  a_{k},A_{k}\right)
}Q_{\left(  \left(  a,A\right)  ,\left(  c,C\right)  \right)  } \label{e.14.5}%
\end{equation}
depends linearly on $\left(  \widetilde{R},\widetilde{\nabla R},\dots
,\widetilde{\nabla^{k}R}\right)  .$ Hence the expressions in Eq.
(\ref{e.14.5}) are bounded functions on $O\left(  M\right)  $ for $0\leq
k\leq\kappa-1$ provided $\left\vert \nabla^{j}R\right\vert _{M}<\infty$ for
$0\leq j\leq\kappa-1.$ The proof is then complete by an application of
Proposition \ref{pro.10.35} and Theorem \ref{thm.10.27}.
\end{proof}

\section{Appendix: Variational Estimates\label{sec.15}}

\begin{lemma}
\label{lem.15.1}If $\sigma\in AC\left(  \left[  0,T\right]  ,M\right)  $ and
$o\in M$ is fixed, then $\left[  0,T\right]  \ni t\rightarrow d\left(
o,\sigma\left(  t\right)  \right)  $ is absolutely continuous and moreover,%
\[
\left\vert \frac{d}{dt}d\left(  o,\sigma\left(  t\right)  \right)  \right\vert
\leq\left\vert \dot{\sigma}\left(  t\right)  \right\vert \text{ for a.e. }t.
\]

\end{lemma}

\begin{proof}
By the (reverse) triangle equality for metrics and the definition of $d$ in
terms of minimization over absolutely continuous paths we find,%
\[
\left\vert d\left(  o,\sigma\left(  t\right)  \right)  -d\left(
o,\sigma\left(  s\right)  \right)  \right\vert \leq d\left(  \sigma\left(
s\right)  ,\sigma\left(  t\right)  \right)  \leq\ell_{g}\left(  \sigma
|_{J\left(  s,t\right)  }\right)  \leq\int_{s}^{t}\left\vert \dot{\sigma
}\left(  \tau\right)  \right\vert d\tau.
\]
This suffices to prove the claimed results in the lemma.
\end{proof}

\begin{corollary}
\label{cor.15.2}Suppose that $\rho:[0,\infty)\rightarrow\left(  0,\infty
\right)  $ is a continuous function,
\[
G\left(  s\right)  :=\int_{0}^{s}\frac{1}{\rho\left(  \sigma\right)  }%
d\sigma,
\]
$o\in M$ is fixed, and $\sigma\in AC\left(  \left[  0,T\right]  ,M\right)  $
satisfies
\begin{equation}
\left\vert \dot{\sigma}\left(  t\right)  \right\vert \leq\rho\left(  d\left(
o,\sigma\left(  t\right)  \right)  \right)  \text{ for a.e. }t\in\left[
0,T\right]  , \label{e.15.1}%
\end{equation}
then
\begin{equation}
G\left(  d\left(  o,\sigma\left(  t\right)  \right)  \right)  \leq G\left(
d\left(  o,\sigma\left(  0\right)  \right)  \right)  +t. \label{e.15.2}%
\end{equation}

\end{corollary}

\begin{proof}
By Lemma \ref{lem.15.1} and the assumption in Eq. (\ref{e.15.1}) we have,%
\begin{equation}
\frac{d}{d\tau}d\left(  o,\sigma\left(  \tau\right)  \right)  \leq\left\vert
\frac{d}{d\tau}d\left(  o,\sigma\left(  \tau\right)  \right)  \right\vert
\leq\left\vert \dot{\sigma}\left(  \tau\right)  \right\vert \leq\rho\left(
d\left(  o,\sigma\left(  \tau\right)  \right)  \right)  \text{ for a.e. }\tau.
\label{e.15.3}%
\end{equation}
Since $G$ is locally Lipschitz, $\tau\rightarrow G\left(  d\left(
o,\sigma\left(  \tau\right)  \right)  \right)  $ is still absolutely
continuous and
\begin{align*}
\frac{d}{d\tau}G\left(  d\left(  o,\sigma\left(  \tau\right)  \right)
\right)   &  =G^{\prime}\left(  d\left(  o,\sigma\left(  \tau\right)  \right)
\right)  \frac{d}{d\tau}d\left(  o,\sigma\left(  \tau\right)  \right) \\
&  =\frac{1}{\rho\left(  d\left(  o,\sigma\left(  \tau\right)  \right)
\right)  }\frac{d}{d\tau}d\left(  o,\sigma\left(  \tau\right)  \right)  \text{
for a.e. }\tau.
\end{align*}
Thus dividing Eq. (\ref{e.15.3}) by $\rho\left(  d\left(  o,\sigma\left(
\tau\right)  \right)  \right)  $ and integrating the result shows%
\begin{align*}
G\left(  d\left(  o,\sigma\left(  t\right)  \right)  \right)   &  -G\left(
d\left(  o,\sigma\left(  0\right)  \right)  \right) \\
&  =\int_{0}^{t}\frac{1}{\rho\left(  d\left(  o,\sigma\left(  \tau\right)
\right)  \right)  }\frac{d}{d\tau}d\left(  o,\sigma\left(  \tau\right)
\right)  d\tau\leq\int_{0}^{t}1d\tau=t
\end{align*}
from which the result follows.
\end{proof}

\begin{lemma}
\label{lem.15.3}If $\left(  M,g\right)  $ is a complete Riemannian manifold
and $\rho$ and $G$ are as in Corollary \ref{cor.15.2} such that
\[
\lim_{s~\uparrow\infty}G\left(  s\right)  =\int_{0}^{\infty}\frac{1}%
{\rho\left(  \sigma\right)  }d\sigma=\infty
\]
which guarantees that $G:[0,\infty)\rightarrow\lbrack0,\infty)$ is bijective.
If $Y\in\Gamma\left(  TM\right)  $ satisfies the bound $\left\vert Y\left(
m\right)  \right\vert _{g}\leq\rho\left(  d\left(  o,m\right)  \right)  $ for
all $m\in M,$ then $Y$ is complete and we have the estimate,%
\begin{equation}
d\left(  o,e^{tY}\left(  m\right)  \right)  \leq G^{-1}\left(  G\left(
d\left(  o,m\right)  \right)  +t\right)  \text{ }\forall~t\geq0.
\label{e.15.4}%
\end{equation}

\end{lemma}

\begin{proof}
Let $\sigma:[0,T)\rightarrow M$ denote a solution to $\dot{\sigma}\left(
t\right)  =Y\left(  \sigma\left(  t\right)  \right)  .$ Then by Corollary
\ref{cor.15.2},%
\[
G\left(  d\left(  o,\sigma\left(  t\right)  \right)  \right)  \leq G\left(
d\left(  o,\sigma\left(  0\right)  \right)  \right)  +t
\]
which is equivalent to Eq. (\ref{e.15.4}) since $G$ is bijective and
increasing and bijective. From this expression we see that $\sigma\left(
t\right)  $ can not explode and this shows $Y$ is complete. The completeness
of $\left(  M,g\right)  $ is used here to imply closed bounded sets are compact.
\end{proof}

There are many other possible examples for $\rho.$ We will give two here.

\begin{example}
[Gronwall Estimates]\label{ex.15.4}The standard example for $\rho$ is
$\rho\left(  \sigma\right)  =C\left(  1+\sigma\right)  $ for some $C<\infty$
in which case
\begin{align*}
G\left(  s\right)   &  =C^{-1}\ln\left(  1+s\right)  ,\\
G^{-1}\left(  u\right)   &  =e^{Cu}-1,
\end{align*}
and%
\begin{align*}
G^{-1}\left(  G\left(  s\right)  +t\right)   &  =G^{-1}\left(  C^{-1}%
\ln\left(  1+s\right)  +t\right) \\
&  =e^{C\left(  C^{-1}\ln\left(  1+s\right)  +t\right)  }-1=\left(
1+s\right)  e^{Ct}-1.
\end{align*}
Thus the estimate in Eq. (\ref{e.15.4}) becomes the usual Bellman-Gronwall
type estimate,%
\[
d\left(  o,e^{tY}\left(  m\right)  \right)  \leq\left(  1+d\left(  o,m\right)
\right)  e^{Ct}-1\text{ }\forall~t\geq0.
\]

\end{example}

\begin{example}
[Double Exponential Growth Estiamtes]\label{ex.15.5}In this example let us
suppose that%
\[
\rho\left(  \sigma\right)  =C\left(  1+\sigma\right)  \left(  1+\ln\left(
1+\sigma\right)  \right)  .
\]
After making the change of variables, $u=\ln\left(  1+\sigma\right)  ,$ one
shows
\begin{align*}
G\left(  s\right)   &  =C^{-1}\ln\left(  1+\ln\left(  1+s\right)  \right)
,\text{ and}\\
G^{-1}\left(  u\right)   &  =\exp\left(  \exp\left(  Cu\right)  -1\right)  -1.
\end{align*}
Using these expressions the estimate in Eq. (\ref{e.15.4}) becomes,
\[
d\left(  o,e^{tY}\left(  m\right)  \right)  \leq e^{\left[  e^{Ct}-1\right]
}\left(  1+d\left(  o,m\right)  \right)  ^{e^{Ct}}-1.
\]

\end{example}

Our next goal is to understand the previous estimates in terms of conformal
changes of the Riemannian metric $g.$ We begin with the following basic lemma.

\begin{lemma}
\label{lem.15.6}Suppose $\left(  M,g,o\right)  $ is a pointed Riemannian
manifold, $\rho\in C^{1}\left(  [0,\infty),\left(  0,\infty\right)  \right)
,$ and $\tilde{g}$ is the continuous metric on $M$ defined by%
\[
\tilde{g}\left(  v,w\right)  =\rho^{2}\left(  d\left(  o,m\right)  \right)
g\left(  v,w\right)  \text{ for all }v,w\in T_{m}M\text{ and }m\in M.
\]
If $p\in M$ and there exists a $g$-length minimizing geodesic, $\sigma:\left[
0,1\right]  \rightarrow M,$ joining $o$ to $p,$ then%
\[
d_{\tilde{g}}\left(  o,p\right)  =\int_{0}^{d\left(  o,p\right)  }\rho\left(
u\right)  du
\]
and moreover, $d_{\tilde{g}}\left(  o,p\right)  =\ell_{\tilde{g}}\left(
\sigma\right)  ,$ where for any $\gamma\in AC\left(  \left[  0,1\right]
,M\right)  ,$%
\[
\ell_{\tilde{g}}\left(  \gamma\right)  :=\int_{0}^{1}\sqrt{\tilde{g}\left(
\dot{\gamma}\left(  t\right)  ,\dot{\gamma}\left(  t\right)  \right)  }%
dt=\int_{0}^{1}\rho\left(  d\left(  o,\gamma\left(  t\right)  \right)
\right)  \left\vert \dot{\gamma}\left(  t\right)  \right\vert _{g}dt.
\]

\end{lemma}

\begin{proof}
Let $AC_{o,p}\left(  \left[  0,1\right]  ,M\right)  $ denote those paths,
$\gamma\in AC\left(  \left[  0,1\right]  ,M\right)  ,$ such that
$\gamma\left(  0\right)  =o$ and $\gamma\left(  1\right)  =p.$ For any
$\gamma\in AC_{o,p}\left(  \left[  0,1\right]  ,M\right)  ,$ we have
\begin{align*}
\ell_{\tilde{g}}\left(  \gamma\right)   &  =\int_{0}^{1}\rho\left(  d\left(
o,\gamma\left(  t\right)  \right)  \right)  \left\vert \dot{\gamma}\left(
t\right)  \right\vert _{g}dt\\
&  \geq\int_{0}^{1}\rho\left(  d\left(  o,\gamma\left(  t\right)  \right)
\right)  \frac{d}{dt}d\left(  o,\gamma\left(  t\right)  \right)  dt=\int%
_{0}^{d\left(  o,p\right)  }\rho\left(  u\right)  du,
\end{align*}
wherein the inequality is a result of Lemma \ref{lem.15.1}. Thus we conclude
that%
\[
\tilde{d}\left(  o,p\right)  =\inf\left\{  \ell_{\tilde{g}}\left(
\gamma\right)  :\gamma\in AC_{o,p}\left(  \left[  0,1\right]  ,M\right)
\right\}  \geq\int_{0}^{d\left(  o,p\right)  }\rho\left(  u\right)  du.
\]
If we now further assume that there exists a $g$-length minimizing geodesic,
$\sigma\left(  t\right)  ,$ joining $o$ to $p,$ then $\left\vert \dot{\sigma
}\left(  t\right)  \right\vert _{g}=d\left(  o,p\right)  $ and $d\left(
o,\sigma\left(  t\right)  \right)  =td\left(  o,p\right)  $ for $0\leq t\leq1$
and so for this $\sigma,$%
\begin{align*}
\ell_{\tilde{g}}\left(  \sigma\right)   &  =\int_{0}^{1}\rho\left(  d\left(
o,\sigma\left(  t\right)  \right)  \right)  \left\vert \dot{\sigma}\left(
t\right)  \right\vert _{g}dt\\
&  =\int_{0}^{1}\rho\left(  td\left(  o,p\right)  \right)  d\left(
o,p\right)  dt=\int_{0}^{d\left(  o,p\right)  }\rho\left(  u\right)  du
\end{align*}
which completes the proof.
\end{proof}

Our next goal is to smooth out the conformal factor, $\rho^{2}\left(  d\left(
o,m\right)  \right)  ,$ used Lemma \ref{lem.15.6} by replacing $d\left(
o,\cdot\right)  $ by an appropriate smooth approximation. The next lemma
explains how one can do this in Euclidean space.

\begin{lemma}
\label{lem.15.7}For every $\varepsilon>0,$ let $\left\vert x\right\vert
_{\varepsilon}:=\sqrt{\left\vert x\right\vert ^{2}+\varepsilon^{2}}$ for
$x\in\mathbb{R}^{d}.$ Then $\left\vert \cdot\right\vert _{\varepsilon}\in
C^{\infty}\left(  \mathbb{R}^{d}\rightarrow\lbrack0,\infty)\right)  $
satisfies, for all $x,y\in\mathbb{R}^{d},$ the estimates;%
\begin{align}
&  \left\vert x\right\vert \leq\left\vert x\right\vert _{\varepsilon}%
\leq\left\vert x\right\vert +\varepsilon,\text{ and}\label{e.15.5}\\
&  \left\vert \left\vert x\right\vert _{\varepsilon}-\left\vert y\right\vert
_{\varepsilon}\right\vert \leq\left\vert x-y\right\vert \text{.}
\label{e.15.6}%
\end{align}

\begin{proof}
Equation (\ref{e.15.5}) easily follows from the fact that $\sqrt{a+b}\leq
\sqrt{a}+\sqrt{b}$ for $a,b\geq0.$ If we let $g_{\varepsilon}\left(  t\right)
:=\sqrt{t^{2}+\varepsilon^{2}}$ for $t\in\mathbb{R},$ then%
\[
\left\vert g_{\varepsilon}\left(  t\right)  -g_{\varepsilon}\left(  s\right)
\right\vert \leq\left\vert \int_{s}^{t}\left\vert \dot{g}_{\varepsilon}\left(
\tau\right)  \right\vert d\tau\right\vert =\left\vert \int_{s}^{t}\left\vert
\frac{\tau}{\sqrt{\tau^{2}+\varepsilon^{2}}}\right\vert d\tau\right\vert
\leq\left\vert t-s\right\vert
\]
for all $s,t\in\lbrack0,\infty).$ Taking $t=\left\vert x\right\vert $ and
$s=\left\vert y\right\vert $ in this inequality shows
\[
\left\vert \left\vert x\right\vert _{\varepsilon}-\left\vert y\right\vert
_{\varepsilon}\right\vert \leq\left\vert \left\vert x\right\vert -\left\vert
y\right\vert \right\vert \leq\left\vert x-y\right\vert
\]
which verifies Eq. (\ref{e.15.6}).
\end{proof}
\end{lemma}

According to Greene and Wu \cite[Section 2]{Greene1979}, \cite{Matveev2013},
and \cite[Theorem 1]{Azagra2007a} the analogue of Lemma \ref{lem.15.7} is
valid on any Riemannian manifold $\left(  M,g\right)  .$ What is shown in
these references is that for any $r>0$ and $\delta\in C\left(  M,\left(
0,\infty\right)  \right)  ,$ there exists $f\in C^{\infty}\left(
M,\mathbb{R}\right)  $ such that $\operatorname{Lip}\left(  f\right)  \leq1+r$
and
\[
\left\vert d\left(  o,p\right)  -f\left(  p\right)  \right\vert \leq
\delta\left(  p\right)  \text{ for all }p\in M.
\]
If we now fix $\varepsilon>0$ and take $\delta\left(  p\right)  =\frac{1}%
{2}\varepsilon$ and $h\left(  p\right)  :=f\left(  p\right)  +\frac{1}%
{2}\varepsilon,$ then $\operatorname{Lip}\left(  h\right)  =\operatorname{Lip}%
\left(  f\right)  \leq1+r$ and the above displayed inequality is equivalent
to
\[
d\left(  o,p\right)  \leq h\left(  p\right)  \leq d\left(  o,p\right)
+\varepsilon\text{ for all }p\in M.
\]
We now fix $\varepsilon,r>0$ and $h\in C^{\infty}\left(  M,\mathbb{R}\right)
$ so that%
\begin{equation}
d\left(  o,p\right)  \leq h\left(  p\right)  \leq d\left(  o,p\right)
+\varepsilon\text{ and }\operatorname{Lip}\left(  h\right)  \leq1+r.
\label{e.15.7}%
\end{equation}

\begin{proposition}
[Conformal change of metrics]\label{pro.15.8}If $\left(  M,g\right)  $ is a
Riemannian manifold with Levi-Civita derivative, $\nabla,$ and $\bar
{g}=e^{2\varphi}g,$ then the associated Levi-Civita covariant derivative,
$\bar{\nabla},$ is given by
\begin{equation}
\bar{\nabla}_{X}Y=\nabla_{X}Y+X\varphi\cdot Y+Y\varphi\cdot X-g\left(
X,Y\right)  \nabla\varphi\label{e.15.8}%
\end{equation}
where $\nabla\varphi$ is being used to denote the gradient of $\varphi$
relative to the metric $g.$
\end{proposition}

\begin{proof}
Using
\[
X\left[  \bar{g}\left(  Y,Z\right)  \right]  =X\left[  e^{2\varphi}g\left(
X,Y\right)  \right]  =2\left[  X\varphi\right]  \bar{g}\left(  Y,Z\right)
+e^{2\varphi}X\left[  g\left(  Y,Z\right)  \right]
\]
and similar expressions for $Y\left[  \bar{g}(X,Z)\right]  $ and $Z\left[
\bar{g}(X,Y)\right]  ,$ the Koszul formula $\bar{\nabla}_{X}Y$ and $\nabla
_{X}Y$ may be used to show,%
\begin{align*}
2\bar{g}\left(  \bar{\nabla}_{X}Y,Z\right)   &  =e^{2\varphi}2g\left(
\nabla_{X}Y,Z\right)  +2\left[  X\varphi\right]  \bar{g}\left(  Y,Z\right)
+2\left[  Y\varphi\right]  \bar{g}\left(  X,Z\right)  -\left[  Z\varphi
\right]  \bar{g}\left(  X,Y\right) \\
&  =2\bar{g}\left(  \nabla_{X}Y,Z\right)  +2\left[  X\varphi\right]  \bar
{g}\left(  Y,Z\right)  +2\left[  Y\varphi\right]  \bar{g}\left(  X,Z\right)
-g\left(  Z,\nabla\varphi\right)  e^{2\varphi}g\left(  X,Y\right) \\
&  =2\bar{g}\left(  \nabla_{X}Y,Z\right)  +2\left[  X\varphi\right]  \bar
{g}\left(  Y,Z\right)  +2\left[  Y\varphi\right]  \bar{g}\left(  X,Z\right)
-g\left(  X,Y\right)  \bar{g}\left(  \nabla\varphi,Z\right)  .
\end{align*}
which easily implies Eq. (\ref{e.15.8}).
\end{proof}

\begin{corollary}
\label{cor.15.9}Let $\varepsilon,r>0$ be given and choose $h\in C^{\infty
}\left(  M,[0,\infty)\right)  $ such that Eq. (\ref{e.15.7}) holds and define
$\bar{g}$ to be the metric on $M$ given by%
\[
\bar{g}\left(  v,w\right)  =\left[  \frac{1}{1+h\left(  m\right)  }\right]
^{2}g\left(  v,w\right)  \text{ for }v,w\in T_{m}M\text{ and }m\in M.
\]
If $\left(  M,g\right)  $ is a complete Riemannian manifold, then;

\begin{enumerate}
\item for all $p\in M,$%
\begin{equation}
\frac{1}{1+\varepsilon}\ln\left(  1+d\left(  o,p\right)  \right)  \leq\bar
{d}\left(  o,p\right)  \leq\ln\left(  1+d\left(  o,p\right)  \right)  ,
\label{e.15.11}%
\end{equation}

\item and $\left(  M,\bar{g}\right)  $ is still a complete Riemannian manifold.
\end{enumerate}
\end{corollary}

\begin{proof}
Let
\[
\tilde{g}\left(  v,w\right)  =\left[  \frac{1}{1+d\left(  o,p\right)
}\right]  ^{2}g\left(  v,w\right)  ~\text{ for }v,w\in T_{p}M\text{ and }p\in
M.
\]
From Eq. (\ref{e.15.7}) it follows that%
\[
\frac{1}{1+d\left(  o,p\right)  }\frac{1}{1+\varepsilon}\leq\frac
{1}{1+d\left(  o,p\right)  +\varepsilon}\leq\frac{1}{1+h\left(  p\right)
}\leq\frac{1}{1+d\left(  o,p\right)  }\text{ }\forall~p\in M
\]
and therefore%
\[
\frac{1}{1+\varepsilon}\sqrt{\tilde{g}}\leq\sqrt{\bar{g}}\leq\tilde{g}%
\]
and hence
\[
\frac{1}{\sqrt{1+\varepsilon}}\tilde{d}\leq\bar{d}\leq\tilde{d}%
\]
where $\tilde{d}\left(  o,p\right)  =\ln\left(  1+d\left(  0,p\right)
\right)  $ in this case. This proves item (1).

For the proof of item (2), we observe that Eq. (\ref{e.15.11}) shows that $d$
and $\bar{d}$ have the same notions of bounded sets. As $d$ and $\bar{d}$
generate the same topology, $d$ and $\bar{d}$ also have the same notions of
closed sets. Thus a subset, $A\subset M$ is $\bar{d}$-closed and bounded iff
$A$ is $d$-closed and bounded iff $A$ is compact since $\left(  M,g\right)  $
is complete. Combining all of these facts shows $\left(  M,\bar{g}\right)  $
is still complete.
\end{proof}

\begin{proposition}
\label{pro.15.10}Let us continue the notation used in Corollary \ref{cor.15.9}
and let $Y\in\Gamma\left(  TM\right)  .$ Then $Y$ is $\bar{g}$ bounded iff
there exists $\bar{C}<\infty$ such that
\begin{equation}
\left\vert Y\left(  m\right)  \right\vert _{g}\leq\bar{C}\left(  1+d\left(
o,m\right)  \right)  \text{ }\forall~m\in M. \label{e.15.12}%
\end{equation}
Let us now suppose that $Y$ satisfies the estimate in Eq. (\ref{e.15.12}).
Under this assumption, we have
\begin{equation}
\left\vert \bar{\nabla}Y\right\vert _{\bar{g},M}:=\sup_{\left\vert
v\right\vert _{\bar{g}}=1}\left\vert \nabla_{v}Y\right\vert _{\bar{g}}%
<\infty\label{e.15.13}%
\end{equation}
iff
\begin{equation}
\left\vert \nabla Y\right\vert _{g,M}:=\sup_{\left\vert u\right\vert _{g}%
=1}\left\vert \nabla_{u}Y\right\vert _{g}<\infty. \label{e.15.14}%
\end{equation}
Moreover, if $\left\vert \nabla Y\right\vert _{g,M}<\infty,$ then both
estimates in Eq. (\ref{e.15.12}) and (\ref{e.15.13}) hold.
\end{proposition}

\begin{proof}
Suppose that $Y\in\Gamma\left(  TM\right)  ,$ then $Y$ is $\bar{g}$ bounded
iff
\begin{equation}
\infty>C:=\sup_{m\in M}\left\vert Y\left(  m\right)  \right\vert _{\bar{g}%
}=\sup_{m\in M}\left(  \frac{1}{1+h\left(  m\right)  }\left\vert Y\left(
m\right)  \right\vert _{g}\right)  , \label{e.15.15}%
\end{equation}
i.e. iff there exists $C<\infty$ such that
\begin{equation}
\left\vert Y\left(  m\right)  \right\vert _{g}\leq C\left(  1+h\left(
m\right)  \right)  \text{ }\forall~m\in M \label{e.15.16}%
\end{equation}
and it is easily seen that this is equivalent to the existence of $\bar
{C}<\infty$ such that Eq. (\ref{e.15.12}) holds.

Now lets assume $Y\in\Gamma\left(  TM\right)  $ satisfies Eq. (\ref{e.15.12})
or equivalently Eq. (\ref{e.15.16}) for some $C<\infty.$ Let us note for
$v_{m}\in T_{m}M$ that
\[
1=\left\vert v\right\vert _{\bar{g}}=\frac{1}{1+h\left(  m\right)  }\left\vert
v_{m}\right\vert _{g}\iff\left\vert v_{m}\right\vert _{g}=1+h\left(  m\right)
.
\]
Thus if we let
\[
u_{m}:=\frac{1}{1+h\left(  m\right)  }v_{m},
\]
then $\left\vert u_{m}\right\vert _{g}=1$ and%
\[
\left\vert \bar{\nabla}_{v_{m}}Y\right\vert _{\bar{g}}=\frac{1}{1+h\left(
m\right)  }\left\vert \bar{\nabla}_{v_{m}}Y\right\vert _{g}=\left\vert
\bar{\nabla}_{u_{m}}Y\right\vert _{g}%
\]
from which it follows that
\[
\left\vert \bar{\nabla}Y\right\vert _{\bar{g},M}=\sup_{\left\vert u\right\vert
_{g}=1}\left\vert \bar{\nabla}_{u_{m}}Y\right\vert _{g}.
\]

Next we use Proposition \ref{pro.15.8} with
\[
\varphi=\ln\left(  \frac{1}{1+h}\right)  =-\ln\left(  1+h\right)
\]
and
\[
w_{m}\varphi=-\frac{1}{1+h\left(  m\right)  }w_{m}h=-\frac{1}{1+h\left(
m\right)  }g\left(  \nabla h\left(  m\right)  ,w_{m}\right)
\]
in order to see that
\begin{align*}
\bar{\nabla}_{u_{m}}Y  &  =\nabla_{u_{m}}Y+u_{m}\varphi\cdot Y\left(
m\right)  +\left(  Y\varphi\right)  \left(  m\right)  \cdot u_{m}-g\left(
u_{m},Y\left(  m\right)  \right)  \nabla\varphi\left(  m\right) \\
&  =\nabla_{u_{m}}Y-K\left(  u_{m},Y\left(  m\right)  \right)
\end{align*}
where
\[
K\left(  u_{m},Y\left(  m\right)  \right)  =\frac{1}{1+h\left(  m\right)
}\left[  \left(  u_{m}h\right)  Y\left(  m\right)  +\left(  Yh\right)  \left(
m\right)  \cdot u_{m}-g\left(  u_{m},Y\left(  m\right)  \right)  \nabla
h\left(  m\right)  \right]
\]
which satisfies,%
\[
\left\vert K\left(  u_{m},Y\left(  m\right)  \right)  \right\vert \leq\left(
1+r\right)  3C,
\]
with $C$ as in Eq. (\ref{e.15.15}). Thus we see, under the condition that
$C<\infty,$ that%
\[
\infty>\left\vert \bar{\nabla}Y\right\vert _{\bar{g},M}=\sup_{\left\vert
u\right\vert _{g}=1}\left\vert \bar{\nabla}_{u_{m}}Y\right\vert _{g}\iff
\sup_{\left\vert u\right\vert _{g}=1}\left\vert \nabla_{u_{m}}Y\right\vert
<\infty.
\]

Conversely if $\left\vert \nabla Y\right\vert _{g,M}:=\sup_{\left\vert
u\right\vert _{g}=1}\left\vert \nabla_{u_{m}}Y\right\vert <\infty,$ then for
$m\in M$ and $\sigma\left(  t\right)  $ an absolutely continuous curve joining
$o$ to $m$ we have%
\begin{align*}
\left\vert Y\left(  m\right)  \right\vert _{g}  &  =\left\vert \pt_{1}\left(
\sigma\right)  ^{-1}Y\left(  m\right)  \right\vert _{g}=\left\vert
\pt_{1}\left(  \sigma\right)  ^{-1}Y\left(  \sigma\left(  1\right)  \right)
\right\vert _{g}\\
&  =\left\vert Y\left(  o\right)  +\int_{0}^{1}\pt_{1}\left(  \sigma\right)
^{-1}\nabla_{\dot{\sigma}\left(  t\right)  }Y~dt\right\vert _{g}\leq\left\vert
Y\left(  o\right)  \right\vert _{g}+\int_{0}^{1}\left\vert \nabla_{\dot
{\sigma}\left(  t\right)  }Y\right\vert _{g}~dt\\
&  \leq\left\vert Y\left(  o\right)  \right\vert _{g}+\left\vert \nabla
Y\right\vert _{g,M}\int_{0}^{1}\left\vert \dot{\sigma}\left(  t\right)
\right\vert ~dt=\left\vert Y\left(  o\right)  \right\vert _{g}+\left\vert
\nabla Y\right\vert _{g,M}\ell_{g}\left(  \sigma\right)  .
\end{align*}
Taking the infimum over all such paths shows that
\[
\left\vert Y\left(  m\right)  \right\vert _{g}\leq\left\vert Y\left(
o\right)  \right\vert _{g}+\left\vert \nabla Y\right\vert _{g,M}d_{g}\left(
o,m\right)  .
\]
and we have shown that in fact $\left\vert \nabla Y\right\vert _{g,M}<\infty$
implies both Eq. (\ref{e.15.12}) and (\ref{e.15.13}) hold.
\end{proof}

\section{Appendix: Rough Path Basics\label{sec.16}}

Let $0<T<\infty,$ $x\left(  \cdot\right)  \in C^{1}\left(  \left[  0,T\right]
,\mathbb{R}^{d}\right)  ,$ and for $s,t\in\left[  0,T\right]  $ let
$X_{s,t}\in G_{\text{geo}}^{\left(  \kappa\right)  }\left(  \mathbb{R}%
^{d}\right)  $ denote the solution to the ODE,%
\begin{equation}
\frac{d}{dt}X_{s,t}=X_{s,t}\dot{x}\left(  t\right)  \text{ with }X_{s,s}=1.
\label{e.16.1}%
\end{equation}
If we let $g\left(  t\right)  :=X_{0,t},$ then
\[
\dot{g}\left(  t\right)  =g\left(  t\right)  \dot{x}\left(  t\right)  \text{
with }g\left(  0\right)  =1
\]
and we easily see that $X_{s,t}=g\left(  s\right)  ^{-1}g\left(  t\right)  $
and in particular, if $s,t,u\in\left[  0,T\right]  ,$ then
\[
X_{s,t}X_{t,u}=g\left(  s\right)  ^{-1}g\left(  t\right)  g\left(  t\right)
^{-1}g\left(  u\right)  =g\left(  s\right)  ^{-1}g\left(  u\right)  =X_{s,u}.
\]

\begin{proposition}
\label{pro.16.1}If $X_{s,t}^{\left(  \kappa\right)  }\in G_{\text{geo}%
}^{\left(  \kappa\right)  }\left(  \mathbb{R}^{d}\right)  $ is as in Eq.
(\ref{e.16.1}) and for all $k\in\mathbb{N}$ we let
\begin{equation}
X_{s,t}^{k}:=\int_{s}^{t}dt_{k}\int_{s}^{t_{k}}dt_{k-1}\dots\int_{s}^{t_{2}%
}dt_{1}~\dot{x}\left(  t_{1}\right)  \otimes\dots\otimes\dot{x}\left(
t_{k}\right)  , \label{e.16.2}%
\end{equation}
then
\begin{equation}
X_{s,t}^{\left(  \kappa\right)  }=1+\sum_{k=1}^{\kappa}X_{s,t}^{k}.
\label{e.16.3}%
\end{equation}
We may also write $X_{st}^{k}$ as%
\begin{equation}
X_{st}^{k}=\int_{s\leq t_{1}\leq t_{2}\leq\dots\leq t_{k}\leq t}dx\left(
t_{1}\right)  dx\left(  t_{2}\right)  \dots dx\left(  t_{k}\right)  \text{ if
}t\geq s \label{e.16.4}%
\end{equation}
and%
\begin{equation}
X_{st}^{k}=\left(  -1\right)  ^{k}\int_{t\leq t_{1}\leq t_{2}\leq\dots\leq
t_{k-1}\leq t_{k}\leq s}dx\left(  t_{k}\right)  dx\left(  t_{k-1}\right)
\dots dx\left(  t_{1}\right)  \text{ if }t\leq s. \label{e.16.5}%
\end{equation}

\end{proposition}

\begin{proof}
We will show, for all $0\leq m\leq\kappa,$ that%
\begin{equation}
X_{s,t}=1+\sum_{k=1}^{m}\int_{s}^{t}dt_{1}\int_{s}^{t_{1}}dt_{2}\dots\int%
_{s}^{t_{k-1}}dt_{k}~\dot{x}\left(  t_{k}\right)  \dots\dot{x}\left(
t_{1}\right)  +R_{m}\left(  s,t\right)  \label{e.16.6}%
\end{equation}
where%
\begin{equation}
R_{m}\left(  s,t\right)  =\int_{s}^{t}dt_{1}\int_{s}^{t_{1}}dt_{2}\dots
\int_{s}^{t_{m-1}}dt_{m}\int_{s}^{t_{m}}dt_{m+1}X_{s,t_{m+1}}\dot{x}\left(
t_{m+1}\right)  \dot{x}\left(  t_{m}\right)  \dots\dot{x}\left(  t_{1}\right)
\label{e.16.7}%
\end{equation}
and by convention, $\sum_{k=1}^{0}\left(  \dots\right)  \equiv0.$ For $m=0,$
Eq. (\ref{e.16.6}) reads,%
\[
X_{s,t}=1+R_{0}\left(  s,t\right)  =1+\int_{s}^{t}dt_{1}~X_{s,t_{1}}\dot
{x}\left(  t_{1}\right)
\]
which holds true because the fundamental theorem of calculus combined with the
ODE in Eq. (\ref{e.16.1}). To complete the inductive proof we use the
identity,%
\[
X_{s,t_{m+1}}=1+\int_{s}^{t_{m+1}}dt_{m+2}~X_{s,t_{m+2}}\dot{x}\left(
t_{m+2}\right)
\]
in Eq. (\ref{e.16.7}) to find,%
\begin{align*}
R_{m}\left(  s,t\right)  =  &  \int_{s}^{t}dt_{1}\int_{s}^{t_{1}}dt_{2}%
\dots\int_{s}^{t_{m-1}}dt_{m}\int_{s}^{t_{m}}dt_{m+1}\dot{x}\left(
t_{m+1}\right)  \dot{x}\left(  t_{m}\right)  \dots\dot{x}\left(  t_{1}\right)
\\
&  +R_{m+1}\left(  s,t\right)  .
\end{align*}

Taking $m=\kappa$ in Eq. (\ref{e.16.6}) while using $R_{\kappa}\left(
s,t\right)  \equiv0$ in $T^{\left(  \kappa\right)  }\left(  \mathbb{R}%
^{d}\right)  ,$ gives,%
\[
X_{s,t}=1+\sum_{k=1}^{\kappa}\int_{s}^{t}dt_{1}\int_{s}^{t_{1}}dt_{2}\dots
\int_{s}^{t_{k-1}}dt_{k}~\dot{x}\left(  t_{k}\right)  \dots\dot{x}\left(
t_{1}\right)
\]
and then relabeling the $\left(  t_{1},\dots,t_{k}\right)  $ to $\left(
t_{k},\dots,t_{1}\right)  $ in each term gives Eq. (\ref{e.16.3}). The
identities in Eqs. (\ref{e.16.4}) and (\ref{e.16.5}) are fairly simple
rewrites of Eq. (\ref{e.16.2}) For example if $t\leq s,$ the limits in each of
the iterated integrals go from larger times to smaller time and so switching
each of these limits gives rise to the factor $\left(  -1\right)  ^{k}.$ The
relationship between all of times $\left(  t_{1},\dots,t_{k}\right)  $ when
$t\leq s$ are $t\leq t_{k}\leq t_{k-1}\leq\dots\leq t_{1}\leq s$ and so%
\[
X_{s,t}^{k}=\left(  -1\right)  ^{k}\int_{t\leq t_{k}\leq t_{k-1}\leq\dots\leq
t_{2}\leq t_{1}\leq s}dx\left(  t_{1}\right)  dx\left(  t_{2}\right)  \dots
dx\left(  t_{k}\right)  .
\]
Lastly relabeling the $\left(  t_{1},\dots,t_{k}\right)  $ to $\left(
t_{k},\dots,t_{1}\right)  $ in this identity gives the second identity in Eq.
(\ref{e.16.4}).
\end{proof}

\begin{notation}
\label{not.16.2}For $1\leq k<\infty,$ let $\sigma_{k}:\left[  \mathbb{R}%
^{d}\right]  ^{\otimes k}\rightarrow\left[  \mathbb{R}^{d}\right]  ^{\otimes
k}$ be the isometric isomorphism uniquely determined by
\[
\sigma_{k}\left[  v_{1}\otimes v_{2}\otimes\dots\otimes v_{k}\right]
=v_{k}\otimes\dots\otimes v_{2}\otimes v_{1}\text{ }\forall~v_{1},v_{2}%
,\dots,v_{k}\in\mathbb{R}^{d}.
\]

\end{notation}

\begin{corollary}
\label{cor.16.3}If $0\leq s\leq t\leq T$ and $1\leq k\leq\kappa,$ then
$X_{t,s}^{k}=\left(  -1\right)  ^{k}\sigma_{k}X_{s,t}^{k}$ or equivalently
stated,
\[
X_{t,s}=1+\sum_{k=1}^{\kappa}\left(  -1\right)  ^{k}\sigma_{k}X_{s,t}^{k}.
\]

\end{corollary}

\begin{proof}
From Eqs. (\ref{e.16.4}) and (\ref{e.16.5}),%
\[
X_{s,t}=1+\sum_{k=1}^{\kappa}\int_{s\leq t_{k}\leq t_{k-1}\leq\dots\leq
t_{1}\leq t}dx\left(  t_{k}\right)  \dots dx\left(  t_{2}\right)  dx\left(
t_{1}\right)
\]
and
\begin{align*}
X_{t,s}  &  =1+\sum_{k=1}^{\kappa}\left(  -1\right)  ^{k}\int_{s\leq t_{1}\leq
t_{2}\leq\dots\leq t_{k}\leq t}dx\left(  t_{k}\right)  \dots dx\left(
t_{2}\right)  dx\left(  t_{1}\right) \\
&  =1+\sum_{k=1}^{\kappa}\left(  -1\right)  ^{k}\int_{s\leq t_{k}\leq
t_{k-1}\leq\dots\leq t_{1}\leq t}dx\left(  t_{1}\right)  dx\left(
t_{2}\right)  \dots dx\left(  t_{k}\right) \\
&  =1+\sum_{k=1}^{\kappa}\left(  -1\right)  ^{k}\sigma_{k}X_{s,t}^{\left(
k\right)  }.
\end{align*}

\end{proof}

\begin{corollary}
\label{cor.16.4}If $g=1+\sum_{k=1}^{\kappa}g_{k}\in G_{\text{geo}}^{\left(
\kappa\right)  }\left(  \mathbb{R}^{d}\right)  ,$ then
\begin{equation}
g^{-1}=1+\sum_{k=1}^{\kappa}\left(  -1\right)  ^{k}\sigma_{k}g_{k}.
\label{e.16.8}%
\end{equation}

\end{corollary}

\begin{proof}
By Chow's theorem, to each $g\in G_{\text{geo}}^{\left(  \kappa\right)
}\left(  \mathbb{R}^{d}\right)  ,$ there exists a path $x\left(  \cdot\right)
\in C^{1}\left(  \left[  0,1\right]  ,\mathbb{R}^{d}\right)  $ such that
$g=X_{0,1}.$ Therefore,
\[
g^{-1}=X_{1,0}=1+\sum_{k=1}^{\kappa}\left(  -1\right)  ^{k}\sigma_{k}%
X_{0,1}^{\left(  k\right)  }=1+\sum_{k=1}^{\kappa}\left(  -1\right)
^{k}\sigma_{k}g_{k}.
\]

\end{proof}

\begin{corollary}
\label{cor.16.5}If $g\in G_{\text{geo}}^{\left(  \kappa\right)  }\left(
\mathbb{R}^{d}\right)  ,$ then $\left\vert g_{k}\right\vert =\left\vert
\left[  g^{-1}\right]  _{k}\right\vert $ for $1\leq k\leq\kappa$ and in
particular $N\left(  g\right)  =N\left(  g^{-1}\right)  .$
\end{corollary}

The next result provides motivation for the constructions used in this paper.
Let $V:\mathbb{R}^{d}\rightarrow\Gamma\left(  TM\right)  $ be a dynamical
system, $x\left(  \cdot\right)  \in C^{1}\left(  \left[  0,T\right]
,\mathbb{R}^{d}\right)  ,$ and $\varphi_{t,s}\in\mathrm{Diff}\left(  M\right)
$ denote the solution to
\begin{equation}
\dot{\varphi}_{t,s}=V_{\dot{x}\left(  t\right)  }\circ\varphi_{t,s}\text{ with
}\varphi_{s,s}=Id_{M}. \label{e.16.11}%
\end{equation}
In the proof below we will make use of the simple observation that
\begin{equation}
\frac{d}{d\sigma}X_{\sigma,t}^{\kappa}=-\dot{x}\left(  \sigma\right)
X_{\sigma,t}^{\kappa-1}\text{ for any }\kappa\in\mathbb{N}. \label{e.16.12}%
\end{equation}

\begin{theorem}
\label{thm.16.6}If $f\in C^{\infty}\left(  M\right)  ,$ then for any
$s,t\in\left[  0,T\right]  $ and $\kappa\in\mathbb{N}_{0},$
\begin{align}
f\circ\varphi_{t,s}  &  =V_{X_{s,t}^{\left(  \kappa\right)  }}f-\int_{s}%
^{t}\left(  V_{\frac{d}{d\sigma}X_{\sigma,t}^{\kappa+1}}f\right)  \circ
\varphi_{\sigma,s}d\sigma\label{e.16.13}\\
&  =V_{X_{s,t}^{\left(  \kappa\right)  }}f+\int_{s}^{t}\left(  V_{\dot
{x}\left(  \sigma\right)  \otimes X_{\sigma,t}^{\kappa}}f\right)  \circ
\varphi_{\sigma,s}d\sigma. \label{e.16.14}%
\end{align}

\end{theorem}

\begin{proof}
When $\kappa=0,$ Eq. (\ref{e.16.14}) states
\[
f\circ\varphi_{t,s}=f+\int_{s}^{t}\left(  V_{\dot{x}\left(  \sigma\right)
}f\right)  \circ\varphi_{\sigma,s}d\sigma
\]
which holds by the fundamental theorem of calculus applied to the differential
identity,%
\[
\frac{d}{dt}f\circ\varphi_{t,s}=\left(  V_{\dot{x}\left(  t\right)  }f\right)
\circ\varphi_{t,s}\text{ with }f\circ\varphi_{s,s}=f.\text{ }%
\]
For the inductive step we use the fundamental theorem of calculus along with
the chain rule to conclude,
\begin{align*}
-V_{X_{s,t}^{\kappa+1}}f  &  =\left(  V_{X_{\sigma,t}^{\kappa+1}}f\right)
\circ\varphi_{\sigma,s}|_{\sigma=s}^{\sigma=t}=\int_{s}^{t}\frac{d}{d\sigma
}\left(  \left(  V_{X_{\sigma,t}^{\kappa+1}}f\right)  \circ\varphi_{\sigma
,s}\right)  d\sigma\\
&  =\int_{s}^{t}\left(  V_{\frac{d}{d\sigma}X_{\sigma,t}^{\kappa+1}}f\right)
\circ\varphi_{\sigma,s}d\sigma+\int_{s}^{t}\left(  V_{\dot{x}\left(
\sigma\right)  }V_{X_{\sigma,t}^{\kappa+1}}f\right)  \circ\varphi_{\sigma
,s}d\sigma\\
&  =\int_{s}^{t}\left(  V_{\frac{d}{d\sigma}X_{\sigma,t}^{\kappa+1}}f\right)
\circ\varphi_{\sigma,s}d\sigma-\int_{s}^{t}\left(  V_{\frac{d}{d\sigma
}X_{\sigma,t}^{\kappa+2}}f\right)  \circ\varphi_{\sigma,s}d\sigma.
\end{align*}
Putting this identity into Eq. (\ref{e.16.13}) gives,%
\begin{align*}
f\circ\varphi_{t,s}  &  =V_{X_{s,t}^{\left(  \kappa\right)  }}f+V_{X_{s,t}%
^{\kappa+1}}f-\int_{s}^{t}\left(  V_{\frac{d}{d\sigma}X_{\sigma,t}^{\kappa+2}%
}f\right)  \circ\varphi_{\sigma,s}d\sigma\\
&  =V_{X_{s,t}^{\left(  \kappa+1\right)  }}f-\int_{s}^{t}\left(  V_{\frac
{d}{d\sigma}X_{\sigma,t}^{\kappa+2}}f\right)  \circ\varphi_{\sigma,s}d\sigma
\end{align*}
which completes the inductive proof.
\end{proof}

\providecommand{\bysame}{\leavevmode\hbox to3em{\hrulefill}\thinspace}
\providecommand{\MR}{\relax\ifhmode\unskip\space\fi MR }
\providecommand{\MRhref}[2]{%
  \href{http://www.ams.org/mathscinet-getitem?mr=#1}{#2}
}
\providecommand{\href}[2]{#2}


\end{document}

%% file: defs.tex
\def \pt{/\!/}



%% file: svg-graphics.tex
\newcommand{\executeiffilenewer}[3]{%
\ifnum\pdfstrcmp{\pdffilemoddate{#1}}%
{\pdffilemoddate{#2}}>0%
{\immediate\write18{#3}}\fi%
}

\graphicspath{{\GraphicsDirectory}}

\def\svgwidth{2in}

\newcommand{\psize}[1]{\def\svgwidth{#1}}